\documentclass[11pt,a4paper]{article}

\usepackage[headinclude]{typearea}
\usepackage[english]{babel}           
\usepackage[T1]{fontenc}
\usepackage{scrtime}
\usepackage{amsmath,amsthm,amsfonts,amssymb}
\usepackage{xcolor}                                  
\usepackage {pifont}                                 
\usepackage {multicol,multirow}                             
\usepackage[numbers]{natbib}     
\usepackage[normalem]{ulem}  
\usepackage [scaled=.90]{helvet}   
\usepackage {courier}     
\usepackage {graphicx}       
\usepackage {array}  
\usepackage {longtable}   
\usepackage{fancyvrb}
\usepackage{enumerate} %
\usepackage{ifpdf}
\usepackage{caption}
\usepackage{mathrsfs}
\usepackage{setspace}
\usepackage{upgreek}
\usepackage{marvosym}
\usepackage{mathtools}
\mathtoolsset{showonlyrefs} 
\usepackage{verbatim}
\usepackage{dsfont}
\usepackage{hyperref}
\makeatletter
\def\namedlabel#1#2{\begingroup
    #2%
    \def\@currentlabel{#2}%
    \phantomsection\label{#1}\endgroup
}

\numberwithin{equation}{section}
\newtheorem{theorem}{Theorem}[section]

\newtheorem{lemma}[theorem]{Lemma}

\newtheorem{proposition}[theorem]{Proposition}

\theoremstyle{definition}
\newtheorem{assumption}{Assumption}[section]
\newtheorem{definition}[theorem]{Definition}

\newtheorem{remark}[theorem]{Remark}

\newcommand{\E}{{\mathbb{E}}}
\newcommand{\N}{{\mathbb{N}}}
\renewcommand{\P}{{\mathbb{P}}}

\newcommand{\R}{{\mathbb{R}}}
\newcommand{\F}{\mathbb{F}}

\newcommand{\cF}{{\cal F}}
\newcommand{\diff}{\mathop{}\!\mathrm{d}}
\newcommand{\D}{{\mathbb{D}}}

\title{A higher order approximation method for jump-diffusion SDEs with discontinuous drift coefficient}

\author{Pawe{\l} Przyby{\l}owicz \and Verena Schwarz \and Michaela Sz\"olgyenyi}
\date{Preprint, December 2023}

\begin{document}

\maketitle

\begin{abstract}
We present the first higher-order approximation scheme for solutions of jump-diffusion stochastic differential equations with discontinuous drift.
For this transformation-based jump-adapted quasi-Milstein scheme we prove $L^p$-convergence order 3/4. 
To obtain this result, we prove that under slightly stronger assumptions (but still weaker than anything known before) a related jump-adapted quasi-Milstein scheme has convergence order 3/4 -- in a special case even order 1.
Order 3/4 is conjectured to be optimal.
\end{abstract}

\section{Introduction}\label{sec:intro}

We consider time-homogeneous jump-diffusion stochastic differential equations (SDEs) of the form
\begin{equation}\label{eq:SDE}
\begin{aligned}
\diff X_t = \mu(X_t) \diff t + \sigma(X_t)\diff W_t + \rho(X_{t-})\diff N_t, \quad t\in[0,T],\quad X_0 =\xi,
\end{aligned}
\end{equation}
where $\xi\in\R$, $\mu,\sigma,\rho\colon\R\to \R$ are measurable functions, $T\in(0,\infty)$, $W=(W_t)_{t\geq 0}$ is a
standard Brownian motion, and $N= (N_t)_{t\geq 0}$ is a homogeneous Poisson process with intensity $\lambda \in (0,\infty)$ on the probability space $(\Omega,\cF,\P)$. We define the filtration $\mathbb{F}=(\mathbb{F}_t)_{t\geq 0}$ as the augmented natural filtration of $N$ and $W$, i.e.~for all $t\geq 0$ we set $\mathbb{F}_t = \sigma(\{(W_s,N_s) \colon s\leq t\}\cup \mathcal{N})$, where $\mathcal N = \{A\in \cF\colon \P(A)=0\}$. With this, $(\Omega,\cF,\mathbb{F},\P)$ is a filtered probability space that satisfies the usual conditions.
Existence and uniqueness of the solution of the above SDE under our assumptions is well settled by \cite{PS20,PSX21}.

This paper aims to provide a higher order scheme for SDE \eqref{eq:SDE} for the case where the drift is allowed to be discontinuous. This setting is, for example, relevant for modeling energy prices or control actions on energy markets (\cite{benth2014,SS16, shardin2017}) as well as for modeling physical phenomena, cf.~\cite{situ2005,PBL2010}.

In the case without jumps numerical approximation of SDEs with discontinuous drift has attracted a lot of attention in recent years, see, e.g., \cite{halidias2006,halidias2008,gyongy2011,LS16,ngo2016,LS17,ngo2017a,ngo2017b,pamen2017,LS18,hefter2018,muellergronbach2019,muellergronbach2019b,NSS19,Dareiotis2020,Neuenkirch2021,Butkovsky2021,Dareiotis2021,Yaroslavtseva2022,Mullergronbach2022,Hu2022}.
In the case of presence of jumps, under standard assumptions, classical and jump-adapted Itô-Taylor approximations and Runge-Kutta methods are studied, e.g., in \cite{Gardon2004,Gardon2006,PBL2010,buckwar2011}. Approximation results for jump-diffusion SDEs under non-standard assumptions can be found, e.g., in \cite{Higham2005,Higham2006,Higham2007,dareiotis2016,Deng2019a,Deng2019b,Tambue2019,Chen2020,Liu2020,Butkovsky2022,Kieu2022,Nasroallah2022}. Asymptotically optimal approximation rates are proven in \cite{Przybylowicz2016,Kaluza2018,Przybylowicz2019a,Przybylowicz2019b,Przybylowicz2022,Kaluza2022}. All above mentioned results assume continuity of the drift coefficient. In \cite{PS20} they allow for finitely many discontinuities in the drift and prove $L^2$-convergence order $1/2$ for the Euler-Maruyama approximation.

We would like to highlight the contributions that were most influential for the current paper:
In \cite{PBL2010} jump-adapted schemes are studied. Transformation-based schemes have first been introduced in \cite{LS16,LS17} for the case without jumps ($\rho\equiv 0$), where transformation-based Euler-Maruyama schemes as well as a proof technique have been introduced, which inspired many of the subsequent contributions. A transformation-based quasi-Milstein scheme for the case without jumps ($\rho\equiv 0$) has been introduced in \cite{muellergronbach2019b}; it has convergence order $3/4$ in $L^p$, which is optimal for the class of all algorithms with deterministic grid points, cf.~\cite{MY20}.
Approximation of discontinuous-drift-SDEs with jumps has been studied in \cite{PS20}.

In the current paper we construct the first higher-order scheme for jump-diffusion SDEs with discontinuous drift -- a transformation-based jump-adapted quasi-Milstein scheme. Our main contribution is to prove convergence order $3/4$ in $L^p$, $p\in(1,\infty)$ for this scheme. This is proven to be optimal in \cite{PSS23}. As a side result we prove that under slightly stronger assumptions (but still weaker than anything known in the literature) a related jump-adapted quasi-Milstein scheme has convergence order 3/4; in a special case even order 1.

For the proof we proceed in two steps. First we introduce the jump-adapted quasi-Milstein scheme and prove convergence order $3/4$ (respectively 1) under the stronger assumptions. Then we make use of a transformation and the previous result to prove our target result in Theorem \ref{MainWeakAss}.

\section{Setting}\label{sec:setting}

In the following we denote by $\tilde N=(\tilde N_t)_{t\geq 0}$ the compensated Poisson process, i.e. $\tilde N_t=N_t-\lambda t$ for all $t\in[0,\infty)$. It is well-known that $\tilde N$ is a square integrable $\mathbb{F}$-martingale. 
We denote by $L_f$ the Lipschitz constant of a generic Lipschitz continuous function $f$, and by $\operatorname{Id}$ the identity mapping.
Furthermore, for $(A_1,\mathcal{A}_1)$ and $(A_2,\mathcal{A}_2)$ two measurable spaces, $Y_1\colon(\Omega,\mathcal{F})\to (A_1,\mathcal{A}_1)$, $Y_2\colon(\Omega,\mathcal{F})\to (A_2,\mathcal{A}_2)$, $f\colon A_1\times A_2 \to \R$ measurable functions, and $\mathcal{G}$ a sub-$\sigma$-algebra of $\mathcal{F}$, we denote by $\E[f(Y_1,y)|\mathcal{G}] |_{y=Y_2}$ the random variable $g(Y_2)$, where $g\colon A_2\to\R$ is for all $y\in A_2$ defined by $g(y)= \E[f(Y_1,y)|\mathcal{G}]$. Especially, we use this notation in case that $\mathcal{G} = \{\emptyset, \Omega\}$, where we omit conditioning and write $\E[f(Y_1,y)] |_{y=Y_2}$. For random variables $Y_3,Y_4\colon\Omega \to\R$, denote by $\E[Y_3|Y_4 = y]|_{y=Y_4}$ the random variable $h(Y_4)$, where $h\colon\R\to\R$ is for all $y\in\R$ defined by $h(y) = \E[Y_3|Y_4 =y]$.

Moreover, we recall the following definition.

\begin{definition}[\text{\cite[Definition 2.1]{LS17}}]
Let $I\subseteq\R$ be an interval and let $m\in\N$. We say a function $f\colon I\to\R$
is piecewise Lipschitz, if there are finitely many points $\zeta_1<\ldots<\zeta_m\in
I$ such that $f$ is Lipschitz on each of the intervals $(-\infty,\zeta_1)\cap I,
(\zeta_m,\infty)\cap I$, and $(\zeta_k,\zeta_{k+1}),k=1,\ldots,m-1$.
\end{definition}

\begin{assumption}\label{assX}
We assume on the coefficients of SDE \eqref{eq:SDE} that there exist $m\in\N$ and $-\infty = \zeta_0 < \zeta_1 < \ldots< \zeta_m< \zeta_{m+1} =\infty$ such that:
\begin{itemize}
\item[(i)] The drift coefficient $\mu\colon\R \to \R$ is piecewise Lipschitz continuous with potential discontinuities in the points $\zeta_1,\dots,\zeta_m\in\R$.
\item[(ii)] The diffusion coefficient $\sigma\colon\R\to\R$ is Lipschitz continuous and for all $k\in\{1,\dots,m\}$, $\sigma(\zeta_k) \neq 0$.
\item[(iii)] The jump coefficient $\rho\colon\R\to\R$ is Lipschitz continuous.
\item[(iv)] The drift coefficient $\mu$ and the diffusion coefficient $\sigma$ are differentiable with Lipschitz continuous derivatives on $(\zeta_{i-1},\zeta_i)$ for all $i\in\{1,\ldots,m\}$.
\end{itemize}
\end{assumption}

The following lemma is a direct consequence of Assumption \ref{assX}.

\begin{lemma}[\text{\cite[Lemma 2]{PS20}}]
Let Assumption \ref{assX} hold. Then $\mu$, $\sigma$, and $\rho$ satisfy a linear growth condition, that is there exist constants $c_\mu, c_\sigma, c_\rho \in (0,\infty)$ such that
\begin{equation*}
|\mu(x)|\le c_\mu(1+ |x|), \qquad |\sigma(x)|\le c_\sigma(1+ |x|), \qquad |\rho(x)|\le c_\rho(1+ |x|).
\end{equation*}
\end{lemma}

The existence and uniqueness of the strong solution of the SDE \eqref{eq:SDE} is guaranteed by the following result.

\begin{theorem}[\text{\cite[Theorem 3.1]{PS20}}]
Let Assumption \ref{assX} hold.
Then the SDE \eqref{eq:SDE} has a unique global strong solution.
\end{theorem}

\section{Convergence of the jump-adapted quasi-Milstein scheme}

We are going to introduce the jump-adapted quasi-Milstein scheme and prove under stronger assumptions than Assumption \ref{assX} convergence order $3/4$, respectively order 1. This result in combination with a transformation technique will be used later to prove our main result. In this section we consider the time-homogeneous jump-diffusion SDE
\begin{align}\label{eq:TSDE}
\diff Z_t &=\widetilde \mu(Z_t)  \diff t + \widetilde \sigma(Z_t) \diff W_t+\widetilde \rho(Z_{t-})\diff N_t , \quad t\in[0,T], \quad Z_0=\widetilde \xi,
\end{align}
where $\widetilde \xi\in\R$, $\widetilde\mu,\widetilde \sigma,\widetilde\rho\colon\R\to \R$ are measurable functions, $T\in(0,\infty)$, $W=(W_t)_{t\geq 0}$, $N= (N_t)_{t\geq 0}$, and  $(\Omega,\cF,\F,\P)$ are defined as in Section \ref{sec:setting}.

\begin{assumption}\label{assZ} 
We assume on the coefficients of SDE \eqref{eq:TSDE} that
\begin{itemize}
\item[\namedlabel{Zi}{(i)}] the drift coefficient $\widetilde\mu\colon\R \to \R$,  the diffusion coefficient $\widetilde\sigma\colon\R\to\R$, and  the jump coefficient $\widetilde\rho\colon\R\to\R$ are Lipschitz continuous;
\item[\namedlabel{Zii}{(ii)}] there exists $m_{\widetilde\mu}\in\N$ and $-\infty = \zeta_0 < \zeta_1 < \ldots< \zeta_{m_{\widetilde\mu}}< \zeta_{m_{\widetilde\mu}+1} =\infty$ such that the drift coefficient $\widetilde\mu$ is differentiable with Lipschitz continuous derivatives on $(\zeta_{k-1},\zeta_{k})$ for $k\in\{1,\ldots,m_{\widetilde\mu}+1\}$;
\item[\namedlabel{Ziii}{(iii)}] there exists $m_{\widetilde\sigma}\in\N$ and $-\infty = \eta_0 < \eta_1 < \ldots< \eta_{m_{\widetilde\sigma}}< \eta_{m_{\widetilde\sigma}+1} =\infty$ such that the diffusion coefficient $\widetilde\sigma$ is differentiable with Lipschitz continuous derivatives on $(\eta_{k-1},\eta_{k})$ for $k\in\{1,\ldots,m_{\widetilde\sigma}+1\}$;
\item[\namedlabel{Zv}{(iv)}] for all $k\in\{1,\dots,m_{\widetilde\mu}\}$, $\widetilde\sigma(\zeta_k) \neq 0$ and for all $j\in\{1,\ldots,m_{\widetilde\sigma}\}$, $\widetilde\sigma(\eta_j) \neq 0$.
\end{itemize}
\end{assumption}

Under Assumption \ref{assZ} the existence of a strong unique solution is ensured by \cite[p.~255,~Theorem~6]{protter2005}. Further we define the functions $d_f\colon\R\to\R$ for $f\in\{\widetilde\mu,\widetilde\sigma\}$ by
\begin{equation}
d_f(x) = \left\{
\begin{array}{ll}
f'(x) & f \text{ differentiable at } x \\
0 & \, \textrm{else.} \\
\end{array}
\right. 
\end{equation}

In the following lemma we state some direct consequences of Assumption \ref{assZ}.

\begin{lemma}\label{ConAss}
Let Assumption \ref{assZ} hold and let $f\in\{\widetilde\mu,\widetilde\sigma,\widetilde\rho\}$, $g\in\{\widetilde\mu,\widetilde\sigma\}$. 
\begin{itemize}
\item[\namedlabel{Ci}{(i)}] Then there exists a constant $c_f \in (0,\infty)$ such that for all $x\in\R$, 
\begin{equation}
|f(x)|\leq c_{f}(1+|x|).
\end{equation}
\item[\namedlabel{Cii}{(ii)}] It holds that 
\begin{equation}
\|d_{g}\|_{\infty}\leq L_{g} <\infty.
\end{equation}
\item[\namedlabel{Ciii}{(iii)}] There exists a constant $b_{g} \in(0,\infty)$ such that for all $x,y\in\R$ for which $g$ is differentiable on $[x,y]$  it holds that 
\begin{equation}
\begin{aligned}
|g(y)-g(x)-g'(x)(y-x)|\leq b_{g} |y-x|^2.
\end{aligned}
\end{equation}
\end{itemize}
\end{lemma}

Now we define our jump-adapted time discretisation. For this, let $\delta = \frac{T}{M}$ for all $M\in\N$ and define the equidistant time discretisation $(s_m)_{m\in\{0,\ldots,M\}}$ by $s_m = \delta m$ for $m\in\{0,\ldots,M\}$. We add to these deterministic points the family $(\nu_i)_{i\in\N}$ of all jump times of the Poisson process. Then we order all points so that they are increasing, yielding the time discretisation $(\tau_n)_{n\in\N}$ with the maximum step size less or equal than $\delta$.
Formally, we define $(\tau_n)_{n\in\N}$ as follows:
\begin{align}
\label{EqTau0}
\tau_0 &= 0
\\\label{EqTauGen}
\tau_{n+1} &= \min\{\nu_i: i\in\N, \nu_i > \tau_{n}\} \wedge \min\{s_m: m\in\{0,\ldots,M\}, s_m > \tau_{n}\}\wedge T.
\end{align}
Given a time grid $(\tau_n)_{n\in\N}$ we define the time-continuous jump-adapted quasi-Milstein scheme $(Z^{(M)}_t)_{t\in[0,T]}$ by
\begin{equation}\label{eqDefMS0}
\begin{aligned}
Z^{(M)}_0 =\widetilde\xi,
\end{aligned}
\end{equation}
and for all $n\in\N_0$ by 
\begin{equation}
\begin{aligned}\label{EqDefMSLL}
Z^{(M)}_{\tau_{n+1}-} =& Z^{(M)}_{\tau_{n}} + \widetilde \mu\big(Z^{(M)}_{\tau_{n}}\big)  \big(\tau_{n+1}- \tau_{n}\big) + \widetilde\sigma \big(Z^{(M)}_{\tau_{n}}\big)\big(W_{\tau_{n+1}} -W_{\tau_{n}}\big)  \\
&+ \frac{1}{2} \widetilde\sigma \big(Z^{(M)}_{\tau_{n}}\big) d_{\widetilde\sigma} \big(Z^{(M)}_{\tau_{n}}\big)  \big(\big(W_{\tau_{n+1}} -W_{\tau_{n}}\big)^2 -\big(\tau_{n+1}-\tau_{n}\big)\big)
\end{aligned}
\end{equation}
and 
\begin{equation}
\begin{aligned}\label{EqDefMS}
Z^{(M)}_{\tau_{n+1}} = Z^{(M)}_{\tau_{n+1}-} + \widetilde \rho\big(Z^{(M)}_{\tau_{n+1}-}\big) \big(N_{\tau_{n+1}}- N_{\tau_{n}}\big),
\end{aligned}
\end{equation}
or equivalently 
\begin{equation}
\begin{aligned}\label{EqDefMSAlt}
Z^{(M)}_{\tau_{n+1}} = Z^{(M)}_{\tau_{n+1}-} + \widetilde \rho\big(Z^{(M)}_{\tau_{n+1}-}\big) \big(N_{\tau_{n+1}}- N_{\tau_{n+1}-}\big).
\end{aligned}
\end{equation}
Between the points of the time discretisation, i.e.~for $t\in\big(\tau_{n},\tau_{n+1}\big)$, $n\in\N_0$, we set
\begin{equation}
\begin{aligned}
Z^{(M)}_{t} =& Z^{(M)}_{\tau_{n}} + \widetilde \mu\big(Z^{(M)}_{\tau_{n}}\big) \big(t- \tau_{n}\big) + \widetilde\sigma \big(Z^{(M)}_{\tau_{n}}\big) \big(W_{t} -W_{\tau_{n}}\big)  \\
&+ \widetilde\sigma \big(Z^{(M)}_{\tau_{n}}\big) d_{\widetilde\sigma} \big(Z^{(M)}_{\tau_{n}}\big) \frac{1}{2} \big(\big(W_{t} -W_{\tau_{n}}\big)^2 -\big(t-\tau_{n}\big)\big).
\end{aligned}
\end{equation}
Note that the following integral notation is equivalent to \eqref{eqDefMS0}, \eqref{EqDefMSLL}, \eqref{EqDefMS}:
\begin{equation}\label{SumAppr}
\begin{aligned}
Z^{(M)}_{t} & =  \widetilde\xi
+ \int_0^t \sum_{n=0}^{\infty} \widetilde \mu\big(Z^{(M)}_{\tau_{n}}\big)\mathds{1}_{(\tau_n,\tau_{n+1}]}(u) \diff u\\
&\quad+ \int_0^t \sum_{n=0}^{\infty} \Big(\widetilde \sigma\big(Z^{(M)}_{\tau_{n}}\big) + \widetilde\sigma \big(Z^{(M)}_{\tau_n}\big) d_{\widetilde\sigma} \big(Z^{(M)}_{\tau_n}\big) (W_{u}-W_{\tau_n}) \Big)\mathds{1}_{(\tau_n,\tau_{n+1}]}(u) \diff W_u \\
&\quad+ \int_0^t \widetilde \rho\big(Z^{(M)}_{u-}\big)\diff N_u.
\end{aligned}
\end{equation}
Similarly, we can express the solution on SDE \eqref{eq:TSDE} by
\begin{equation}
\begin{aligned}
Z_{t} = &\, \widetilde\xi
+ \int_0^t \sum_{n=0}^{\infty} \widetilde \mu\big(Z_{u}\big)\mathds{1}_{(\tau_n,\tau_{n+1}]}(u) \diff u
+ \int_0^t \sum_{n=0}^{\infty} \widetilde \sigma\big(Z_{u}\big)\mathds{1}_{(\tau_n,\tau_{n+1}]}(u) \diff W_u
+ \int_0^t \widetilde \rho\big(Z_{u-}\big)\diff N_u.
\end{aligned}
\end{equation}

To simplify the notation we denote for given $M\in\N$ and time point $t\in[0,T]$, the largest value of the sequence $(s_m)_{m=0,\ldots,M}$ which is smaller or equal $t$, by $\underline{t}$,
\begin{equation}
    \begin{aligned}
    &\underline{t} = \max\{s_m: m\in\{0,\ldots,M\}, s_m \leq t\} =  \Big\lfloor\frac{tM}{T}\Big\rfloor \frac{T}{M}.
    \end{aligned}
\end{equation}

\subsection{Preparatory lemmas}

In this section we provide some basic properties of the time discretisation, moment estimates, and Markov properties.

We start providing some basic properties of the time discretisation $(\tau_n)_{n\in\N}$. We consider two filtrations, which we both need later on. In addition to $\F$, we define the filtration $\widetilde\F=(\widetilde\F_t)_{t\geq 0}$ for all $t\in[0,\infty)$ by 
\begin{equation}
\begin{aligned}
\widetilde\F_t = \sigma(\{W_s:s\leq t\}\cup \{N_s: s\geq 0\} \cup \mathcal N).
\end{aligned}
\end{equation}
Recall that $\mathcal N$ is the set of all nullsets of $\cF$.
Because $W$ and $N$ are independent $W$ is also a Brownian motion with respect to the filtration $(\widetilde\F_t)_{t\geq 0}$. Further it is obvious that $Z^{(M)}$ as well as $Z$ are $\widetilde \F$-adapted càdlàg processes, since for all $t\in[0,\infty)$, $\F_t\subset\widetilde\F_t$.

\begin{lemma}\label{PropDiscGrid}
Let $M\in\N$  and let the sequence $(\tau_n)_{n\in\N}$ be defined by \eqref{EqTau0} and \eqref{EqTauGen}. Then for all $n\in\N$,
\begin{itemize}
    \item[\namedlabel{STi}{(i)}] $\tau_n$ is an $\F$-stopping time and also an $\widetilde\F$-stopping time;
    \item[\namedlabel{STii}{(ii)}] $Z^{(M)}_{\tau_n}$ is $\F_{\tau_n}$-measurable and also $\widetilde\F_{\tau_n}$-measurable;
    \item[\namedlabel{STiii}{(iii)}] $(W_{\tau_n+s}- W_{\tau_n})_{s\geq 0}$ is a Brownian motion with respect to the filtration $(\F_{\tau_n+s})_{s\geq0}$,  $(N_{\tau_n+s}- N_{\tau_n})_{s\geq 0}$ is a Poisson process with respect to the filtration $(\F_{\tau_n+s})_{s\geq0}$, both processes are independent of $\F_{\tau_n}$, and it holds that $(W_{\tau_n+s}- W_{\tau_n})_{s\geq 0}$ is independent of $(N_{\tau_n+s}- N_{\tau_n})_{s\geq 0}$;
    \item[\namedlabel{STiv}{(iv)}] $(W_{\tau_n+s}- W_{\tau_n})_{s\geq 0}$ is a Brownian motion with respect to the filtration $(\widetilde \F_{\tau_n+s})_{s\geq0}$, which is independent of $\widetilde \F_{\tau_n}$.
\end{itemize}
\end{lemma}

\begin{proof}
For \ref{STi} we fix $t\in[0,\infty)$ and see that
\begin{equation}
\begin{aligned}\label{ST1}
\{\tau_n \leq t\} = \Big\{ N_t + \Big\lfloor \frac{tM}{T} \Big\rfloor \geq n\Big\} = \Big\{ N_t \geq n -  \Big\lfloor \frac{tM}{T} \Big\rfloor\Big\}\in\mathbb{F}_0.
\end{aligned}
\end{equation}
Hence $\tau_n$ is an $\F$-stopping time and an $\widetilde\F$-stopping time.
By \cite[p.~5, Theorem 6]{protter2005}, item \ref{STii} follows directly from item \ref{STi}. Item \ref{STiii} is implied by \cite[Theorem 40.10]{Sato2013} by considering the two-dimensional Lévy process $(N,W)$. Item \ref{STiv} is proven in \cite[Theorem 11.11]{Kallenberg1997}.
\end{proof}

Next we provide moment bounds for the jump-adapted quasi-Milstein scheme.

\begin{lemma}\label{LMEW}
Let $p\in\N$. Then there exists a constant $c_{W_p}\in(0,\infty)$ such that it holds for all stopping times $\tau$, all $M\in\N$, $\delta =\frac{T}{M}$, and all $t\in[0,\infty)$ that 
\begin{equation}\label{MEW}
\begin{aligned}
&\E\big[|W_{\tau+t} - W_\tau|^p\big] \leq c_{W_p} \, t^\frac{p}{2},
\end{aligned}
\end{equation}
and 
\begin{equation}\label{MEW2}
\begin{aligned}
&\E\Big[\sup_{s\in[0,\delta]} |W_{\tau+s} - W_\tau|^p\Big] \leq c_{W_p} \, \delta^\frac{p}{2}.
\end{aligned}
\end{equation}
\end{lemma}

These well known estimates can for example be obtained by combining \cite[p.~299]{protter2005} and \cite[p.~53, respectively p.~59, Proposition 1.88]{Pardoux2014}.

\begin{lemma}\label{BDGaKunita}
Assume that $q\in [2,\infty)$, $a,b\in[0,T]$ with $a<b$, $Z\in\{\operatorname{Id},W,N\}$, and that $Y = (Y(t))_{t\in[a,b]}$ is a predictable stochastic process with respect to $(\F_t)_{t\in[a,b]}$ and with
\begin{equation}
\begin{aligned}
\E\Big[\int_a^b |Y(t)|^q \diff t\Big] < \infty. 
\end{aligned}
\end{equation}
Then there exists a constant $\hat c\in(0,\infty)$ such that for all $t\in[a,b]$,
\begin{equation}
\begin{aligned}
\E\Big[ \sup_{s\in[a,t]} \Big|\int_a^s Y(u) \diff Z(u)\Big|^q \Big] \leq \hat c \int_a^t \E[|Y(u)|^q]\diff u. 
\end{aligned}
\end{equation}
\end{lemma}

\begin{proof}
In the case that $Z=\operatorname{Id}$ the claim is proven directly using Jensen's inequality. 
In the case that $Z=W$, we know that $(\int_a^t Y(s)\diff W(s))_{t\in[a,b]}$ is a martingale. Hence we apply the Burkholder-Davis-Gundy inequality and afterwards Jensen's inequality to obtain that there exists some $c\in(0,\infty)$ such that 
\begin{equation}
\begin{aligned}
&\E\Big[ \sup_{s\in[{a},t]} \Big|\int_{a}^s Y(u) \diff W(u)\Big|^q \Big] 
\leq c\, \E\Big[ \Big(\int_{{a}}^t (Y(u))^2 \diff u\Big)^{\frac{q}{2}} \Big]
\leq  c\, ({b-a})^{\frac{q}{2}-1} \int_{{a}}^t \E[|Y(u)|^q]\diff u.
\end{aligned}
\end{equation}
For the case that $Z=N$ the claim is proven by applying \cite[Lemma 2.1]{maghsoodi1996}.
\end{proof}

We have not yet proven finiteness of moments of the approximation scheme, but this is not required for the following lemma; the lemma is proven in Appendix \ref{appendix}.

\begin{lemma}\label{MW}
    For all $p\in\N_0$, $q\in\N$, $u\in[0,T]$, $M\in \N$, $\delta = \frac{T}{M}$, $n\in\N$, $k\in\N$, and the constants $c_{W_q}$ as in Lemma \ref{MEW} it holds that
    \begin{equation}\label{MWeq1}
        \begin{aligned}
        &\E\Big[\big(1+ \big|Z^{(M)}_{\tau_n}\big|^p\big)   |W_u-W_{\tau_n}|^q \mathds{1}_{(\tau_n,\tau_{n+1}]}(u)  \mathds{1}_{\{N_u = k\}} \Big]\\
        &\leq c_{W_q} \delta^{\frac{q}{2}} \E\Big[ \big(1+ \big|Z^{(M)}_{\tau_n}\big|^p\big) \mathds{1}_{(\tau_n,\tau_{n+1}]}(u)  \mathds{1}_{\{N_u = k\}} \Big],
        \end{aligned}
    \end{equation}
    and
    \begin{equation}\label{MWeq2}
        \begin{aligned}
        &\E\Big[\sum_{n=0}^\infty \big(1+ \big|Z^{(M)}_{\tau_n}\big|^p\big)   |W_u-W_{\tau_n}|^q \mathds{1}_{(\tau_n,\tau_{n+1}]}(u) \Big]
        \leq c_{W_q} \delta^{\frac{q}{2}} \E\Big[ \sum_{n=0}^\infty\big(1+ \big|Z^{(M)}_{\tau_n}\big|^p\big) \mathds{1}_{(\tau_n,\tau_{n+1}]}(u) \Big].
        \end{aligned}
    \end{equation}
\end{lemma}

In the next lemma we provide moment estimates for the jump-adapted quasi-Milstein scheme, which depend on the fineness of the discretisation scheme. We use these afterwards to prove moment estimates with constants independent of $M$. The proof can be found in Appendix \ref{appendix}.

\begin{lemma}\label{FiniteMom}
Let Assumption \ref{assZ} hold and let $p\in\N$, $p\geq2$. Then for all $M\in\N$, $\delta = \frac{T}{M}$, and all $\widetilde\xi\in\R$ there exists a constant $c_M \in(0,\infty)$ such that 
\begin{equation}\label{LFM1}
\begin{aligned}
\E\Big[ \sum_{n=0}^{N_T + M} \big(1+ \big|Z^{(M)}_{\tau_{n}}\big|^p\big) \Big] 
+\E\Big[ N_T^{p-1} \sum_{n=0}^{N_T + M} \big(1+ \big|Z^{(M)}_{\tau_{n}-}\big|^p\big) \Big] \leq c_M.
\end{aligned}
\end{equation}
\end{lemma}

Next we provide several moment estimates for the jump-adapted quasi-Milstein scheme. Recall that $\widetilde \xi$ denotes the initial value of $Z^{(M)}$. Also this lemma is proven in Appendix \ref{appendix}.

\begin{lemma}\label{ME}
Let Assumption \ref{assZ} hold and let $p\in\N$, $p\geq2$. Then there exist constants $c_1, c_2, c_3 \in(0,\infty)$ such that for all $M\in\N$, $\delta=\frac{T}{M}$, $\widetilde\xi\in\R$, $s\in[0,T]$, and all $t\in[0,T]$ with $t \leq T-\delta$ it holds that 
\begin{equation}\label{LME1}
\begin{aligned}
\E\Big[\sup_{t\in[0,T]} \big|Z^{(M)}_t\big|^p\Big] \leq c_1 \big(1+ |\widetilde\xi|^p\big).
\end{aligned}
\end{equation}
\begin{equation}\label{LME2}
\begin{aligned}
\E\Big[ \sum_{n=0}^{\infty} \big|Z^{(M)}_s- Z^{(M)}_{\tau_{n}}\big|^p\mathds{1}_{[\tau_n,\tau_n+1)}(s)\Big] \leq c_2 \big(1+ |\widetilde\xi|^p\big)\delta^{\frac{p}{2}}.
\end{aligned}
\end{equation}
\begin{equation}\label{LME4}
    \begin{aligned}
    \E\Big[\sup_{s\in[t,t+\delta]} \big|Z^{(M)}_s- Z^{(M)}_{t}\big|^p\Big] \leq c_3 \big(1+ |\widetilde\xi|^p\big)\delta.
    \end{aligned}
\end{equation}
\end{lemma}

Let $Z^{(M),\widetilde\xi}$ be the solution to SDE \eqref{SumAppr} with initial value $\widetilde\xi$.

\begin{lemma}\label{MarkovProperty}
For all $\widetilde \xi\in\R$, all $M\in\N$, and all $m\in\{0,\ldots,M-1\}$, 
\begin{equation}\label{MP1}
    \begin{aligned}
    \P^{\big(Z^{(M),\widetilde\xi}_{s_m+t}\big)_{t\in[0,T-s_m]}\big| Z^{(M),\widetilde\xi}_{s_m}} = \P^{\big(Z^{(M),\widetilde\xi}_{t_m+t}\big)_{t\in[0,T-s_m]}\big| \F_{s_m}}.
    \end{aligned}
\end{equation}
Further it holds for all $\widetilde \xi\in\R$, $M\in\N$, $m\in\{0,\ldots,M-1\}$, and  $\P^{Z^{(M),\widetilde\xi}_{s_m}}$-almost all $y\in\R$ that
\begin{equation}\label{MP2}
    \begin{aligned}
    \P^{\big(Z^{(M),\widetilde\xi}_{s_m+t}\big)_{t\in[0,T-s_m]}\big| Z^{(M),\widetilde\xi}_{s_m} = y} = \P^{\big(Z^{(M),y}_{t}\big)_{t\in[0,T-s_m]}}.
    \end{aligned}
\end{equation}
\end{lemma}

We denote by $\D_k = \{f\colon[0,T]\to \R^k : f \text{ is càdlàg}\}$ for all $k\in\N$, $T\in(0,\infty)$. Further, for a càdlàg stochastic process $Y$ on $[0,T]$ and a $\sigma$-algebra $\mathcal G \subset \mathcal F$ we denote by $\P^{Y | \mathcal{G}}$ the measure defined for all $A\in\D_1$ by
\[\P^{Y|\mathcal{G}}(A) =\E[\mathds{1}_{A}(Y) | \mathcal{G}].\]
Additionally, for a càdlàg stochastic process $Y$, a random variable $Y_1$, and a real value $y_1\in\R$ we define by $\P^{Y |Y_1=y_1}$ the measure defined for all $A\in\mathcal{B}(\D_1)$ by
\[\P^{Y|Y_1=y_1}(A) = \E[\mathds{1}_{A}(Y) |Y_1=y_1].\]

\begin{proof}
For fixed $m\in\{0,\ldots,M\}$ consider SDE \eqref{SumAppr} on the time interval $[0,T-s_m]$. This SDE satisfies \cite[Assumption 2.1.]{PSS22}.  Also the process $(Z^{(M),\widetilde\xi}_{s_m+t})_{t\in[0,T-s_m]}$ satisfies \cite[Assumption 2.1.]{PSS22} with the same function in the integrand. Hence by \cite[Theorem 3.1]{PSS22} there exists a Skorohod measurable function $\hat \Psi\colon \R\times  \D_3 \to \D_1$ such that 
\begin{equation}
    \begin{aligned}
    (Z^{(M),\widetilde\xi}_t)_{t\in[0,T-s_m]} = \hat\Psi(\widetilde\xi,(t,W_t,N_t)_{t\in[0,T-s_m]})
    \end{aligned}
\end{equation}
and
\begin{equation}
    \begin{aligned}
    (Z^{(M),\widetilde\xi}_{s_m+t})_{t\in[0,T-s_m]} = &\hat\Psi(Z^{(M),\widetilde\xi}_{s_m},(t+s_m -s_m,W_{t+s_m} - W_{s_m},N_{t+s_m} - N_{s_m})_{t\in[0,T-s_m]}).
    \end{aligned}
\end{equation}
Using these equalities the Markov property follows by direct computations.\footnote{Another way of proving this result is by using \cite[Lemma 9.2]{Mao1997} for the first statement and proceeding similar to \cite[proof of Theorem 9.5]{Mao1997} for the second statement.}
\end{proof}

\subsection{Occupation time estimates}

In this section we provide occupation time estimates for our approximation scheme. For this we compose ideas of \cite{muellergronbach2019b} and \cite{PS20}. The novelty here lies in the additional complexity caused by the adapted scheme.

\begin{lemma}\label{LOTE}
Let Assumption \ref{assZ} hold. Let $\zeta\in\R$ with $\widetilde\sigma (\zeta) \neq 0$. Then there exists a constant $c_5\in(0,\infty)$ such that for all $\widetilde\xi\in\R$, all $M\in\N$, $\delta = \frac{T}{M}$, and all $\varepsilon \in (0,\infty)$ we have
\begin{equation}
    \begin{aligned}
    \int_0^T \P\Big(\Big|Z_t^{(M),\widetilde\xi} - \zeta\Big| \leq \varepsilon\Big) \diff t \leq c_5 (1 +\widetilde\xi^2) (\varepsilon + \delta^{\frac{1}{2}}).
    \end{aligned}
\end{equation}
\end{lemma}

\begin{proof}
Recall that we denote by $Z_t^{(M),\widetilde\xi}$ the solution of SDE \eqref{SumAppr} with initial value $\widetilde\xi\in\R$. It holds by equation \eqref{SumAppr} that
\begin{equation}\label{pOTE1}
    \begin{aligned}
    Z_t^{(M),\widetilde\xi} &= \widetilde\xi + \int_0^t \sum_{n=0}^{\infty} \widetilde \mu \big(Z^{(M),\widetilde\xi}_{\tau_n}\big)\mathds{1}_{(\tau_n,\tau_{n+1}]}(u)\diff u \\
    &\quad+ \int_0^t \sum_{n=0}^{\infty} \Big( \widetilde\sigma \big(Z^{(M),\widetilde\xi}_{\tau_n}\big) + \int_{\tau_n}^u \widetilde\sigma \big(Z^{(M),\widetilde\xi}_{\tau_n}\big)d_{\widetilde\sigma} \big(Z^{(M),\widetilde\xi}_{\tau_n}\big)\diff W_s \Big)\mathds{1}_{(\tau_n,\tau_{n+1}]}(u)\diff W_u \\
    &\quad+ \int_0^t \widetilde\rho\big(Z^{(M),\widetilde\xi}_{u-}\big)\diff N_u.
    \end{aligned}
\end{equation}
By \cite[Lemma 158]{situ2005} we obtain for all $a\in\R$ that
\begin{equation}\label{pOTE8}
    \begin{aligned}
    &\big| Z^{(M),\widetilde\xi}_t - a\big|\\
    &= \big| \widetilde\xi -a\big|
    +\int_0^t \operatorname{sgn}\big(Z^{(M),\widetilde\xi}_{u-} -a\big) \diff Z^{(M),\widetilde\xi}_{u}
    +L^a_t \big(Z^{(M),\widetilde\xi}\big) \\
    &\quad+ \int_0^t \Big( \big| Z^{(M),\widetilde\xi}_{u-} + \widetilde\rho\big(Z^{(M),\widetilde\xi}_{u-}\big)- a\big| - \big| Z^{(M),\widetilde\xi}_{u-} - a\big| - \operatorname{sgn}\big(Z^{(M),\widetilde\xi}_{u-} -a\big) \widetilde\rho\big(Z^{(M),\widetilde\xi}_{u-}\big)\Big) \diff N_u,
    \end{aligned}
\end{equation}
where $\operatorname{sgn}(x) = \mathds{1}_{(0,\infty)}(x) - \mathds{1}_{(-\infty,0]}(x)$ for all $x\in\R$ and $L^a_t \big(Z^{(M),\widetilde\xi}\big)$ is the local time of $Z^{(M),\widetilde\xi}$ in $a$. By \eqref{pOTE8} we obtain
\begin{equation}\label{pOTE9}
    \begin{aligned}
    &L^a_t \big(Z^{(M),\widetilde\xi}\big) = \big|L^a_t \big(Z^{(M),\widetilde\xi}\big)\big|\\
    &\leq \Big|\big| Z^{(M),\widetilde\xi}_t - a\big| - \big| \widetilde\xi -a\big|\Big| 
    + \Big|\int_0^t \operatorname{sgn}\big(Z^{(M),\widetilde\xi}_{u-} -a\big) \diff Z^{(M),\widetilde\xi}_{u}\Big|\\
    &\quad + \Big|\int_0^t \Big( \big| Z^{(M),\widetilde\xi}_{u-} + \widetilde\rho\big(Z^{(M),\widetilde\xi}_{u-}\big)- a\big| - \big| Z^{(M),\widetilde\xi}_{u-} - a\big| - \operatorname{sgn}\big(Z^{(M),\widetilde\xi}_{u-} -a\big) \widetilde\rho\big(Z^{(M),\widetilde\xi}_{u-}\big)\Big) \diff N_u\Big|.
    \end{aligned}
\end{equation}
Now we estimate the expectation of each summand on the right hand side of \eqref{pOTE9} separately. For the first one we use Lemma \ref{ME} \eqref{LME1} to obtain that there exists $c_1\in(0,\infty)$ such that
\begin{equation}\label{pOTE10}
    \begin{aligned}
    &\E\Big[\Big|\big| Z^{(M),\widetilde\xi}_t - a\big| - \big| \widetilde\xi -a\big|\Big|\Big]
    \leq \big(1+c_1^\frac{1}{2}\big)\big(1+ \big| \widetilde\xi\big| \big).
    \end{aligned}
\end{equation}
Using \eqref{pOTE1} we get for the second summand of \eqref{pOTE9} that
\begin{equation}\label{pOTE12}
    \begin{aligned}
    &\Big|\int_0^t \operatorname{sgn}\big(Z^{(M),\widetilde\xi}_{u-} -a\big) \diff Z^{(M),\widetilde\xi}_{u}\Big|\\
    &\leq \Big|\int_0^t \operatorname{sgn}\big(Z^{(M),\widetilde\xi}_{u-} -a\big)\Big(\sum_{n=0}^{\infty} \widetilde \mu \big(Z^{(M),\widetilde\xi}_{\tau_n}\big)\mathds{1}_{(\tau_n,\tau_{n+1}]}(u)\Big) \diff u\Big|\\
    &\quad+\Big|\int_0^t \operatorname{sgn}\big(Z^{(M),\widetilde\xi}_{u-} -a\big)\Big(\sum_{n=0}^{\infty} \Big( \widetilde\sigma \big(Z^{(M),\widetilde\xi}_{\tau_n}\big)\\
    &\quad\quad\quad\quad \quad\quad\quad\quad \quad\quad\quad\quad \quad\quad\quad + \int_{\tau_n}^u \widetilde\sigma \big(Z^{(M),\widetilde\xi}_{\tau_n}\big)d_{\widetilde\sigma} \big(Z^{(M),\widetilde\xi}_{\tau_n}\big)\diff W_s \Big)\mathds{1}_{(\tau_n,\tau_{n+1}]}(u)\Big) \diff W_u\Big|\\
    &\quad+\Big|\int_0^t \operatorname{sgn}\big(Z^{(M),\widetilde\xi}_{u-} -a\big)\widetilde\rho\big(Z^{(M),\widetilde\xi}_{u-}\big) \diff N_u\Big|.
    \end{aligned}
\end{equation}
We estimate the expectation of each summand of \eqref{pOTE12} separately. Lemma \ref{ME} \eqref{LME1} ensures
\begin{equation}\label{pOTE13}
    \begin{aligned}
    &\E\Big[\Big|\int_0^t \operatorname{sgn}\big(Z^{(M),\widetilde\xi}_{u-} -a\big)\Big(\sum_{n=0}^{\infty} \widetilde \mu \big(Z^{(M),\widetilde\xi}_{\tau_n}\big)\mathds{1}_{(\tau_n,\tau_{n+1}]}(u)\Big) \diff u\Big|\Big]\\
    &\leq c_{\widetilde\mu}\, \E\Big[\int_0^t \Big(\sum_{n=0}^{\infty} \big(1+ \big|Z^{(M),\widetilde\xi}_{\tau_n}\big|\big)\mathds{1}_{(\tau_n,\tau_{n+1}]}(u)\Big)\diff u\Big]\\
    &\leq c_{\widetilde\mu}\, \int_0^t\E\Big[\big(1+ \sup_{v\in[0,T]}\big|Z^{(M),\widetilde\xi}_{v}\big|\big)\Big]\diff u
    \leq c_{\widetilde\mu} T (1+ c_1^{\frac{1}{2}})(1+|\widetilde\xi|).
    \end{aligned}
\end{equation}
Next we estimate the second summand of \eqref{pOTE12} using the Cauchy-Schwarz inequality, and the Itô isometry, Lemma \ref{MW} \eqref{MWeq2}, and Lemma \ref{ME} \eqref{LME1}:
\begin{equation}\label{pOTE15}
    \begin{aligned}
    &\E\Big[\Big|\int_0^t \operatorname{sgn}\big(Z^{(M),\widetilde\xi}_{u-} -a\big)\Big(\sum_{n=0}^{\infty} \Big( \widetilde\sigma \big(Z^{(M),\widetilde\xi}_{\tau_n}\big) + \int_{\tau_n}^u \widetilde\sigma \big(Z^{(M),\widetilde\xi}_{\tau_n}\big)d_{\widetilde\sigma} \big(Z^{(M),\widetilde\xi}_{\tau_n}\big)\diff W_s \Big)\mathds{1}_{(\tau_n,\tau_{n+1}]}(u)\Big) \diff W_u\Big|\Big]\\
    &\leq \E\Big[\Big|\int_0^t \operatorname{sgn}\big(Z^{(M),\widetilde\xi}_{u-} -a\big)\Big(\sum_{n=0}^{\infty} \Big( \widetilde\sigma \big(Z^{(M),\widetilde\xi}_{\tau_n}\big) \\
    &\quad\quad\quad\quad \quad\quad\quad\quad \quad\quad\quad\quad \quad+ \int_{\tau_n}^u \widetilde\sigma \big(Z^{(M),\widetilde\xi}_{\tau_n}\big)d_{\widetilde\sigma} \big(Z^{(M),\widetilde\xi}_{\tau_n}\big)\diff W_s \Big)\mathds{1}_{(\tau_n,\tau_{n+1}]}(u)\Big) \diff W_u\Big|^2\Big]^{\frac{1}{2}}\\
    &= \E\Big[\int_0^t \sum_{n=0}^{\infty} \Big|\widetilde\sigma \big(Z^{(M),\widetilde\xi}_{\tau_n}\big) + \int_{\tau_n}^u \widetilde\sigma \big(Z^{(M),\widetilde\xi}_{\tau_n}\big)d_{\widetilde\sigma} \big(Z^{(M),\widetilde\xi}_{\tau_n}\big)\diff W_s \Big|^2 \mathds{1}_{(\tau_n,\tau_{n+1}]}(u)\diff u\Big]^{\frac{1}{2}}\\
    &\leq \Big( 4 c_{\widetilde\sigma}^2 \int_0^t \E\Big[ \sum_{n=0}^{\infty} \big(1 + \big|Z^{(M),\widetilde\xi}_{\tau_n}\big|^2 \big)\mathds{1}_{(\tau_n,\tau_{n+1}]}(u)\Big]\diff u\\
    &\quad\quad + 4 c_{\widetilde\sigma}^2 L_{\widetilde\sigma}^2 \int_0^t \E\Big[ \sum_{n=0}^{\infty}
    \big(1 + \big|Z^{(M),\widetilde\xi}_{\tau_n}\big|^2 \big) (W_u -W_{\tau_n})^2\mathds{1}_{(\tau_n,\tau_{n+1}]}(u)\Big]\diff u\Big) ^{\frac{1}{2}}\\
    &\leq \Big( 4 c_{\widetilde\sigma}^2 \int_0^t \E\Big[\big(1+ \sup_{t\in[0,T]}\big|Z^{(M),\widetilde\xi}_{t}\big|^2 \big)\Big]\diff u\\
    &\quad\quad + 4 c_{\widetilde\sigma}^2 L_{\widetilde\sigma}^2 c_{W_2} \delta \int_0^t \E\Big[ \sum_{n=0}^{\infty}
    \big(1 + \big|Z^{(M),\widetilde\xi}_{\tau_n}\big|^2 \big)\mathds{1}_{(\tau_n,\tau_{n+1}]}(u)\Big]\diff u\Big) ^{\frac{1}{2}}\\
    &\leq \Big( 4 c_{\widetilde\sigma}^2 T + 4 c_{\widetilde\sigma}^2 L_{\widetilde\sigma}^2 c_{W_2} T^2\Big)^\frac{1}{2} (1+c_1)^\frac{1}{2} (1+|\widetilde\xi|).
    \end{aligned}
\end{equation}
For the third summand of \eqref{pOTE12} we use Cauchy-Schwarz inequality, Lemma \ref{BDGaKunita}, and Lemma \ref{ME} \eqref{LME1}. This shows that there exists $\hat c\in(0,\infty)$ such that
\begin{equation}\label{pOTE16}
    \begin{aligned}
    &\E\Big[\Big|\int_0^t \operatorname{sgn}\big(Z^{(M),\widetilde\xi}_{u-} -a\big)\widetilde\rho\big(Z^{(M),\widetilde\xi}_{u-}\big) \diff N_u\Big|\Big]
    \leq\E\Big[\Big|\int_0^t \operatorname{sgn}\big(Z^{(M),\widetilde\xi}_{u-} -a\big)\widetilde\rho\big(Z^{(M),\widetilde\xi}_{u-}\big) \diff N_u\Big|^2\Big]^\frac{1}{2}\\
    &\leq \Big(\hat c\,\E\Big[\int_0^t\big| \operatorname{sgn}\big(Z^{(M),\widetilde\xi}_{u-} -a\big)\big|^2\big|\widetilde\rho\big(Z^{(M),\widetilde\xi}_{u-}\big)\big|^2 \diff u\Big]\Big)^\frac{1}{2}
    \leq \Big( 2\hat c\, c_{\widetilde\rho} \int_0^t\E\Big[ \big( 1+ \big|Z^{(M),\widetilde\xi}_{u-}\big|^2\big) \Big]\diff u\Big)^\frac{1}{2}\\
    &\leq \big(2\hat c\, c_{\widetilde\rho} T ( 1+ c_1)\big)^\frac{1}{2} (1+ |\widetilde\xi|).
    \end{aligned}
\end{equation}
Combining \eqref{pOTE12}, \eqref{pOTE13}, \eqref{pOTE15}, and \eqref{pOTE16} we obtain that there exists a constant $\widetilde c_1\in(0,\infty)$ such that
\begin{equation}\label{pOTE17}
    \begin{aligned}
    &\E\Big[\Big|\int_0^t \operatorname{sgn}\big(Z^{(M),\widetilde\xi}_{u-} -a\big) \diff Z^{(M),\widetilde\xi}_{u}\Big|\Big]
    \leq \widetilde c_1 (1+ |\widetilde\xi|).
    \end{aligned}
\end{equation}
Next we consider the expectation of the third summand of \eqref{pOTE9}. We use the Cauchy-Schwarz inequality, Lemma \ref{BDGaKunita}, and Lemma  \ref{ME} \eqref{LME1} to obtain that
\begin{equation}\label{pOTE18}
    \begin{aligned}
    & \E\Big[\Big|\int_0^t \Big( \big| Z^{(M),\widetilde\xi}_{u-} + \widetilde\rho\big(Z^{(M),\widetilde\xi}_{u-}\big)- a\big| - \big| Z^{(M),\widetilde\xi}_{u-} - a\big| - \operatorname{sgn}\big(Z^{(M),\widetilde\xi}_{u-} -a\big) \widetilde\rho\big(Z^{(M),\widetilde\xi}_{u-}\big)\Big) \diff N_u
    \Big|\Big]\\
    &\leq \hat c\, \E\Big[\int_0^t \Big| \big| Z^{(M),\widetilde\xi}_{u-} + \widetilde\rho\big(Z^{(M),\widetilde\xi}_{u-}\big)- a\big| - \big| Z^{(M),\widetilde\xi}_{u-} - a\big| - \operatorname{sgn}\big(Z^{(M),\widetilde\xi}_{u-} -a\big) \widetilde\rho\big(Z^{(M),\widetilde\xi}_{u-}\big)\Big|^2 \diff u
    \Big]^\frac{1}{2}\\
    &\leq \hat c\, \E\Big[2\, \int_0^t \Big| \big| Z^{(M),\widetilde\xi}_{u-} + \widetilde\rho\big(Z^{(M),\widetilde\xi}_{u-}\big)- a\big| - \big| Z^{(M),\widetilde\xi}_{u-} - a\big|\Big|^2 + \Big| \operatorname{sgn}\big(Z^{(M),\widetilde\xi}_{u-} -a\big) \widetilde\rho\big(Z^{(M),\widetilde\xi}_{u-}\big)\Big|^2 \diff u
    \Big]^\frac{1}{2}\\
    &\leq \hat c\, \Big(4\, \int_0^t \E\Big[ \big|\widetilde\rho\big(Z^{(M),\widetilde\xi}_{u-}\big)\big|^2\Big]\diff u\Big)^\frac{1}{2}
    \leq \hat c\, \Big(8\,c_{\widetilde\rho}^2 \int_0^t \E\Big[ \big(1+ \big|Z^{(M),\widetilde\xi}_{u-}\big|^2\big)\Big]\diff u\Big)^\frac{1}{2}\\
    &\leq \hat c\, \Big(8\,c_{\widetilde\rho}^2\, T\, (1+ c_1)\Big)^\frac{1}{2} (1+|\widetilde\xi|).
    \end{aligned}
\end{equation}
Plugging \eqref{pOTE10}, \eqref{pOTE17}, and \eqref{pOTE18} into \eqref{pOTE9} we obtain that there exists a constant $\widetilde c_2 \in(0,\infty)$ such that
\begin{equation}\label{pOTE19}
    \begin{aligned}
    &\E\big[L^a_t \big(Z^{(M),\widetilde\xi}\big)\big] 
    \leq \widetilde c_2(1+|\widetilde\xi|).
    \end{aligned}
\end{equation}
Note that the continuous martingale part of \eqref{pOTE1} is
\begin{equation}\label{pOTE20}
    \begin{aligned}
    &\Big(\int_0^t \sum_{n=0}^{\infty} \Big( \widetilde\sigma \big(Z^{(M),\widetilde\xi}_{\tau_n}\big) + \int_{\tau_n}^u \widetilde\sigma \big(Z^{(M),\widetilde\xi}_{\tau_n}\big)d_{\widetilde\sigma} \big(Z^{(M),\widetilde\xi}_{\tau_n}\big)\diff W_s \Big)\mathds{1}_{(\tau_n,\tau_{n+1}]}(u)\diff W_u\Big)_{t\in[0,T]}.
    \end{aligned}
\end{equation}
For this we get using \cite[Theorem 88]{situ2005},
\begin{equation}\label{pOTE20a}
    \begin{aligned}
    &\Big\langle \int_0^\cdot \sum_{n=0}^{\infty} \Big( \widetilde\sigma \big(Z^{(M),\widetilde\xi}_{\tau_n}\big) + \int_{\tau_n}^u \widetilde\sigma \big(Z^{(M),\widetilde\xi}_{\tau_n}\big)d_{\widetilde\sigma} \big(Z^{(M),\widetilde\xi}_{\tau_n}\big)\diff W_s \Big)\mathds{1}_{(\tau_n,\tau_{n+1}]}(u)\diff W_u\Big\rangle_t\\
    &= \int_0^t \Big(\sum_{n=0}^{\infty} \Big( \widetilde\sigma \big(Z^{(M),\widetilde\xi}_{\tau_n}\big) + \int_{\tau_n}^u \widetilde\sigma \big(Z^{(M),\widetilde\xi}_{\tau_n}\big)d_{\widetilde\sigma} \big(Z^{(M),\widetilde\xi}_{\tau_n}\big)\diff W_s \Big)\mathds{1}_{(\tau_n,\tau_{n+1}]}(u)\Big)^2\diff u\\
    \end{aligned}
\end{equation}
Since $f\colon\R\to\R$, $f(x) =\mathds{1}_{[\zeta-\varepsilon,\zeta+\varepsilon]}(x)$ is a Borel-measurable and bounded function and the continuous martingale part of \eqref{pOTE1} is given by \eqref{pOTE20}, we can apply \cite[Lemma 159]{situ2005} to these. Using this and  \eqref{pOTE19} we obtain
\begin{equation}\label{pOTE21}
    \begin{aligned}
    &\E\Big[\int_0^t \mathds{1}_{[\zeta-\varepsilon,\zeta+\varepsilon]}(Z^{(M),\widetilde\xi}_{u}) \Big(\sum_{n=0}^{\infty} \Big( \widetilde\sigma \big(Z^{(M),\widetilde\xi}_{\tau_n}\big) + \int_{\tau_n}^u \widetilde\sigma \big(Z^{(M),\widetilde\xi}_{\tau_n}\big)d_{\widetilde\sigma} \big(Z^{(M),\widetilde\xi}_{\tau_n}\big)\diff W_s \Big)\mathds{1}_{(\tau_n,\tau_{n+1}]}(u) \Big)^2 \diff u \Big]\\
    &=\E\Big[\int_0^t \mathds{1}_{[\zeta-\varepsilon,\zeta+\varepsilon]}(Z^{(M),\widetilde\xi}_{u}) \diff \Big\langle \int_0^\cdot \sum_{n=0}^{\infty} \Big( \widetilde\sigma \big(Z^{(M),\widetilde\xi}_{\tau_n}\big) \\
    &\quad\quad\quad\quad \quad\quad\quad\quad \quad\quad\quad\quad \quad\quad\quad\quad \quad\quad+ \int_{\tau_n}^v \widetilde\sigma \big(Z^{(M),\widetilde\xi}_{\tau_n}\big)d_{\widetilde\sigma} \big(Z^{(M),\widetilde\xi}_{\tau_n}\big)\diff W_s \Big)\mathds{1}_{(\tau_n,\tau_{n+1}]}(v)\diff W_v\Big\rangle_u\Big]\\
    &=\E\Big[\int_\R \mathds{1}_{[\zeta-\varepsilon,\zeta+\varepsilon]}(a)L^a_t \big(Z^{(M),\widetilde\xi}\big) \diff a \Big]
    \leq 2 \varepsilon \widetilde c_2 (1+|\widetilde \xi|).
    \end{aligned}
\end{equation}
Since $|a^2-b^2| = |a-b|\cdot|a+b|$ for all $a,b\in\R$, we 
get 
\begin{equation}\label{pOTE22a}
    \begin{aligned}
    &\E\Big[ \Big| \widetilde\sigma^2\big(Z^{(M),\widetilde\xi}_{u} \big)- \Big(\sum_{n=0}^{\infty} \Big( \widetilde\sigma \big(Z^{(M),\widetilde\xi}_{\tau_n}\big) + \int_{\tau_n}^u \widetilde\sigma \big(Z^{(M),\widetilde\xi}_{\tau_n}\big)d_{\widetilde\sigma} \big(Z^{(M),\widetilde\xi}_{\tau_n}\big)\diff W_s \Big)\mathds{1}_{(\tau_n,\tau_{n+1}]}(u) \Big)^2\Big|\Big]\\
    &\leq\E\Big[\Big| \widetilde\sigma\big(Z^{(M),\widetilde\xi}_{u} \big)- \sum_{n=0}^{\infty} \Big( \widetilde\sigma \big(Z^{(M),\widetilde\xi}_{\tau_n}\big) + \int_{\tau_n}^u \widetilde\sigma \big(Z^{(M),\widetilde\xi}_{\tau_n}\big)d_{\widetilde\sigma} \big(Z^{(M),\widetilde\xi}_{\tau_n}\big)\diff W_s \Big)\mathds{1}_{(\tau_n,\tau_{n+1}]}(u) \Big|^2\Big]^{\frac{1}{2}}\\
    &\quad \cdot \E\Big[\Big| \widetilde\sigma\big(Z^{(M),\widetilde\xi}_{u} \big)+ \sum_{n=0}^{\infty} \Big( \widetilde\sigma \big(Z^{(M),\widetilde\xi}_{\tau_n}\big) + \int_{\tau_n}^u \widetilde\sigma \big(Z^{(M),\widetilde\xi}_{\tau_n}\big)d_{\widetilde\sigma} \big(Z^{(M),\widetilde\xi}_{\tau_n}\big)\diff W_s \Big)\mathds{1}_{(\tau_n,\tau_{n+1}]}(u)\Big|^2\Big]^{\frac{1}{2}}.
    \end{aligned}
\end{equation}
For estimating the first factor in \eqref{pOTE22a} we will use
\begin{equation}\label{pOTE23}
    \begin{aligned}
    &\Big| \widetilde\sigma\big(Z^{(M),\widetilde\xi}_{u} \big)- \sum_{n=0}^{\infty} \Big( \widetilde\sigma \big(Z^{(M),\widetilde\xi}_{\tau_n}\big) + \int_{\tau_n}^u \widetilde\sigma \big(Z^{(M),\widetilde\xi}_{\tau_n}\big)d_{\widetilde\sigma} \big(Z^{(M),\widetilde\xi}_{\tau_n}\big)\diff W_s \Big)\mathds{1}_{(\tau_n,\tau_{n+1}]}(u) \Big|\\
    &= \sum_{n=0}^{\infty} \Big| \widetilde\sigma\big(Z^{(M),\widetilde\xi}_{u} \big)- \widetilde\sigma \big(Z^{(M),\widetilde\xi}_{\tau_n}\big) - \int_{\tau_n}^u \widetilde\sigma \big(Z^{(M),\widetilde\xi}_{\tau_n}\big)d_{\widetilde\sigma} \big(Z^{(M),\widetilde\xi}_{\tau_n}\big)\diff W_s \Big|\mathds{1}_{(\tau_n,\tau_{n+1}]}(u) \\
    &\leq \sum_{n=0}^{\infty} \Big( L_{\widetilde\sigma} \big| Z^{(M),\widetilde\xi}_{u} -Z^{(M),\widetilde\xi}_{\tau_n}\big| + L_{\widetilde\sigma} c_{\widetilde\sigma} \big(1+ \big| Z^{(M),\widetilde\xi}_{\tau_n}\big|\big) |W_u-W_{\tau_n}|\Big)\mathds{1}_{(\tau_n,\tau_{n+1}]}(u).
    \end{aligned}
\end{equation}
This, Lemma \ref{MW}, and Lemma \ref{ME} \eqref{LME1} and \eqref{LME2} yield 
\begin{equation}\label{pOTE23a}
    \begin{aligned}
    &\E\Big[\Big| \widetilde\sigma\big(Z^{(M),\widetilde\xi}_{u} \big)- \sum_{n=0}^{\infty} \Big( \widetilde\sigma \big(Z^{(M),\widetilde\xi}_{\tau_n}\big) + \int_{\tau_n}^u \widetilde\sigma \big(Z^{(M),\widetilde\xi}_{\tau_n}\big)d_{\widetilde\sigma} \big(Z^{(M),\widetilde\xi}_{\tau_n}\big)\diff W_s \Big)\mathds{1}_{(\tau_n,\tau_{n+1}]}(u) \Big|^2\Big]^{\frac{1}{2}}\\
    &\leq 2\Big( L_{\widetilde\sigma}^2\, \E\Big[ \sum_{n=0}^{\infty} \big| Z^{(M),\widetilde\xi}_{u} -Z^{(M),\widetilde\xi}_{\tau_n}\big|^2 \mathds{1}_{(\tau_n,\tau_{n+1}]}(u) \Big] \\
    &\quad\quad\quad+  L_{\widetilde\sigma}^2 c_{\widetilde\sigma}^2\, \E\Big[ \sum_{n=0}^{\infty} \big(1+ \big| Z^{(M),\widetilde\xi}_{\tau_n}\big|^2\big) |W_u-W_{\tau_n}|^2\mathds{1}_{(\tau_n,\tau_{n+1}]}(u) \Big]\Big)^\frac{1}{2}\\
    &\leq 2\Big( L_{\widetilde\sigma}^2 c_2(1+|\widetilde\xi|^2)\delta 
    + L_{\widetilde\sigma}^2 c_{\widetilde\sigma}^2  c_{W_2} \delta\, \E\Big[ \sum_{n=0}^{\infty} \big(1+ \big| Z^{(M),\widetilde\xi}_{\tau_n}\big|^2\big) \mathds{1}_{(\tau_n,\tau_{n+1}]}(u) \Big]\Big)^\frac{1}{2}\\
    &\leq 2\Big( L_{\widetilde\sigma}^2 c_2
    + L_{\widetilde\sigma}^2 c_{\widetilde\sigma}^2 c_{W_2} (1+c_1) \Big)^\frac{1}{2}(1+|\widetilde\xi|)\delta^{\frac{1}{2}}.
    \end{aligned}
\end{equation}
For estimating the second factor of \eqref{pOTE22a} we will use \begin{equation}\label{pOTE24}
    \begin{aligned}
    &\Big| \widetilde\sigma\big(Z^{(M),\widetilde\xi}_{u} \big)+ \sum_{n=0}^{\infty} \Big( \widetilde\sigma \big(Z^{(M),\widetilde\xi}_{\tau_n}\big) + \int_{\tau_n}^u \widetilde\sigma \big(Z^{(M),\widetilde\xi}_{\tau_n}\big)d_{\widetilde\sigma} \big(Z^{(M),\widetilde\xi}_{\tau_n}\big)\diff W_s \Big)\mathds{1}_{(\tau_n,\tau_{n+1}]}(u) \Big|\\
    &\leq \sum_{n=0}^{\infty} \Big( c_{\widetilde\sigma}\big(1+ \big| Z^{(M),\widetilde\xi}_{u}\big|\big) + c_{\widetilde\sigma} \big(1+ \big| Z^{(M),\widetilde\xi}_{\tau_n}\big|\big)
    + L_{\widetilde\sigma} c_{\widetilde\sigma} \big(1+ \big| Z^{(M),\widetilde\xi}_{\tau_n}\big|\big) |W_u-W_{\tau_n}|\Big)\mathds{1}_{(\tau_n,\tau_{n+1}]}(u).
    \end{aligned}
\end{equation}
This, Lemma \ref{MW}, and Lemma \ref{ME} \eqref{LME1} show
\begin{equation}\label{pOTE24a}
    \begin{aligned}
    &\E\Big[\Big| \widetilde\sigma\big(Z^{(M),\widetilde\xi}_{u} \big)+ \sum_{n=0}^{\infty} \Big( \widetilde\sigma \big(Z^{(M),\widetilde\xi}_{\tau_n}\big) + \int_{\tau_n}^u \widetilde\sigma \big(Z^{(M),\widetilde\xi}_{\tau_n}\big)d_{\widetilde\sigma} \big(Z^{(M),\widetilde\xi}_{\tau_n}\big)\diff W_s \Big)\mathds{1}_{(\tau_n,\tau_{n+1}]}(u)\Big|^2\Big]^{\frac{1}{2}}\\
    &\leq \sqrt{3}\, \Big( 2 c_{\widetilde\sigma}^2\, \E\Big[\sum_{n=0}^{\infty} \big(1+ \big| Z^{(M),\widetilde\xi}_{u}\big|^2\big) \mathds{1}_{(\tau_n,\tau_{n+1}]}(u)\Big]
    +2 c_{\widetilde\sigma}^2\, \E\Big[\sum_{n=0}^{\infty} \big(1+ \big| Z^{(M),\widetilde\xi}_{\tau_n}\big|^2\big)\mathds{1}_{(\tau_n,\tau_{n+1}]}(u)\Big]\\
    &\quad\quad\quad\quad +2 L_{\widetilde\sigma}^2 c_{\widetilde\sigma}^2 \, \E\Big[\sum_{n=0}^{\infty}  \big(1+ \big| Z^{(M),\widetilde\xi}_{\tau_n}\big|^2\big) |W_u-W_{\tau_n}|^2\mathds{1}_{(\tau_n,\tau_{n+1}]}(u)\Big]\Big)^{\frac{1}{2}}\\
    &\leq \sqrt{3}\, \Big( 2 c_{\widetilde\sigma}^2 \,\E\Big[ \big(1+\sup_{t\in[0,T]} \big| Z^{(M),\widetilde\xi}_{t}\big|^2\big) \Big]
    +2 c_{\widetilde\sigma}^2\, \E\Big[ \big(1+\sup_{t\in[0,T]} \big| Z^{(M),\widetilde\xi}_{t}\big|^2\big)\Big]\\
    &\quad\quad\quad\quad +2 L_{\widetilde\sigma}^2 c_{\widetilde\sigma}^2 c_{W_2} \delta\, \E\Big[\sum_{n=0}^{\infty}  \big(1+ \big| Z^{(M),\widetilde\xi}_{\tau_n}\big|^2\big) \mathds{1}_{(\tau_n,\tau_{n+1}]}(u)\Big]\Big)^{\frac{1}{2}}\\
    &\leq \sqrt{3}\, \Big( 4 c_{\widetilde\sigma}^2 
    +2 L_{\widetilde\sigma}^2 c_{\widetilde\sigma}^2 c_{W_2} T \Big)^{\frac{1}{2}}(1+c_1)^\frac{1}{2}  (1+|\widetilde\xi|).
    \end{aligned}
\end{equation}
Plugging \eqref{pOTE23a} and \eqref{pOTE24a} into \eqref{pOTE22a} we obtain that there exists a constant $\widetilde c_3\in(0,\infty)$ such that
\begin{equation}\label{pOTE25}
    \begin{aligned}
    &\E\Big[ \Big| \widetilde\sigma^2\big(Z^{(M),\widetilde\xi}_{u} \big)- \Big(\sum_{n=0}^{\infty} \Big( \widetilde\sigma \big(Z^{(M),\widetilde\xi}_{\tau_n}\big) + \int_{\tau_n}^u \widetilde\sigma \big(Z^{(M),\widetilde\xi}_{\tau_n}\big)d_{\widetilde\sigma} \big(Z^{(M),\widetilde\xi}_{\tau_n}\big)\diff W_s \Big)\mathds{1}_{(\tau_n,\tau_{n+1}]}(u) \Big)^2\Big|\Big]\\
    &\leq \widetilde c_3 (1+|\widetilde\xi|)^2 \delta^{\frac{1}{2}} 
    \leq 2 \widetilde  c_3 (1+|\widetilde\xi|^2) \delta^{\frac{1}{2}}.
    \end{aligned}
\end{equation}
Since $\widetilde\sigma(\zeta)\neq 0$, there exist $\varepsilon_0 , \widetilde c_4 \in(0,\infty)$ such that
\begin{equation}\label{pOTE26}
    \begin{aligned}
    &\inf_{\substack{z\in\R\\|z-\zeta|<\varepsilon_0} } \big|\widetilde\sigma(z)\big|^2 >\widetilde c_4.
    \end{aligned}
\end{equation}
Hence \eqref{pOTE26}, \eqref{pOTE21}, and \eqref{pOTE25} ensure for all $\varepsilon\in(0,\varepsilon_0)$,
\begin{equation}\label{pOTE27}
    \begin{aligned}
    &\int_0^T \P\Big( \big| Z^{(M),\widetilde\xi}_{u} -\zeta\big| <\varepsilon\Big)\diff u
    = \frac{1}{\widetilde c_4}\int_0^T \E\Big[ \mathds{1}_{(\zeta-\varepsilon,\zeta+\varepsilon)} \big(\big| Z^{(M),\widetilde\xi}_{u}\big|\big) \widetilde c_4 \Big]\diff u\\
    &\leq \frac{1}{\widetilde c_4} \E\Big[ \int_0^T\mathds{1}_{(\zeta-\varepsilon,\zeta+\varepsilon)} \big(\big| Z^{(M),\widetilde\xi}_{u}\big|\big) \widetilde\sigma^2\big(Z^{(M),\widetilde\xi}_{u}\big) \diff u\Big]\\
    &= \frac{1}{\widetilde c_4} \E\Big[ \int_0^T\mathds{1}_{(\zeta-\varepsilon,\zeta+\varepsilon)} \big(\big| Z^{(M),\widetilde\xi}_{u}\big|\big)\Big(\sum_{n=0}^{N_T+M} \Big( \widetilde\sigma \big(Z^{(M),\widetilde\xi}_{\tau_n}\big)\\
    &\quad\quad\quad\quad \quad\quad\quad\quad \quad\quad\quad\quad \quad\quad\quad+ \int_{\tau_n}^u \widetilde\sigma \big(Z^{(M),\widetilde\xi}_{\tau_n}\big)d_{\widetilde\sigma} \big(Z^{(M),\widetilde\xi}_{\tau_n}\big)\diff W_s \Big)\mathds{1}_{(\tau_n,\tau_{n+1}]}(u) \Big)^2  \diff u\Big]\\
    &\quad + \frac{1}{\widetilde c_4} \E\Big[ \int_0^T\mathds{1}_{(\zeta-\varepsilon,\zeta+\varepsilon)} \big(\big| Z^{(M),\widetilde\xi}_{u}\big|\big) \Big(\widetilde\sigma^2\big(Z^{(M),\widetilde\xi}_{u}\big) - \Big(\sum_{n=0}^{N_T+M} \Big( \widetilde\sigma \big(Z^{(M),\widetilde\xi}_{\tau_n}\big) \\
    &\quad\quad\quad\quad \quad\quad\quad\quad \quad\quad\quad\quad \quad\quad\quad\quad+ \int_{\tau_n}^u \widetilde\sigma \big(Z^{(M),\widetilde\xi}_{\tau_n}\big)d_{\widetilde\sigma} \big(Z^{(M),\widetilde\xi}_{\tau_n}\big)\diff W_s \Big)\mathds{1}_{(\tau_n,\tau_{n+1}]}(u) \Big)^2\Big)\diff u\Big]\\
    &\leq \frac{1}{\widetilde c_4}  2\varepsilon\, \widetilde c_2 (1+|\widetilde\xi|) +2 \frac{1}{\widetilde c_4} T  \widetilde  c_3 (1+|\widetilde\xi|^2) \delta^{\frac{1}{2}}
    \leq \frac{1}{\widetilde c_4} ( 4\widetilde c_2  + 2 T \widetilde  c_3) (1+|\widetilde\xi|^2)(\varepsilon + \delta^{\frac{1}{2}}),
    \end{aligned}
\end{equation}
where in the last inequality we used that for $\widetilde \xi \in\R$, $|\widetilde\xi|\leq 1+ |\widetilde\xi|^2$.
For $\varepsilon\in[\varepsilon_0,\infty)$ we have
\begin{equation}\label{pOTE28}
    \begin{aligned}
    &\int_0^T \P\Big( \big| Z^{(M),\widetilde\xi}_{u} -\zeta\big| <\varepsilon\Big)\diff u
    \leq T = \frac{T}{\varepsilon_0} \varepsilon_0 
    \leq \frac{T}{\varepsilon_0} (\varepsilon + \delta^{\frac{1}{2}})
    \leq \frac{T}{\varepsilon_0} (1+|\widetilde\xi|^2) (\varepsilon + \delta^{\frac{1}{2}}).
    \end{aligned}
\end{equation}
Choosing $c_5 = \max\big\{\frac{1}{\widetilde c_4} (4\widetilde c_2  + 2 T \widetilde  c_3),\frac{T}{\varepsilon_0}\big\}$ proves the claim.
\end{proof}

\begin{lemma}\label{LemmaEst1}
Let Assumption \ref{assZ} hold. Let $q\in\N$ and $\zeta\in\R$ with $\widetilde\sigma (\zeta) \neq 0$. Let
\begin{equation}
    \begin{aligned}
    S_{\zeta} = \{(x,y)\in\R^2 : (x-\zeta)(y-\zeta) \leq 0\}.
    \end{aligned}
\end{equation}
Then there exists a constant $c_6\in(0,\infty)$ such that for all $M\in\N$, $\delta =\frac{T}{M}$, all $s,t\in[0,T]$ with $\underline{t} -s \geq \delta$, and all $\F_s$-measurable, non-negative, real-valued random variables $Y$,
\begin{equation}
    \begin{aligned}
    &\E\Big[Y\sum_{n=0}^{\infty} |W_t -W_{\tau_n}|^q \mathds{1}_{S_\zeta} (Z_t^{(M)},Z^{(M)}_{\tau_n}) \mathds{1}_{(\tau_n,\tau_{n+1}]} (t) \Big]\\
    &\quad \leq c_6 \Big( \delta^{\frac{q+1}{2}} \E[Y] + \delta^{\frac{q}{2}} \int_\R |z|^q \E\Big[ Y \mathds{1}_{\{|Z^{(M)}_{\underline{t} -(t-\underline{t})} -\zeta|\leq c_6 (1+|z|)\sqrt{\delta}\}} \Big] e^{-\frac{z^2}{2}} \diff z\Big).
    \end{aligned}
\end{equation}
\end{lemma}

\begin{proof}
First, observe that
\begin{equation}\label{LEst1eq1}
    \begin{aligned}
    &\E\Big[Y\sum_{n=0}^{\infty} |W_t -W_{\tau_n}|^q \mathds{1}_{S_\zeta} (Z_t^{(M)},Z^{(M)}_{\tau_n}) \mathds{1}_{(\tau_n,\tau_{n+1}]} (t) \Big]\\
    & =  \E\Big[Y\sum_{n=0}^{N_T +M} |W_t -W_{\tau_n}|^q \mathds{1}_{S_\zeta} (Z_t^{(M)},Z^{(M)}_{\tau_n}) \mathds{1}_{(\tau_n,\tau_{n+1}]} (t) \mathds{1}_{\{N_t - N_{\underline{t}-\delta} > 0\}}\Big]\\
    &\quad\quad+ \E\Big[Y \sum_{n=0}^{N_T +M}|W_t -W_{\tau_n}|^q \mathds{1}_{S_\zeta} (Z_t^{(M)},Z^{(M)}_{\tau_n}) \mathds{1}_{(\tau_n,\tau_{n+1}]} (t) \mathds{1}_{\{N_t - N_{\underline{t}-\delta} = 0\}}\Big].
    \end{aligned}
\end{equation}
We prove that $Y$ and 
\begin{equation}\label{Indeq1}
    \sum_{n=0}^{N_T +M} \sup_{u\in[0,\delta]}|W_{u+\tau_n} -W_{\tau_n}|^q \mathds{1}_{(\tau_n,\tau_{n+1}]} (t) \mathds{1}_{\{N_t - N_{\underline{t}-\delta} > 0\}}
\end{equation}
are independent.
For this we rewrite \eqref{Indeq1} as 
\begin{equation}\label{Indeq2}
    \begin{aligned}
    &\sum_{n=0}^{N_T +M} \sup_{u\in[0,\delta]}|W_{u+\tau_n} -W_{\tau_n}|^q \mathds{1}_{(\tau_n,\tau_{n+1}]} (t) \mathds{1}_{\{N_t - N_{\underline{t}-\delta} > 0\}}\\
    &=\sum_{n=N_{\underline{t}- \delta}+\underline{t}\delta^{-1}-1}^{N_T +M} \sup_{u\in[0,\delta]}|W_{u+\tau_n} -W_{\tau_n}|^q \mathds{1}_{(\tau_n,\tau_{n+1}]} (t) \mathds{1}_{\{N_t - N_{\underline{t}-\frac{T}{M}} > 0\}}\\
    &=\sum_{n=0}^{N_T -N_{\underline{t}- \delta}+M-\underline{t}\delta^{-1}+1} \sup_{u\in[0,\delta]}|W_{u+\tau_{n+N_{\underline{t}- \delta}+\underline{t}\delta^{-1}-1}} -W_{\tau_{n+N_{\underline{t}- \delta}+\underline{t}\delta^{-1}-1}}|^q\\
    &\quad\quad\quad\quad \quad\quad\quad\quad \quad\quad\quad\quad \cdot\mathds{1}_{(\tau_{n+N_{\underline{t}- \delta}+\underline{t}\delta^{-1}-1},\tau_{n+N_{\underline{t}- \delta}+\underline{t}\delta^{-1}}]} (t) \mathds{1}_{\{N_t - N_{\underline{t}-\delta} > 0\}}.
    \end{aligned}
\end{equation}
Next we denote by $(\hat\tau_n)_{n\in\N}$ a sequence of points defined similar to $(\tau_n)_{n\in\N}$ but based on $(N_{\underline{t}-\delta+t}-N_{\underline{t} -\delta})_{t\in[0,\infty)}$ instead of $(N_{t})_{t\in[0,\infty)}$.
Then it holds that 
\begin{equation}\label{Indeq3}
    \begin{aligned}
    &\{\tau_{n+N_{\underline{t}- \delta}+\underline{t}\delta^{-1}-1}\leq t\} 
    = \Big\{N_{t} + \Big\lfloor \frac{tM}{T}\Big\rfloor \leq n+N_{\underline{t}- \delta}+\underline{t}\delta^{-1}-1\Big\}\\
    &= \Big\{N_{t} -N_{\underline{t}-\delta} + 1 \leq n \Big\}
    = \Big\{\hat\tau_n  \leq t-\underline{t}+\delta \Big\}
    = \Big\{\hat\tau_n +\underline{t}-\delta  \leq t\Big\}.
    \end{aligned}
\end{equation}
Hence,
\begin{equation}\label{Indeq4}
    \begin{aligned}
    &\sum_{n=0}^{N_T +M} \sup_{u\in[0,\delta]}|W_{u+\tau_n} -W_{\tau_n}|^q \mathds{1}_{(\tau_n,\tau_{n+1}]} (t) \mathds{1}_{\{N_t - N_{\underline{t}-\delta} > 0\}}\\
    &=\sum_{n=0}^{N_T -N_{\underline{t}- \delta}+M-\underline{t}\delta^{-1}+1} \sup_{u\in[0,\delta]}|W_{u+\hat \tau_n + \underline{t}-\delta } -W_{\hat \tau_n + \underline{t}-\delta }|^q  \mathds{1}_{(\hat\tau_n + \underline{t}-\delta,\hat\tau_n + \underline{t}-\delta+1]} (t) \mathds{1}_{\{N_t - N_{\underline{t}-\delta} > 0\}},
    \end{aligned}
\end{equation}
which shows that \eqref{Indeq1} is independent of $\mathbb{F}_{\underline{t}-\delta}$. As $Y$ is $\mathbb{F}_{s}\subset\mathbb{F}_{\underline{t}-\delta}$-measurable, this proves the independence of $Y$ and \eqref{Indeq1}.
Therefore the first summand of \eqref{LEst1eq1} becomes
\begin{equation}\label{LEst1eq6}
    \begin{aligned}
    &\E\Big[Y\sum_{n=0}^{N_T +M} |W_t -W_{\tau_n}|^q \mathds{1}_{S_\zeta} (Z_t^{(M)},Z^{(M)}_{\tau_n}) \mathds{1}_{(\tau_n,\tau_{n+1}]} (t) \mathds{1}_{\{N_t - N_{\underline{t}-\delta\}} > 0}\Big]\\
    &\leq \E\Big[Y\sum_{n=0}^{N_T +M}  \sup_{u\in[0,\delta]}|W_{u+\tau_n} -W_{\tau_n}|^q  \mathds{1}_{(\tau_n,\tau_{n+1}]} (t) \mathds{1}_{\{N_t - N_{\underline{t}-\delta} > 0\}}\Big]\\
    &=\E[Y]\,\E\Big[\sum_{n=0}^{N_T +M} \sup_{u\in[0,\delta]}|W_{u+ \tau_n } -W_{\tau_n}|^q  \mathds{1}_{(\tau_n,\tau_n +1]} (t) \mathds{1}_{\{N_t - N_{\underline{t}-\delta} > 0\}}\Big].
    \end{aligned}
\end{equation}
By the Cauchy-Schwarz inequality,
\begin{equation}\label{LEst1eq7}
    \begin{aligned}
    &\E\Big[Y\sum_{n=0}^{N_T +M} |W_t -W_{\tau_n}|^q \mathds{1}_{S_\zeta} (Z_t^{(M)},Z^{(M)}_{\tau_n}) \mathds{1}_{(\tau_n,\tau_{n+1}]} (t) \mathds{1}_{\{N_t - N_{\underline{t}-\delta} > 0\}}\Big]\\
    &\leq \E[Y]\,\E\Big[\Big(\sum_{n=0}^{N_T +M} \sup_{u\in[0,\delta]}|W_{u+ \tau_n } -W_{\tau_n}|^q  \mathds{1}_{(\tau_n,\tau_n +1]} (t) \big)^2\Big]^\frac{1}{2} \E\big[\mathds{1}_{\{N_t - N_{\underline{t}-\delta} > 0\}}\big]^\frac{1}{2}.
    \end{aligned}
\end{equation}
Observe that
\begin{equation}\label{LEst1eq8}
    \begin{aligned}
    &\E\big[\mathds{1}_{\{N_t - N_{\underline{t}-\delta} > 0\}}\big]
    =\P\big( N_t - N_{\underline{t}-\delta} > 0\big)
    = 1 - \P\big( N_t - N_{\underline{t}-\delta} = 0\big)\\
    &= 1 - \exp\big(-(t-\underline{t}+\delta)\lambda\big)
    \leq 1 - \exp\big(-2 \lambda \delta\big)
    \leq 2\lambda\delta.
    \end{aligned}
\end{equation}
Lemma \ref{MW} equation \eqref{MWeq2} gives
\begin{equation}\label{LEst1eq9}
    \begin{aligned}
    &\E\Big[\Big(\sum_{n=0}^{N_T +M} \sup_{u\in[0,\delta]}|W_{u+ \tau_n } -W_{\tau_n}|^q  \mathds{1}_{(\tau_n,\tau_n +1]} (t) \Big)^2\Big]
    \leq c_{W_{2q}} \delta^{q}.
    \end{aligned}
\end{equation}
Plugging \eqref{LEst1eq8} and  \eqref{LEst1eq9} into \eqref{LEst1eq7} we obtain
\begin{equation}\label{LEst1eq10}
    \begin{aligned}
    &\E\Big[Y\sum_{n=0}^{N_T +M} |W_t -W_{\tau_n}|^q \mathds{1}_{S_\zeta} (Z_t^{(M)},Z^{(M)}_{\tau_n}) \mathds{1}_{(\tau_n,\tau_{n+1}]} (t) \mathds{1}_{\{N_t - N_{\underline{t}-\delta} > 0\}}\Big]\\
    &\leq \E[Y]\,\big(c_{W_{2q}} \delta^{q}\big)^\frac{1}{2} \big(2\lambda\delta \big)^\frac{1}{2}
    =  (2c_{W_{2q}}\lambda)^\frac{1}{2} \delta^\frac{q+1}{2}\E[Y].
    \end{aligned}
\end{equation}
For estimating the second summand of \eqref{LEst1eq1} note that
\begin{equation}\label{LEst1eq12}
    \begin{aligned}
    \mathds{1}_{(\tau_n,\tau_{n+1}]} (t) \mathds{1}_{\{N_t - N_{\underline{t}-\delta} = 0\}}
    = \mathds{1}_{(\tau_n,\tau_{n+1}]} (t) \mathds{1}_{\{N_t - N_{\underline{t}-\delta} = 0\}} \sum_{m=0}^{M} \mathds{1}_{\{\tau_n = s_m\}},
    \end{aligned}
\end{equation}
since assuming $\tau_n \neq s_m$ for all $m\in\{0,\ldots,M\}$ implies $\underline{t}-\delta < \tau_n < t$ and $\Delta N_{\tau_n} = 1$, which is a contradiction to $\mathds{1}_{\{N_t - N_{\underline{t}-\delta} = 0\}}$.
This ensures, 
\begin{equation}\label{LEst1eq13}
    \begin{aligned}
    &\E\Big[Y \sum_{n=0}^{N_T +M}|W_t -W_{\tau_n}|^q \mathds{1}_{S_\zeta} (Z_t^{(M)},Z^{(M)}_{\tau_n}) \mathds{1}_{(\tau_n,\tau_{n+1}]} (t) \mathds{1}_{\{N_t - N_{\underline{t}-\delta} = 0\}}\Big]\\
    &= \E\Big[Y\sum_{n=0}^{N_T +M}\sum_{m=0}^{M} |W_t -W_{\tau_n}|^q \mathds{1}_{S_\zeta} (Z_t^{(M)},Z^{(M)}_{\tau_n}) \mathds{1}_{(\tau_n,\tau_{n+1}]} (t) \mathds{1}_{\{N_t - N_{\underline{t}-\delta} = 0\}}  \mathds{1}_{\{\tau_n = s_m\}}\Big]\\
    &= \E\Big[Y\sum_{n=0}^{N_T +M}\sum_{m=0}^{M} |W_t -W_{s_m}|^q \mathds{1}_{S_\zeta} (Z_t^{(M)},Z^{(M)}_{s_m})\\
    &\quad\quad\quad\quad \quad\quad\quad\quad \quad \cdot \mathds{1}_{(s_m,s_{m+1} \wedge \min\{\nu_i\colon i\in\N,\, \nu_i >s_m\}]} (t) \mathds{1}_{\{N_t - N_{\underline{t}-\delta} = 0\}}  \mathds{1}_{\{\tau_n = s_m\}}\Big]\\
    &= \sum_{m=0}^{M} \E\Big[Y |W_t -W_{s_m}|^q \mathds{1}_{S_\zeta} (Z_t^{(M)},Z^{(M)}_{s_m}) \mathds{1}_{(s_m,s_{m+1} \wedge  \min\{\nu_i\colon i\in\N, \, \nu_i >s_m\}]} (t) \mathds{1}_{\{N_t - N_{\underline{t}-\delta} = 0\}} \Big].
    \end{aligned}
\end{equation}
Since $t\in(s_m,s_{m+1} \wedge \min\{\nu_i >s_m\}]$, $s_m = \underline{t}$. Moreover, the sum in the above calculation has at most one summand, since the indicator functions are disjoint. Hence \eqref{LEst1eq13} can be rewritten as follows:
\begin{equation}\label{LEst1eq13a}
    \begin{aligned}
    &\E\Big[Y \sum_{n=0}^{N_T +M}|W_t -W_{\tau_n}|^q \mathds{1}_{S_\zeta} (Z_t^{(M)},Z^{(M)}_{\tau_n}) \mathds{1}_{(\tau_n,\tau_{n+1}]} (t) \mathds{1}_{\{N_t - N_{\underline{t}-\delta} = 0\}}\Big]\\
    &= \E\Big[Y |W_t -W_{\underline{t}}|^q \mathds{1}_{S_\zeta} (Z_t^{(M)},Z^{(M)}_{\underline{t}}) \mathds{1}_{(\underline{t},(\underline{t}+\delta) \wedge \min\{\nu_i\colon i\in\N, \, \nu_i >\underline{t}\}]} (t) \mathds{1}_{\{N_t - N_{\underline{t}-\delta} = 0\}} \Big].
    \end{aligned}
\end{equation}
Next we define
\begin{equation}\label{LEst1eq14}
    \begin{aligned}
    &\bar W_1 = \frac{W_t -W_{\underline{t}}}{\sqrt{t-\underline{t}}}, \quad \bar W_2 = \frac{W_{\underline{t}} -W_{\underline{t}-(t-\underline{t})}}{\sqrt{t-\underline{t}}}, \text{ and} \quad \bar W_3 =  \frac{W_{\underline{t}-(t-\underline{t})} -W_{\underline{t}-\delta}}{\sqrt{\delta - (t-\underline{t})}}.
    \end{aligned}
\end{equation}
It holds that $\bar W_1$, $\bar W_2$, and $\bar W_3$ are standard normally distributed and independent of each other. Further they are all independent of $\F_s$, since $s \leq \underline{t}-\delta$, and $\bar W_1$ and $\bar W_2$ are independent of $\F_{\underline{t}-(t-\underline{t})}$.

In the following we set $\bar c = \max\{c_{\widetilde\mu},c_{\widetilde\sigma}\} $, $\kappa = \bar c\big( 1+ \big|\zeta\big|\big) \big(1+L_{\widetilde\sigma}\big)$, and $M_0 \in\N\setminus\{1,2\}$ such that for all $M\geq M_0$, $\delta =\frac{T}{M}$ it holds that $\delta \leq 1$, $8\ln(T/\delta)\sqrt{\delta}\leq 1$, and $20\kappa\sqrt{\delta} \sqrt{\ln(T/\delta)} \leq \frac{1}{2}$.
Now let $M\geq M_0$. Then equation \eqref{LEst1eq13a} assures
\begin{equation}\label{LEst1eq15}
    \begin{aligned}
    &\E\Big[Y \sum_{n=0}^{N_T +M}|W_t -W_{\tau_n}|^q \mathds{1}_{S_\zeta} (Z_t^{(M)},Z^{(M)}_{\tau_n}) \mathds{1}_{(\tau_n,\tau_{n+1}]} (t) \mathds{1}_{\{N_t - N_{\underline{t}-\delta} = 0\}}\Big]\\
    &= \E\Big[Y |W_t -W_{\underline{t}}|^q \mathds{1}_{S_\zeta} (Z_t^{(M)},Z^{(M)}_{\underline{t}}) \mathds{1}_{(\underline{t},(\underline{t}+\delta) \wedge \min\{\nu_i\colon i\in\N, \, \nu_i >\underline{t}\}]} (t)\\
    &\quad\quad\quad\quad \quad\quad\quad\quad \quad\quad\quad \cdot\mathds{1}_{\{N_t - N_{\underline{t}-\delta} = 0\}} \mathds{1}_{\{\max_{i=1,2,3}|\bar W_i| > 2\sqrt{\ln(T/\delta)}\}}\Big]\\
    &\quad  + \E\Big[Y |W_t -W_{\underline{t}}|^q \mathds{1}_{S_\zeta} (Z_t^{(M)},Z^{(M)}_{\underline{t}}) \mathds{1}_{(\underline{t},(\underline{t}+\delta) \wedge \min\{\nu_i\colon i\in\N, \, \nu_i >\underline{t}\}]} (t)\\
    &\quad\quad\quad\quad \quad\quad\quad\quad \quad\quad\quad\quad \cdot \mathds{1}_{\{N_t - N_{\underline{t}-\delta} = 0\}} \mathds{1}_{\{\max_{i=1,2,3}|\bar W_i| \leq 2\sqrt{\ln(T/\delta)}\}}\Big].
    \end{aligned}
\end{equation}
For the first summand of \eqref{LEst1eq15} we have
\begin{equation}\label{LEst1eq16}
    \begin{aligned}
    &\E\Big[Y |W_t -W_{\underline{t}}|^q \mathds{1}_{S_\zeta} (Z_t^{(M)},Z^{(M)}_{\underline{t}}) \mathds{1}_{(\underline{t},(\underline{t}+\delta) \wedge \min\{\nu_i\colon i\in\N, \, \nu_i >\underline{t}\}]} (t) \mathds{1}_{\{N_t - N_{\underline{t}-\delta} = 0\}} \mathds{1}_{\big\{\max_{i=1,2,3}|\bar W_i| > 2\sqrt{\ln(T/\delta)}\big\}}\Big]\\
    &\leq \E\Big[Y \sup_{u\in[0,\delta]}|W_{u+\underline{t}} -W_{\underline{t}}|^q  \mathds{1}_{\big\{\max_{i=1,2,3}|\bar W_i| > 2\sqrt{\ln(T/\delta)}\big\}} \Big]\\
    &= \E[Y]\,\E\Big[\sup_{u\in[0,\delta]}|W_{u+\underline{t}} -W_{\underline{t}}|^q \mathds{1}_{\big\{\max_{i=1,2,3}|\bar W_i| > 2\sqrt{\ln(T/\delta)}\big\}} \Big]\\
    &\leq \E[Y]\,\E\Big[\sup_{u\in[0,\delta]}|W_{u+\underline{t}} -W_{\underline{t}}|^{2q} \Big]^\frac{1}{2} \E\Big[ \mathds{1}_{\big\{\max_{i=1,2,3}|\bar W_i| > 2\sqrt{\ln(T/\delta)}\big\}}\Big]^\frac{1}{2}\\
    &\leq \E[Y]\, c_{W_{2q}}^\frac{1}{2} \delta^\frac{q}{2} \,\P\Big(\max_{i=1,2,3}|\bar W_i| > 2\sqrt{\ln(T/\delta)}\Big)^\frac{1}{2}.
    \end{aligned}
\end{equation}
Next we observe
\begin{equation}\label{LEst1eq17}
    \begin{aligned}
    &\P\Big(\max_{i=1,2,3}|\bar W_i| > 2\sqrt{\ln(T/\delta)}\Big)
    \leq 3 \P(|\bar W_1| > 2\sqrt{\ln(T/\delta)})
    = \frac{3}{\sqrt{2\pi}} \frac{\delta^2}{\sqrt{\ln(T/\delta)}T^2}.
    \end{aligned}
\end{equation}
Plugging \eqref{LEst1eq17} into \eqref{LEst1eq16} and using $M\geq M_0$ we obtain
\begin{equation}\label{LEst1eq18}
    \begin{aligned}
    &\E\Big[Y |W_t -W_{\underline{t}}|^q \mathds{1}_{S_\zeta} (Z_t^{(M)},Z^{(M)}_{\underline{t}}) \mathds{1}_{(\underline{t},(\underline{t}+\delta) \wedge \min\{\nu_i\colon i\in\N, \, \nu_i >\underline{t}\}]} (t)\\
    &\quad\quad\quad\quad \quad\quad\quad\quad \quad \cdot \mathds{1}_{\{N_t - N_{\underline{t}-\delta} = 0\}} \mathds{1}_{\big\{\max_{i=1,2,3}|\bar W_i| > 2\sqrt{\ln(T/\delta)}\big\}}\Big]
    \leq c_{W_{2q}}^\frac{1}{2}  \,\Big( \frac{6}{\sqrt{2\pi}} \frac{1}{2T^2}\Big)^\frac{1}{2} \delta^{^\frac{q}{2}+1} \E[Y].
    \end{aligned}
\end{equation}
For the second summand of \eqref{LEst1eq15}, define the set 
\begin{equation}\label{LEst1eq19}
    \begin{aligned}
    A 
    = &\big\{(Z_t^{(M)},Z^{(M)}_{\underline{t}}) \in S_\zeta\big\}\cap \{N_t - N_{\underline{t}-\delta} = 0\}\cap \big\{\max_{i=1,2,3}|\bar W_i| \leq 2\sqrt{\ln(T/\delta)}\big\} \\
    &\cap\{t\in(\underline{t},(\underline{t}+\delta) \wedge \min\{\nu_i\colon i\in\N, \, \nu_i >\underline{t}\} ]\},
    \end{aligned}
\end{equation}
and choose $\omega\in A$. With this,
\begin{equation}\label{LEst1eq20}
    \begin{aligned}
    &\big| Z^{(M)}_{\underline{t}}- \zeta\big|
    \leq \big| Z^{(M)}_{\underline{t}}- Z^{(M)}_t\big|\\
    &\leq \big| \widetilde\mu \big(Z^{(M)}_{\underline{t}}\big)(t-\underline{t}) \big|
    +\big| \widetilde\sigma \big(Z^{(M)}_{\underline{t}}\big) (W_t- W_{\underline{t}}) \big|
    + \Big|\frac{1}{2} \widetilde\sigma\big(Z^{(M)}_{\underline{t}}\big) d_{\widetilde\sigma}\big(Z^{(M)}_{\underline{t}}\big) \big((W_t -W_{\underline{t}})^2 - (t-\underline{t})\big)\Big|\\
    &\leq  c_{\widetilde\mu} \big( 1+ \big|Z^{(M)}_{\underline{t}}\big|\big) \delta +c_{\widetilde\sigma} \big( 1+ \big|Z^{(M)}_{\underline{t}}\big|\big) \sqrt{\delta} \big| \bar W_1\big|
    + \frac{1}{2} \|d_{\widetilde\sigma}\|_\infty c_{\widetilde\sigma}\big( 1+ \big|Z^{(M)}_{\underline{t}}\big|\big)\delta \big(\big| \bar W_1\big|^2 +1\big) \\
    &\leq  \bar c \big( 1+ \big|Z^{(M)}_{\underline{t}}\big|\big) \Big(\delta +\sqrt{\delta}\big| \bar W_1\big| + \frac{1}{2} L_{\widetilde\sigma} \delta \big(\big| \bar W_1\big|^2 +1\big) \Big).
    \end{aligned}
\end{equation}
This, the fact that for all $a,b\in\R$,
\begin{equation}\label{LEst1eq21}
    \begin{aligned}
    (1+|a|) \leq (1+|a-b|)(1+|b|),
    \end{aligned}
\end{equation}
and
\begin{equation}\label{LEst1eq22}
    \begin{aligned}
    &\frac{1}{2} \sqrt{\delta} \big(\big| \bar W_1\big|^2 +1\big)
    \leq \frac{1}{2} \sqrt{\delta} \big(\big|2\sqrt{\ln(T/\delta)}\big|^2 +1\big)
    = \frac{4\ln(T/\delta)+1}{2}\sqrt{\delta}
    \leq \frac{5\ln(T/\delta)}{2}\sqrt{\delta}
    \leq 1
    \end{aligned}
\end{equation}
give
\begin{equation}\label{LEst1eq23}
    \begin{aligned}
    &\big| Z^{(M)}_{\underline{t}}- \zeta\big|
    \leq  \bar c \big( 1+ \big|Z^{(M)}_{\underline{t}}-\zeta\big|\big) \big( 1+ \big|\zeta\big|\big)\Big(\delta +\sqrt{\delta}\big| \bar W_1\big| +  L_{\widetilde\sigma} \sqrt{\delta} \Big)\\
    &\leq  \bar c \big( 1+ \big|Z^{(M)}_{\underline{t}}-\zeta\big|\big) \big( 1+ \big|\zeta\big|\big)\sqrt{\delta}\Big(1 +\big| \bar W_1\big| +  L_{\widetilde\sigma} \Big)
    = \kappa \big( 1+ \big|Z^{(M)}_{\underline{t}}-\zeta\big|\big) \big(1 +\big| \bar W_1\big|\big)\sqrt{\delta}.
    \end{aligned}
\end{equation}
Analogously to the derivation of \eqref{LEst1eq23} we get
\begin{equation}\label{LEst1eq26}
    \begin{aligned}
    &\big| Z^{(M)}_{\underline{t}-(t-\underline{t})}- Z^{(M)}_{\underline{t}-\delta}\big|
    \leq\kappa \big( 1+ \big|Z^{(M)}_{\underline{t}-\delta}-\zeta\big|\big) \big(1 +\big| \bar W_3\big|\big)\sqrt{\delta}
    \end{aligned}
\end{equation}
and
\begin{equation}\label{LEst1eq30}
    \begin{aligned}
    &\big| Z^{(M)}_{\underline{t}}-Z^{(M)}_{\underline{t}-(t-\underline{t})}\big|
    \leq \kappa \big( 1+ \big|Z^{(M)}_{\underline{t}-\delta}-\zeta\big|\big) \big(1 +\big| \bar W_2\big|\big)\sqrt{\delta}.
    \end{aligned}
\end{equation}
Further it holds that
\begin{equation}\label{LEst1eq31}
    \begin{aligned}
    &\big| Z^{(M)}_{\underline{t}-(t-\underline{t})} -\zeta \big| 
    \leq \big| Z^{(M)}_{\underline{t}}- Z^{(M)}_{\underline{t}-(t-\underline{t})} \big|
    +\big| Z^{(M)}_{\underline{t}} -\zeta \big|.
    \end{aligned}
\end{equation}
For $i\in\{1,2,3\}$ we estimate
\begin{equation}\label{LEst1eq32}
    \begin{aligned}
    &\kappa \sqrt{\delta} \big(1 + \big|\bar W_i\big| \big)
    \leq \kappa \sqrt{\delta} \big(1 + 2 \sqrt{\ln(T/\delta)}\big)
    \leq \kappa \sqrt{\delta} 3 \sqrt{\ln(T/\delta)}
    \leq \frac{1}{2}.
    \end{aligned}
\end{equation}
Combining \eqref{LEst1eq23} and \eqref{LEst1eq32} ensures 
\begin{equation}\label{LEst1eq33}
    \begin{aligned}
    &\big| Z^{(M)}_{\underline{t}}- \zeta\big|
    \leq \frac{\kappa \big(1 +\big| \bar W_1\big|\big)\sqrt{\delta}}{1- \kappa \big(1 +\big| \bar W_1\big|\big)\sqrt{\delta}}
    \leq 2 \kappa \big(1 +\big| \bar W_1\big|\big)\sqrt{\delta}.
    \end{aligned}
\end{equation}
By \eqref{LEst1eq26} and \eqref{LEst1eq32} we obtain that
\begin{equation}\label{LEst1eq34}
    \begin{aligned}
    & 1+ \big|Z^{(M)}_{\underline{t}-(t-\underline{t})} -\zeta \big| 
    \geq 1+ \big|Z^{(M)}_{\underline{t}-\delta} -\zeta \big| - \big|Z^{(M)}_{\underline{t}-(t-\underline{t})} -Z^{(M)}_{\underline{t}-\delta}\big|\\
    &\geq 1+ \big|Z^{(M)}_{\underline{t}-\delta} -\zeta \big| - \kappa \big( 1+ \big|Z^{(M)}_{\underline{t}-\delta}-\zeta\big|\big) \big(1 +\big| \bar W_3\big|\big)\sqrt{\delta}\\
    &= \big( 1+ \big|Z^{(M)}_{\underline{t}-\delta}-\zeta\big|\big)\big( 1 - \kappa \big(1 +\big| \bar W_3\big|\big)\sqrt{\delta}\big)
    \geq \frac{1}{2} \big( 1+ \big|Z^{(M)}_{\underline{t}-\delta}-\zeta\big|\big).
    \end{aligned}
\end{equation}
Plugging \eqref{LEst1eq34} into \eqref{LEst1eq30} we get
\begin{equation}\label{LEst1eq35}
    \begin{aligned}
    &\big| Z^{(M)}_{\underline{t}}-Z^{(M)}_{\underline{t}-(t-\underline{t})}\big|
    \leq 2 \kappa \big( 1+ \big|Z^{(M)}_{\underline{t}-(t-\underline{t})}-\zeta\big|\big) \big(1 +\big| \bar W_2\big|\big)\sqrt{\delta}.
    \end{aligned}
\end{equation}
Using \eqref{LEst1eq31}, \eqref{LEst1eq33}, and \eqref{LEst1eq35} we obtain
\begin{equation}\label{LEst1eq36}
    \begin{aligned}
    &\big| Z^{(M)}_{\underline{t}-(t-\underline{t})} -\zeta \big| 
    \leq \big| Z^{(M)}_{\underline{t}}- Z^{(M)}_{\underline{t}-(t-\underline{t})} \big|
    +\big| Z^{(M)}_{\underline{t}} -\zeta \big|\\
    &\leq 2 \kappa \big( 1+ \big|Z^{(M)}_{\underline{t}-(t-\underline{t})}-\zeta\big|\big) \big(1 +\big| \bar W_2\big|\big)\sqrt{\delta}
    + 2 \kappa \big(1 +\big| \bar W_1\big|\big)\sqrt{\delta}\\
    &\leq 4 \kappa \sqrt{\delta} \big( 1+ \big|Z^{(M)}_{\underline{t}-(t-\underline{t})}-\zeta\big|\big) \big(1 +\big| \bar W_1\big| +\big| \bar W_2\big|\big).
    \end{aligned}
\end{equation}
We further estimate
\begin{equation}\label{LEst1eq37}
    \begin{aligned}
    & 4 \kappa \sqrt{\delta} \big(1 +\big| \bar W_1\big| +\big| \bar W_2\big|\big)
    \leq 4 \kappa \sqrt{\delta} \big(1 + 2\sqrt{\ln(T/\delta)} + 2 \sqrt{\ln(T/\delta)}\big)
    \leq 20 \kappa \sqrt{\delta} \sqrt{\ln(T/\delta)} 
    \leq \frac{1}{2}.
    \end{aligned}
\end{equation}
Combining \eqref{LEst1eq36} and \eqref{LEst1eq37} we obtain
\begin{equation}\label{LEst1eq38}
    \begin{aligned}
    &\big| Z^{(M)}_{\underline{t}-(t-\underline{t})} -\zeta \big| 
    \leq \frac{4 \kappa \sqrt{\delta} \big(1 +\big| \bar W_1\big| +\big| \bar W_2\big|\big)}{1-4 \kappa \sqrt{\delta} \big(1 +\big| \bar W_1\big| +\big| \bar W_2\big|\big)}
    &\leq 8 \kappa \sqrt{\delta} \big(1 +\big| \bar W_1\big| +\big| \bar W_2\big|\big).
    \end{aligned}
\end{equation}
Hence we have proven that 
\begin{equation}\label{LEst1eq39}
    \begin{aligned}
    A \subset \Big\{\big| Z^{(M)}_{\underline{t}-(t-\underline{t})} -\zeta \big| 
    \leq 8 \kappa \sqrt{\delta} \big(1 +\big| \bar W_1\big| +\big| \bar W_2\big|\big)\Big\}.
    \end{aligned}
\end{equation}
Next we estimate the second summand of \eqref{LEst1eq15}. We use \eqref{LEst1eq39}, the fact that $Y$ is $\F_{\underline{t}-\delta}\subset \F_{\underline{t}-(t-\underline{t})}$ measurable, \cite[p.~33]{Grikhman2004}, and the fact that $\bar W_1$ and $\bar W_2$ are independent of $\F_{\underline{t}-(t-\underline{t})}$ to obtain
\begin{equation}\label{LEst1eq40}
    \begin{aligned}
    &\E\Big[Y |W_t -W_{\underline{t}}|^q \mathds{1}_{S_\zeta} (Z_t^{(M)},Z^{(M)}_{\underline{t}}) \mathds{1}_{(\underline{t},(\underline{t}+\delta) \wedge \min\{\nu_i\colon i\in\N, \, \nu_i >\underline{t}\}]} (t)\\
    &\quad\quad\quad\quad \quad\quad\quad\quad \quad\quad\cdot \mathds{1}_{\{N_t - N_{\underline{t}-\delta} = 0\}} \mathds{1}_{\{\max_{i=1,2,3}|\bar W_i| \leq 2\sqrt{\ln(T/\delta)}\}}\Big]\\
    &\leq \E\Big[Y |W_t -W_{\underline{t}}|^q \mathds{1}_{\{| Z^{(M)}_{\underline{t}-(t-\underline{t})} -\zeta | 
    \leq 8 \kappa \sqrt{\delta} (1 +| \bar W_1| + \bar W_2|)\}}\Big]\\
    &\leq \delta^\frac{q}{2} \E\Big[Y \big|\bar W_1\big|^q \mathds{1}_{\{| Z^{(M)}_{\underline{t}-(t-\underline{t})} -\zeta | 
    \leq 8 \kappa \sqrt{\delta} (1 +| \bar W_1| + \bar W_2|)\}}\Big]\\
    &=\delta^\frac{q}{2} \E\Big[Y \E\Big[ \big|\bar W_1\big|^q \mathds{1}_{\{| Z^{(M)}_{\underline{t}-(t-\underline{t})} -\zeta | 
    \leq 8 \kappa \sqrt{\delta} (1 +| \bar W_1| + \bar W_2|)\}}\Big| \F_{\underline{t}-(t-\underline{t})}\Big]\Big]\\
    &= \delta^\frac{q}{2} \E\Big[Y \E\Big[ \big|\bar W_1\big|^q \mathds{1}_{\{| z -\zeta | 
    \leq 8 \kappa \sqrt{\delta} (1 +| \bar W_1| + \bar W_2|)\}}\Big| \F_{\underline{t}-(t-\underline{t})}\Big]\Big|_{Z^{(M)}_{\underline{t}-(t-\underline{t})} = z}\Big]\\
    &= \delta^\frac{q}{2} \E\Big[Y \E\Big[ \big|\bar W_1\big|^q \mathds{1}_{\{| z -\zeta | 
    \leq 8 \kappa \sqrt{\delta} (1 +| \bar W_1| + \bar W_2|)\}}\Big]\Big|_{Z^{(M)}_{\underline{t}-(t-\underline{t})} = z}\Big].
    \end{aligned}
\end{equation}
Since $\bar W_1$ and $\bar W_2$ are independent, standard normally distributed random variables,
\begin{equation}\label{LEst1eq42}
    \begin{aligned}
    &\E\Big[ \big|\bar W_1\big|^q \mathds{1}_{\{| z -\zeta | 
    \leq 8 \kappa \sqrt{\delta} (1 +| \bar W_1| + \bar W_2|)\}}\Big]\\
    &= \frac{2}{\pi} \int_{0}^\infty \int_{0}^\infty \big|w_1\big|^q \mathds{1}_{\{| z -\zeta | \leq 8 \kappa \sqrt{\delta} (1 +|w_1| + |w_2|)\}}  \exp\Big(-\frac{1}{2}(w_1^2 +w_2^2) \Big) \diff w_1 \diff w_2.
    \end{aligned}
\end{equation}
Plugging \eqref{LEst1eq42} into \eqref{LEst1eq40} we obtain
\begin{equation}\label{LEst1eq43}
    \begin{aligned}
    &\E\Big[Y |W_t -W_{\underline{t}}|^q \mathds{1}_{S_\zeta} (Z_t^{(M)},Z^{(M)}_{\underline{t}}) \mathds{1}_{(\underline{t},(\underline{t}+\delta) \wedge \min\{\nu_i\colon i\in\N, \, \nu_i >\underline{t}\}]} (t)\\
    &\quad\quad\quad\quad \quad\quad\quad\quad \quad\quad\cdot \mathds{1}_{\{N_t - N_{\underline{t}-\delta} = 0\}} \mathds{1}_{\{\max_{i=1,2,3}|\bar W_i| \leq 2\sqrt{\ln(T/\delta)}\}}\Big]\\
    &\leq \delta^\frac{q}{2} \frac{2^{\frac{q}{2}+1}}{\pi} \int_{\R^2} \E\Big[Y \Big(\frac{|w_1 +w_2|}{\sqrt{2}}\Big)^q \mathds{1}_{\{| Z^{(M)}_{\underline{t}-(t-\underline{t})} -\zeta | \leq 8\sqrt{2} \kappa \sqrt{\delta} (1 +\frac{|w_1 + w_2|}{\sqrt{2}})\}}  e^{-\frac{1}{2}(w_1^2 +w_2^2)} \Big]\diff (w_1,w_2).
    \end{aligned}
\end{equation}
Furthermore,
\begin{equation}\label{LEst1eq44}
    \begin{aligned}
    &\E\Big[Y |W_t -W_{\underline{t}}|^q \mathds{1}_{S_\zeta} (Z_t^{(M)},Z^{(M)}_{\underline{t}}) \mathds{1}_{(\underline{t},(\underline{t}+\delta) \wedge \min\{\nu_i\colon i\in\N, \, \nu_i >\underline{t}\}]} (t)\\
    &\quad\quad\quad\quad \quad\quad\quad\quad \quad\quad\cdot \mathds{1}_{\{N_t - N_{\underline{t}-\delta} = 0\}} \mathds{1}_{\{\max_{i=1,2,3}|\bar W_i| \leq 2\sqrt{\ln(T/\delta)}\}}\Big]\\
    &\leq \delta^\frac{q}{2} 2^{\frac{q}{2}}\int_\R \E\Big[Y |w|^q \mathds{1}_{\{| Z^{(M)}_{\underline{t}-(t-\underline{t})} -\zeta | \leq 8\sqrt{2} \kappa \sqrt{\delta} (1 +|w|)\}} \Big] \frac{1}{\sqrt{2\pi}} e^{-\frac{w^2}{2}}\diff w.
    \end{aligned}
\end{equation}
Combining \eqref{LEst1eq1}, \eqref{LEst1eq10}, \eqref{LEst1eq15}, \eqref{LEst1eq18}, and \eqref{LEst1eq44} we obtain that there exists a constant $\widetilde c\in(0,\infty)$  such that for all $M\geq M_0$,
\begin{equation}\label{LEst1eq45}
    \begin{aligned}
    &\E\Big[Y\sum_{n=0}^{N_T +M} |W_t -W_{\tau_n}|^q \mathds{1}_{S_\zeta} (Z_t^{(M)},Z^{(M)}_{\tau_n}) \mathds{1}_{(\tau_n,\tau_{n+1}]} (t) \Big]\\
    &\leq \widetilde c \Big( \delta^{\frac{q+1}{2}} \E[Y] + \delta^{\frac{q}{2}} \int_\R |z|^q \E\Big[ Y \mathds{1}_{\{|Z^{(M)}_{\underline{t} -(t-\underline{t})} -\zeta|\leq \widetilde c (1+|z|)\sqrt{\delta}} \Big] e^{-\frac{z^2}{2}} \diff z\Big).
    \end{aligned}
\end{equation}
For all $M<M_0$ similar calculations as in \eqref{LEst1eq6} and Lemma \ref{MW} yield
\begin{equation}\label{LEst1eq46}
    \begin{aligned}
    &\E\Big[Y\sum_{n=0}^{N_T +M} |W_t -W_{\tau_n}|^q \mathds{1}_{S_\zeta} (Z_t^{(M)},Z^{(M)}_{\tau_n}) \mathds{1}_{(\tau_n,\tau_{n+1}]} (t) \Big]
    \leq \E\Big[Y\sum_{n=0}^{N_T +M} |W_t -W_{\tau_n}|^q \mathds{1}_{(\tau_n,\tau_{n+1}]} (t) \Big]\\
    & \leq \E[Y]\,\E\Big[\sum_{n=0}^{N_T +M} \sup_{u\in[0,\delta]}|W_{u+ \tau_n } -W_{\tau_n}|^q  \mathds{1}_{(\tau_n,\tau_{n+1}]} (t) \Big]
    \leq \E[Y] c_{W_q} \delta ^\frac{q}{2}\, \E\Big[\sum_{n=0}^{N_T +M} \mathds{1}_{(\tau_n,\tau_{n+1}]} (t) \Big]\\
    & = \E[Y] c_{W_q} \delta^{-\frac{1}{2}} \delta ^{\frac{q+1}{2}}
    \leq \E[Y] c_{W_q} \Big(\frac{T}{M_0}\Big)^{-\frac{1}{2}} \delta ^{\frac{q+1}{2}}.
    \end{aligned}
\end{equation}
This and \eqref{LEst1eq45} prove the claim.
\end{proof}

\begin{lemma}\label{LemmaEst2}
Let Assumption \ref{assZ} hold. Let $q\in\N$, $\zeta\in\R$ with $\widetilde\sigma (\zeta) \neq 0$, and
\begin{equation}
    \begin{aligned}
    S_{\zeta} = \{(x,y)\in\R^2 : (x-\zeta)(y-\zeta) \leq 0\}.
    \end{aligned}
\end{equation}
Then there exist constants $c_7, c_8\in(0,\infty)$ such that for all $M\in\N$, $\delta = \frac{T}{M}$, all $m\in\{0,\ldots,M-2\}$, all $s \in [s_m,s_{m+1})$, and all $\F_s$-measurable, non-negative, real-valued random variables $Y$,
\begin{equation}
    \begin{aligned}
    &\int_{s_m +2\delta}^T \E\Big[Y\sum_{n=0}^{\infty} |W_t -W_{\tau_n}|^q \mathds{1}_{S_\zeta} (Z_t^{(M)},Z^{(M)}_{\tau_n}) \mathds{1}_{(\tau_n,\tau_{n+1}]} (t) \Big] \diff t\\
    &\quad\quad\quad\quad \quad\quad\quad\quad \quad\quad\quad\quad \quad\quad\quad\quad \leq c_7 \delta^{\frac{q+1}{2}} \Big(\E[Y] + \E\Big[ Y \big|Z^{(M)}_{s_m+\delta} - \zeta\big|^2\Big]\Big)
    \end{aligned}
\end{equation}
and
\begin{equation}
    \begin{aligned}
    &\int_{s_m +\delta}^T \E\Big[Y\sum_{n=0}^{\infty} |W_t -W_{\tau_n}|^q \mathds{1}_{S_\zeta} (Z_t^{(M)},Z^{(M)}_{\tau_n}) \mathds{1}_{(\tau_n,\tau_{n+1}]} (t) \Big(Z^{(M)}_{\underline{t}+\delta}-\zeta \Big)^2 \Big]\diff t\\
    &\quad\quad\quad\quad \quad\quad\quad\quad \quad\quad\quad\quad \quad\quad\quad\quad \leq c_8 \delta^{\frac{q+1}{2}} \Big(\E[Y] + \E\Big[ Y \big|Z^{(M)}_{s_m+\delta} - \zeta\big|^2\Big]\Big).
    \end{aligned}
\end{equation}
\end{lemma}

\begin{proof}
We start proving the first inequality. Lemma \ref{LemmaEst1} yields that there exists $c_6\in(0,\infty)$ such that
\begin{equation}
    \begin{aligned}\label{pLEst2eq1}
    &\int_{s_m +2\delta}^T \E\Big[Y\sum_{n=0}^{\infty} |W_t -W_{\tau_n}|^q \mathds{1}_{S_\zeta} (Z_t^{(M)},Z^{(M)}_{\tau_n}) \mathds{1}_{(\tau_n,\tau_{n+1}]} (t) \Big]\diff t\\
    &\leq T c_6 \delta^{\frac{q+1}{2}} \E[Y]
    + c_6 \delta^{\frac{q}{2}} \int_\R |z|^q\int_{s_m +2\delta}^T \E\Big[ Y \mathds{1}_{\{|Z^{(M)}_{\underline{t} -(t-\underline{t})} -\zeta|\leq c_6 (1+|z|)\sqrt{\delta}\}} \Big] \diff t\, e^{-\frac{z^2}{2}}\diff z.
    \end{aligned}
\end{equation}
Substitution $u= s_k-(t-s_k)$ gives
\begin{equation}
    \begin{aligned}\label{pLEst2eq2}
    &\int_{s_m +2\delta}^T \E\Big[ Y \mathds{1}_{\{|Z^{(M)}_{\underline{t} -(t-\underline{t})} -\zeta|\leq c_6 (1+|z|)\sqrt{\delta}\}} \Big] \diff t
    = \sum_{k=m+2}^{M-1} \int_{s_k}^{s_{k+1}} \E\Big[ Y \mathds{1}_{\{|Z^{(M)}_{s_k -(t-s_k)} -\zeta|\leq c_6 (1+|z|)\sqrt{\delta}\}} \Big] \diff t\\
    &= \sum_{k=m+2}^{M-1} \int_{s_{k-1}}^{s_k} \E\Big[ Y \mathds{1}_{\{|Z^{(M)}_{u} -\zeta|\leq c_6 (1+|z|)\sqrt{\delta}\}} \Big] \diff u
    = \int_{s_{m}+\delta}^{T-\delta} \E\Big[ Y \mathds{1}_{\{|Z^{(M)}_{u} -\zeta|\leq c_6 (1+|z|)\sqrt{\delta}\}} \Big] \diff u\\
    &= \E\Big[ Y \int_{s_{m}+\delta}^{T-\delta} \E\Big[ \mathds{1}_{\{|Z^{(M)}_{u} -\zeta|\leq c_6 (1+|z|)\sqrt{\delta}\}} \Big| \F_{s_m+\delta}\Big] \diff u \Big].
    \end{aligned}
\end{equation}
By Lemma \ref{MarkovProperty} we get
\begin{equation}
    \begin{aligned}\label{pLEst2eq3}
    &\int_{s_m +2\delta}^T \E\Big[ Y \mathds{1}_{\{|Z^{(M)}_{\underline{t} -(t-\underline{t})} -\zeta|\leq c_6 (1+|z|)\sqrt{\delta}\}} \Big] \diff t
    = \E\Big[ Y \int_{s_{m}+\delta}^{T-\delta} \E\Big[ \mathds{1}_{\{|Z^{(M)}_{u} -\zeta|\leq c_6 (1+|z|)\sqrt{\delta}\}} \Big| Z^{(M)}_{s_m+\delta}\Big] \diff u \Big]\\
    &= \E\Big[ Y \int_{0}^{T-s_{m}-2\delta} \E\Big[ \mathds{1}_{\{|Z^{(M)}_{u+s_{m}+\delta} -\zeta|\leq c_6 (1+|z|)\sqrt{\delta}\}} \Big| Z^{(M)}_{s_m+\delta}=y\Big]\Big|_{y= Z^{(M)}_{s_m+\delta}} \diff u \Big]\\
    &= \E\Big[ Y \int_{0}^{T-s_{m}-2\delta} \E\Big[ \mathds{1}_{\{|Z^{(M),y}_{u} -\zeta|\leq c_6 (1+|z|)\sqrt{\delta}\}} \Big]\Big|_{y= Z_{s_m+\delta}^{(M)}} \diff u \Big]\\
    &= \E\Big[ Y \int_{0}^{T-s_{m}-2\delta} \E\Big[ \mathds{1}_{\{|Z^{(M),y}_{u} -\zeta|\leq c_6 (1+|z|)\sqrt{\delta}\}} \Big] \diff u \Big|_{y= Z_{s_m+\delta}^{(M)}} \Big].
    \end{aligned}
\end{equation}
Lemma \ref{LOTE} assures
\begin{equation}
    \begin{aligned}\label{pLEst2eq4}
    &\int_{s_m +2\delta}^T \E\Big[ Y \mathds{1}_{\{|Z^{(M)}_{\underline{t} -(t-\underline{t})} -\zeta|\leq c_6 (1+|z|)\sqrt{\delta}\}} \Big] \diff t
    \leq c_5 ( c_6 (1+|z|)+1)\sqrt{\delta}\, \E\Big[ Y (1+\big|Z_{s_m+\delta}^{(M)}\big|^2)\Big].\\
    \end{aligned}
\end{equation}
Plugging \eqref{pLEst2eq4} into \eqref{pLEst2eq1}, yields
\begin{equation}
    \begin{aligned}\label{pLEst2eq5}
    &\int_{s_m +2\delta}^T \E\Big[Y\sum_{n=0}^{\infty} |W_t -W_{\tau_n}|^q \mathds{1}_{S_\zeta} (Z_t^{(M)},Z^{(M)}_{\tau_n}) \mathds{1}_{(\tau_n,\tau_{n+1}]} (t) \Big]\diff t\\
    &\leq T c_6 \delta^{\frac{q+1}{2}} \E[Y]
    + c_6 c_5 \delta^{\frac{q+1}{2}} \, \E\Big[ Y (1+\big|Z_{s_m+\delta}^{(M)}\big|^2)\Big] \int_\R |z|^q ( c_6 (1+|z|)+1)\,  e^{-\frac{z^2}{2}} \diff z.
    \end{aligned}
\end{equation}
We observe that $\int_\R |z|^q ( c_6 (1+|z|)+1)\,  e^{-\frac{z^2}{2}}\diff z =: \widetilde c_1 <\infty$.
Since for all $a,b\in\R$, $1+a^2 \leq 2 (1+(a-b)^2)(1+b^2)$,
\begin{equation}
    \begin{aligned}\label{pLEst2eq6}
    &\int_{s_m +2\delta}^T \E\Big[Y\sum_{n=0}^{\infty} |W_t -W_{\tau_n}|^q \mathds{1}_{S_\zeta} (Z_t^{(M)},Z^{(M)}_{\tau_n}) \mathds{1}_{(\tau_n,\tau_{n+1}]} (t) \Big]\diff t\\
    &\leq T c_6 \delta^{\frac{q+1}{2}} \E[Y]
    + 2 c_6 c_5 \widetilde c_1 \delta^{\frac{q+1}{2}} \, \E\Big[ Y (1+\big|Z_{s_m+\delta}^{(M)}-\zeta\big|^2)(1+|\zeta|^2)\Big] \\
    &\leq T c_6 \delta^{\frac{q+1}{2}} \E[Y]
    + 2 c_6  c_5 \widetilde c_1 \delta^{\frac{q+1}{2}} (1+|\zeta|^2) \Big(\E[ Y] + \E\Big[ Y \big|Z_{s_m+\delta}^{(M)}-\zeta\big|^2\Big]\Big).
    \end{aligned}
\end{equation}
This proves the first inequality. 

For $\omega\in\{\bar\omega\in \Omega\colon \sum_{n=0}^{\infty} \mathds{1}_{S_{\zeta}}(Z_t^{(M)}(\bar\omega),Z^{(M)}_{\tau_n}(\bar\omega))\mathds{1}_{(\tau_n,\tau_{n+1}]}(t) = 1$\} we have
\begin{equation}
    \begin{aligned}\label{pLEst2eq7}
    &\big|Z_{\underline{t}+\delta}^{(M)}-\zeta\big| 
    \leq \big|Z_{\underline{t}+\delta}^{(M)}-Z_{t}^{(M)}\big| +\big|Z_{t}^{(M)}-\zeta\big| 
    \leq \big|Z_{\underline{t}+\delta}^{(M)}-Z_{t}^{(M)}\big| +\big|Z_{t}^{(M)}-Z_{\tau_n}^{(M)}\big|.
    \end{aligned}
\end{equation}
Since $Y$ is $\F_s$-measurable, it is $\F_{s_{m+1}}$-measurable. Hence, using \eqref{pLEst2eq7} and Hölder's inequality for the conditional expectation we get 
\begin{equation}
    \begin{aligned}\label{pLEst2eq8}
    &\E\Big[Y\sum_{n=0}^{\infty} |W_t -W_{\tau_n}|^q \mathds{1}_{S_\zeta} (Z_t^{(M)},Z^{(M)}_{\tau_n}) \mathds{1}_{(\tau_n,\tau_{n+1}]} (t) \Big(Z^{(M)}_{\underline{t}+\delta} - \zeta\Big)^2 \Big]\\
    &\leq \E\Big[Y\sum_{n=0}^{\infty} |W_t -W_{\tau_n}|^q \mathds{1}_{S_\zeta} (Z_t^{(M)},Z^{(M)}_{\tau_n}) \mathds{1}_{(\tau_n,\tau_{n+1}]} (t) \Big(\big|Z_{\underline{t}+\delta }^{(M)}-Z_{t}^{(M)}\big| +\big|Z_{t}^{(M)}-Z_{\tau_n}^{(M)}\big|\Big)^2 \Big]\\
    &\leq \E\Big[Y\sum_{n=0}^{\infty} |W_t -W_{\tau_n}|^q \mathds{1}_{(\tau_n,\tau_{n+1}]} (t)\\
    &\quad\quad\quad\quad \quad\quad \times\Big(\big|Z_{\underline{t}+\delta}^{(M)}-Z_{\underline{t}}^{(M)}\big|+ \big|Z_{\underline{t}}^{(M)}-Z_{t}^{(M)}\big| +\big|Z_{t}^{(M)}-Z_{\underline{t}}^{(M)}\big|+\big|Z_{\underline{t}}^{(M)}-Z_{\tau_n}^{(M)}\big|\Big)^2 \Big]\\ 
    &= \E\Big[\E\Big[Y\sum_{n=0}^{\infty} |W_t -W_{\tau_n}|^q \mathds{1}_{(\tau_n,\tau_{n+1}]} (t) \Big(4 \sup_{u\in[\underline{t},\underline{t}+\delta]} \big|Z_{u}^{(M)}-Z_{\underline{t}}^{(M)}\big|\Big)^2 \Big|\F_{s_{m+1}}\Big]\Big]\\
    &\leq \E\Big[Y\,\E\Big[\sum_{n=0}^{\infty} |W_t -W_{\tau_n}|^{2q} \mathds{1}_{(\tau_n,\tau_{n+1}]} (t)\Big|\F_{s_{m+1}}\Big]^\frac{1}{2} \E\Big[ \Big(4 \sup_{u\in[\underline{t},\underline{t}+\delta]} \big|Z_{u}^{(M)}-Z_{\underline{t}}^{(M)}\big|\Big)^4 \Big|\F_{s_{m+1}}\Big]^\frac{1}{2}\Big].
    \end{aligned}
\end{equation}
For the first conditional expectation in \eqref{pLEst2eq8}, recalling that $t\geq s_m+\delta  = s_{m+1}$, we get
\begin{equation}
    \begin{aligned}\label{pLEst2eq9}
    &\E\Big[\sum_{n=0}^{\infty} |W_t -W_{\tau_n}|^{2q} \mathds{1}_{(\tau_n,\tau_{n+1}]} (t)\Big|\F_{s_{m+1}}\Big]
    =\E\Big[\sum_{n=N_{m+1}+m+1}^{N_T +M} |W_t -W_{\tau_n}|^{2q} \mathds{1}_{(\tau_n,\tau_{n+1}]} (t)\Big|\F_{s_{m+1}}\Big]\\
    &=\E\Big[\sum_{n=N_{m+1}+m+1}^{N_T +M} |W_t -W_{\tau_n}|^{2q} \mathds{1}_{(\tau_n,\tau_{n+1}]} (t)\Big]
    =\E\Big[\sum_{n=0}^{N_T +M} |W_t -W_{\tau_n}|^{2q} \mathds{1}_{(\tau_n,\tau_{n+1}]} (t)\Big]
    \leq c_{W_{2q}} \delta^q.
    \end{aligned}
\end{equation}
For the second conditional expectation in \eqref{pLEst2eq8}, Lemma \ref{MarkovProperty} and Lemma \ref{ME} equation \eqref{LME4} assure
\begin{equation}
    \begin{aligned}\label{pLEst2eq18}
    &\E\Big[ \Big(4 \sup_{u\in[\underline{t},\underline{t}+\delta]} \big|Z_{u}^{(M)}-Z_{\underline{t}}^{(M)}\big|\Big)^4 \Big|\F_{s_{m+1}}\Big]
    =4^4\, \E\Big[ \sup_{u\in[\underline{t},\underline{t}+\delta]} \big|Z_{u}^{(M)}-Z_{\underline{t}}^{(M)}\big|^4 \Big| Z^{(M)}_{s_{m+1}}\Big]\\
    &=4^4\, \E\Big[ \sup_{u\in[\underline{t},\underline{t}+\delta]} \big|Z_{u}^{(M)}-Z_{\underline{t}}^{(M)}\big|^4 \Big|Z^{(M)}_{s_{m+1}} =z\Big]\Big|_{z=Z^{(M)}_{s_{m+1}}}\\
    & =4^4\, \E\Big[ \sup_{u\in[\underline{t},\underline{t}+\delta]} \big|Z_{u-s_{m+1}}^{(M),z}-Z_{\underline{t}-s_{m+1}}^{(M),z}\big|^4 \Big]\Big|_{z=Z^{(M)}_{s_{m+1}}}
    \leq 4^4 c_3 \Big(1+ \big|Z^{(M)}_{s_{m+1}}\big|^4\Big)\delta
    \end{aligned}
\end{equation}
Then there exists a constant $\widetilde c_2\in(0,\infty)$ such that 
\begin{equation}
    \begin{aligned}\label{pLEst2eq20}
    &\E\Big[ \Big(4 \sup_{u\in[\underline{t},\underline{t}+\delta]} \big|Z_{u}^{(M)}-Z_{\underline{t}}^{(M)}\big|\Big)^4 \Big|\F_{s_{m+1}}\Big] \leq \widetilde c_2 \big(1+\big|Z^{(M)}_{s_{m+1}}-\zeta\big|^4\big)\delta.
    \end{aligned}
\end{equation}
Plugging \eqref{pLEst2eq9} and \eqref{pLEst2eq20} into \eqref{pLEst2eq8} we obtain
\begin{equation}
    \begin{aligned}\label{pLEst2eq21}
    &\E\Big[Y\sum_{n=0}^{\infty} |W_t -W_{\tau_n}|^q \mathds{1}_{S_\zeta} (Z_t^{(M)},Z^{(M)}_{\tau_n}) \mathds{1}_{(\tau_n,\tau_{n+1}]} (t) \Big(Z^{(M)}_{\underline{t}+\frac{T}{M}} - \zeta\Big)^2 \Big]\\
    &\leq \E\Big[Y c_{W_{2q}}^\frac{1}{2} \delta^\frac{q}{2} \Big(\widetilde c_2 \big(1+\big|Z^{(M)}_{s_{m+1}}-\zeta\big|^4\big)\delta\Big)^\frac{1}{2}\Big]
    \leq c_{W_{2q}}^\frac{1}{2}\widetilde c_2^\frac{1}{2} \delta^\frac{q+1}{2} \, \E\Big[Y\big(1+\big|Z^{(M)}_{s_{m+1}}-\zeta\big|^2\big)\Big].
    \end{aligned}
\end{equation}
Hence, we obtain
\begin{equation}
    \begin{aligned}
    &\int_{s_m +\delta}^T \E\Big[Y\sum_{n=0}^{\infty} |W_t -W_{\tau_n}|^q \mathds{1}_{S_\zeta} (Z_t^{(M)},Z^{(M)}_{\tau_n}) \mathds{1}_{(\tau_n,\tau_{n+1}]} (t) \Big(Z^{(M)}_{\underline{t}+\delta}-\zeta \Big)^2 \Big]\diff t\\
    &\leq T c_{W_{2q}}^\frac{1}{2}\widetilde c^\frac{1}{2} \delta^\frac{q+1}{2}  \Big(\E[Y]+\E\Big[Y\big|Z^{(M)}_{s_{m+1}}-\zeta\big|^2\big)\Big]\Big).
    \end{aligned}
\end{equation}
\end{proof}

\begin{lemma}\label{IndW}
Let Assumption \ref{assZ} hold. Let $\zeta\in\R$ such that $\widetilde\sigma (\zeta)\neq 0$ and
\begin{equation}
    \begin{aligned}
    S_{\zeta} = \{(x,y)\in\R^2 : (x-\zeta)(y-\zeta) \leq 0\}.
    \end{aligned}
\end{equation}
Then for all $p,q\in\N$ there exists a constant $c_9\in(0,\infty)$ such that for all $M\in\N$, $\delta = \frac{T}{M}$ it holds that
\begin{equation}
    \begin{aligned}
    \E\Big[\Big|\int_0^T  \sum_{n=0}^{\infty} \big|W_t -W_{\tau_n}\big|^q \mathds{1}_{S_\zeta}(Z^{(M)}_t, Z^{(M)}_{\tau_n}) \mathds{1}_{[\tau_n, \tau_{n+1})}(t) \diff t\Big|^p\Big] \leq c_9\, \delta^{(q+1)\frac{p}{2}}.
    \end{aligned}
\end{equation}
\end{lemma}

The proof of this lemma works analogously to the proof of \cite[equation (68)]{muellergronbach2019b}. The only changes are the notation due to the jump-adapted time grid and that we use Lemma \ref{LemmaEst2} instead of \cite[Lemma 8]{muellergronbach2019b}.

\begin{proposition}\label{Prop}
Let Assumption \ref{assZ} hold. Let $\zeta\in\R$ such that $\widetilde\sigma (\zeta)\neq 0$ and
\begin{equation}
    \begin{aligned}
    S_{\zeta} = \{(x,y)\in\R^2 : (x-\zeta)(y-\zeta) \leq 0\}.
    \end{aligned}
\end{equation}
Then for all $p,q\in\N$ there exists a constant $c_{10}\in(0,\infty)$ such that for all $M\in\N$, $\delta = \frac{T}{M}$ it holds that
\begin{equation}
    \begin{aligned}
    \E\Big[\Big|\int_0^T  \sum_{n=0}^{\infty} \big|Z^{(M)}_t -Z^{(M)}_{\tau_n}\big|^q \mathds{1}_{S_\zeta}(Z^{(M)}_t, Z^{(M)}_{\tau_n}) \mathds{1}_{[\tau_n, \tau_{n+1})}(t) \diff t\Big|^p\Big]^{\frac{1}{p}} \leq c_{10}\, \delta^{\frac{q+1}{2}}.
    \end{aligned}
\end{equation}
\end{proposition}

\begin{proof}
We start calculating that there exist constants $\widetilde c_1, \widetilde c_2 \in(0,\infty)$ such that 
\begin{equation}
    \begin{aligned}\label{eqP1}
    &\sum_{n=0}^{\infty} \big|Z^{(M)}_t -Z^{(M)}_{\tau_n}\big| \mathds{1}_{S_\zeta}(Z^{(M)}_t, Z^{(M)}_{\tau_n}) \mathds{1}_{[\tau_n, \tau_{n+1})}(t)\\
    &\leq \sum_{n=0}^{\infty} \widetilde c_1 \big(1+ \big|Z^{(M)}_{\tau_n}\big|\big) \big(\delta + |W_t -W_{\tau_n}| + |W_t - W_{\tau_n}|^2\big) \mathds{1}_{S_\zeta}(Z^{(M)}_t, Z^{(M)}_{\tau_n}) \mathds{1}_{[\tau_n, \tau_{n+1})}(t)\\
    &\leq \sum_{n=0}^{\infty} \widetilde c_1 \big(1+ \big|Z^{(M)}_{\tau_n}\big|\big) \big(\delta + |W_t - W_{\tau_n}|^2\big) \mathds{1}_{[\tau_n, \tau_{n+1})}(t)\\
    &\quad + \sum_{n=0}^{\infty} \widetilde c_2 |W_t -W_{\tau_n}| \mathds{1}_{S_\zeta}(Z^{(M)}_t, Z^{(M)}_{\tau_n}) \mathds{1}_{[\tau_n, \tau_{n+1})}(t)\\
    &\quad + \sum_{n=0}^{\infty} \widetilde c_2 \big|Z^{(M)}_{\tau_n}- Z^{(M)}_{t} \big| |W_t -W_{\tau_n}| \mathds{1}_{[\tau_n, \tau_{n+1})}(t).
    \end{aligned}
\end{equation}
Using this and Minkowski's inequality we obtain that there exists a constant $\widetilde c_3\in(0,\infty)$ such that
\begin{equation}
    \begin{aligned}\label{P2}
    &\E\Big[\Big|\int_0^T  \sum_{n=0}^{\infty} \big|Z^{(M)}_t -Z^{(M)}_{\tau_n}\big|^q \mathds{1}_{S_\zeta}(Z^{(M)}_t, Z^{(M)}_{\tau_n}) \mathds{1}_{[\tau_n, \tau_{n+1})}(t) \diff t\Big|^p\Big]^{\frac{1}{p}} \\
    &\leq  \E\Big[\Big|\int_0^T  \Big|\sum_{n=0}^{\infty} \widetilde c_1 \big(1+ \big|Z^{(M)}_{\tau_n}\big|\big) \big(\delta + |W_t - W_{\tau_n}|^2\big) \mathds{1}_{[\tau_n, \tau_{n+1})}(t)\\
    &\quad\quad\quad\quad \quad\quad + \sum_{n=0}^{\infty} \widetilde c_2 |W_t -W_{\tau_n}| \mathds{1}_{S_\zeta}(Z^{(M)}_t, Z^{(M)}_{\tau_n}) \mathds{1}_{[\tau_n, \tau_{n+1})}(t)\\
    &\quad\quad\quad\quad \quad\quad+ \sum_{n=0}^{\infty} \widetilde c_2 \big|Z^{(M)}_{\tau_n}- Z^{(M)}_{t} \big| |W_t -W_{\tau_n}| \mathds{1}_{[\tau_n, \tau_{n+1})}(t) \Big|^q \diff t\Big|^p\Big]^{\frac{1}{p}} \\ 
    &\leq \widetilde c_3\, \E\Big[\Big|\int_0^T \Big|\sum_{n=0}^{\infty} \big(1+ \big|Z^{(M)}_{\tau_n}\big|\big) \big(\delta + |W_t - W_{\tau_n}|^2\big) \mathds{1}_{[\tau_n, \tau_{n+1})}(t)\Big|^q\diff t\Big|^p\Big]^{\frac{1}{p}}\\
    &\quad+ \widetilde c_3\, \E\Big[\Big| \int_0^T \Big|\sum_{n=0}^{\infty} |W_t -W_{\tau_n}| \mathds{1}_{S_\zeta}(Z^{(M)}_t, Z^{(M)}_{\tau_n}) \mathds{1}_{[\tau_n, \tau_{n+1})}(t)\Big|^q\diff t\Big|^p\Big]^{\frac{1}{p}}\\
    &\quad+ \widetilde c_3\, \E\Big[\Big| \int_0^T \Big|\sum_{n=0}^{\infty} \big|Z^{(M)}_{\tau_n}- Z^{(M)}_{t} \big| |W_t -W_{\tau_n}| \mathds{1}_{[\tau_n, \tau_{n+1})}(t)
    \Big|^q \diff t\Big|^p\Big]^{\frac{1}{p}}.
    \end{aligned}
\end{equation}
Next we consider each expectation of \eqref{P2} separately. We start with the first one and apply Jensen's inequality to obtain
\begin{equation}
    \begin{aligned}\label{P3}
    &\E\Big[\Big|\int_0^T \Big|\sum_{n=0}^{\infty} \big(1+ \big|Z^{(M)}_{\tau_n}\big|\big) \big(\delta + |W_t - W_{\tau_n}|^2\big) \mathds{1}_{[\tau_n, \tau_{n+1})}(t)\Big|^q\diff t\Big|^p\Big]\\
    &\leq 2^{pq-1} T^{p-1} \E\Big[\int_0^T \sum_{n=0}^{\infty} \big(1+ \big|Z^{(M)}_{\tau_n}\big|\big)^{pq} \delta^{pq} \mathds{1}_{[\tau_n, \tau_{n+1})}(t)\diff t\Big]\\
    &\quad + 2^{pq-1} T^{p-1} \E\Big[\int_0^T \sum_{n=0}^{\infty} \big(1+ \big|Z^{(M)}_{\tau_n}\big|\big)^{pq} |W_t - W_{\tau_n}|^{2pq-1} \mathds{1}_{[\tau_n, \tau_{n+1})}(t)\diff t\Big].
    \end{aligned}
\end{equation}
For the first expectation in \eqref{P3} we calculate using Lemma \ref{ME} equation \eqref{LME1},
\begin{equation}
    \begin{aligned}\label{P4}
    &\E\Big[\int_0^T\sum_{n=0}^{\infty} \big(1+ \big|Z^{(M)}_{\tau_n}\big|\big)^{pq} \delta^{pq} \mathds{1}_{[\tau_n, \tau_{n+1})}(t)\diff t\Big]
    &\leq 2^{pq-1} \delta^{pq} \int_0^T 1+  \E\Big[\sup_{s\in[0,T]}\big|Z^{(M)}_{s}\big|^{pq}\Big] \diff t\\
    &\leq 2^{pq-1} \delta^{pq} T  \big(1+ c_1(1+|\widetilde\xi|^{pq})\big).
    \end{aligned}
\end{equation}
For the second expectation in \eqref{P3}, Lemma \ref{MW} and Lemma \ref{ME} give
\begin{equation}
    \begin{aligned}\label{P5}
    &\E\Big[\int_0^T \sum_{n=0}^{\infty} \big(1+ \big|Z^{(M)}_{\tau_n}\big|\big)^{pq} |W_t - W_{\tau_n}|^{2pq} \mathds{1}_{[\tau_n, \tau_{n+1})}(t)\diff t\Big]\\
    &\leq 2^{pq-1} c_{W_{2pq}} \delta^{pq} \int_0^T \E\Big[ \sum_{n=0}^{\infty} \big(1+ \big|Z^{(M)}_{\tau_n}\big|^{pq}\big) \mathds{1}_{[\tau_n, \tau_{n+1})}(t)\Big]\diff t\\
    &\leq 2^{pq-1} c_{W_{2pq}} \delta^{pq} T \big( 1+ c_1(1+|\widetilde\xi|^{pq})\big).
    \end{aligned}
\end{equation}
Plugging \eqref{P4} and \eqref{P5} into \eqref{P3} we obtain that there exists a constant $\widetilde c_4\in(0,\infty)$ such that
\begin{equation}
    \begin{aligned}\label{P6}
    &\E\Big[\Big|\int_0^T \Big|\sum_{n=0}^{\infty} \big(1+ \big|Z^{(M)}_{\tau_n}\big|\big) \big(\delta + |W_t - W_{\tau_n}|^2\big) \mathds{1}_{[\tau_n, \tau_{n+1})}(t)\Big|^q\diff t\Big|^p\Big]
    \leq\widetilde c_4 \, \delta^{pq}.
    \end{aligned}
\end{equation}
For the second expectation in \eqref{P2}, Lemma \ref{IndW} gives
\begin{equation}
    \begin{aligned}\label{P7}
    &\E\Big[\Big| \int_0^T \Big|\sum_{n=0}^{\infty} |W_t -W_{\tau_n}| \mathds{1}_{S_\zeta}(Z^{(M)}_t, Z^{(M)}_{\tau_n}) \mathds{1}_{[\tau_n, \tau_{n+1})}(t)\Big|^q\diff t\Big|^p\Big]
    \leq c_9\, \delta^{(q+1)\frac{p}{2}}.
    \end{aligned}
\end{equation}
For the third expectation in \eqref{P2} we use Jensen's inequality, the Cauchy-Schwarz inequality, Lemma \ref{MW}, and Lemma \ref{ME} to obtain that there exists a constant $\widetilde c_5\in(0,\infty)$ such that
\begin{equation}
    \begin{aligned}\label{P8}
    &\E\Big[\Big| \int_0^T \Big|\sum_{n=0}^{\infty} \big|Z^{(M)}_{\tau_n}- Z^{(M)}_{t} \big| |W_t -W_{\tau_n}| \mathds{1}_{[\tau_n, \tau_{n+1})}(t) \Big|^q \diff t\Big|^p\Big]\\
    &\leq T^{p-1} \E\Big[\int_0^T \sum_{n=0}^{\infty} \big|Z^{(M)}_{\tau_n}- Z^{(M)}_{t} \big|^{pq} |W_t -W_{\tau_n}|^{pq} \mathds{1}_{[\tau_n, \tau_{n+1})}(t)  \diff t\Big]\\
    &\leq T^{p-1} \int_0^T \E\Big[ \sum_{n=0}^{\infty} \big|Z^{(M)}_{\tau_n}- Z^{(M)}_{t} \big|^{pq}\mathds{1}_{[\tau_n, \tau_{n+1})}(t)\sum_{l=0}^{\infty} \sup_{s\in[0,\delta]} |W_{s+\tau_l} -W_{\tau_l}|^{pq} \mathds{1}_{[\tau_l, \tau_{l+1})}(t) \Big] \diff t\\
    &\leq T^{p-1} \int_0^T \E\Big[ \sum_{n=0}^{\infty} \big|Z^{(M)}_{\tau_n}- Z^{(M)}_{t} \big|^{2pq}\mathds{1}_{[\tau_n, \tau_{n+1})}(t)\Big]^\frac{1}{2} \E\Big[ \sum_{l=0}^{\infty} \sup_{s\in[0,\delta]}|W_{s+\tau_l} -W_{\tau_l}|^{2pq} \mathds{1}_{[\tau_l, \tau_{l+1})}(t) \Big]^\frac{1}{2} \diff t\\
    &\leq \widetilde c_5\, \delta^{pq}.
    \end{aligned}
\end{equation}
Combining \eqref{P2}, \eqref{P6}, \eqref{P7}, and \eqref{P8} we obtain that there exists a constant $c_{10}\in(0,\infty)$ such that
\begin{equation}
    \begin{aligned}\label{P9}
    &\E\Big[\Big|\int_0^T  \sum_{n=0}^{N_T +M} \big|Z^{(M)}_t -Z^{(M)}_{\tau_n}\big|^q \mathds{1}_{S_\zeta}(Z^{(M)}_t, Z^{(M)}_{\tau_n}) \mathds{1}_{[\tau_n, \tau_{n+1})}(t) \diff t\Big|^p\Big]^{\frac{1}{p}} \\
    &\leq \widetilde c_3\, \big(\widetilde c_4 \, \delta^{pq}\big)^{\frac{1}{p}}
    + \widetilde c_3\, \big(c_9\, \delta^{(q+1)\frac{p}{2}}\big)^{\frac{1}{p}}
    + \widetilde c_3\, \big(\widetilde c_5\, \delta^{pq}\big)^{\frac{1}{p}}  
    \leq c_{10} \delta^\frac{q+1}{2}.
    \end{aligned}
\end{equation}
\end{proof}

\subsection{Convergence result}

In this section we provide the convergence rate of the jump-adapted quasi-Milstein scheme. For the proof we make use of ideas from \cite{muellergronbach2019b}.

\begin{theorem}\label{ConvResTSDE}
Let Assumption \ref{assZ} hold. Then for all $p\in[1,\infty)$ there exists a constant $c_{11}\in(0,\infty)$ such that for all $M\in\N$ and $\delta = \frac{T}{M}$ it holds that
\begin{equation}
\begin{aligned}
\E\Big[\sup_{t\in[0,T]}\big\| Z(t) - Z^{(M)}(t)\big\|^p \Big]^{\frac{1}{p}} \leq c_{11} \delta^{\frac{3}{4}}.
\end{aligned}
\end{equation}
If we additionally assume that $m_{\widetilde\sigma} =0$, then for all $p\in[1,\infty)$ there exists a constant $\hat c_{11}\in(0,\infty)$ such that for all $M\in\N$ and $\delta = \frac{T}{M}$ it holds that
\begin{equation}
\begin{aligned}
\E\Big[\sup_{t\in[0,T]}\big\| Z(t) - Z^{(M)}(t)\big\|^p \Big]^{\frac{1}{p}} \leq \hat c_{11} \delta.
\end{aligned}
\end{equation}
\end{theorem}

\begin{remark}
 We formulate our theorems for general $p\in[1,\infty)$, while the lemmas we use are formulated, for simplicity, for $p\in\N$. Extending the statements of the lemmas to $p\in[1,\infty)$ is however straightforward and therefore omitted.
\end{remark}

\begin{proof}
By Hölder's inequality it is enough to consider $p\in\N$, $p\geq 2$.
First denote
\begin{equation}\label{M0}
    \begin{aligned}
    &S_{\widetilde\mu} = \bigcup_{k=1}^{m_{\widetilde\mu}}\{(x,y\in\R^2: (x-\zeta_k)(y-\zeta_k)\leq 0\},\\
    &S_{\widetilde\sigma} = \bigcup_{j=1}^{m_{\widetilde\sigma}}\{(x,y\in\R^2: (x-\eta_j)(y-\eta_j)\leq 0\}.
    \end{aligned}
\end{equation}
For all $s\in[0,T]$,
\begin{equation}\label{M1}
    \begin{aligned}
    &\E\Big[\sup_{t\in[0,s]}|Z_t - Z_t^{(M)}|^p\Big]\\
    &\leq 3^{p-1} \Big(\E\Big[\sup_{t\in[0,s]} \Big|  \int_0^t \sum_{n=0}^{\infty} \big(\widetilde\mu(Z_u) - \widetilde\mu(Z^{(M)}_{\tau_n})\big) \mathds{1}_{(\tau_n,\tau_{n+1}]}(u) \diff u\Big|^p\Big]\\
    &\quad\quad\quad + \E\Big[\sup_{t\in[0,s]} \Big| \int_0^t \sum_{n=0}^{\infty} \Big(\widetilde\sigma(Z_u) -\widetilde\sigma(Z^{(M)}_{\tau_n}) - \int_{\tau_n}^u \widetilde\sigma (Z^{(M)}_{\tau_n}) d_{\widetilde\sigma} (Z^{(M)}_{\tau_n}) \diff W_v \Big)\mathds{1}_{(\tau_n,\tau_{n+1}]}(u) \diff W_u\Big|^p\Big]\\
    &\quad\quad\quad +\E\Big[\sup_{t\in[0,s]} \Big| \int_0^t \big(\widetilde\rho(Z_{u-}) - \widetilde\rho(Z^{(M)}_{u-})\big) \diff N_u\Big|^p\Big]\Big).
    \end{aligned}
\end{equation}
For the first summand we get
\begin{equation}\label{M2}
    \begin{aligned}
    &\E\Big[\sup_{t\in[0,s]} \Big|  \int_0^t \sum_{n=0}^{\infty} \big(\widetilde\mu(Z_u) - \widetilde\mu(Z^{(M)}_{\tau_n})\big) \mathds{1}_{(\tau_n,\tau_{n+1}]}(u) \diff u\Big|^p\Big]\\
    &\leq 2^{p-1}\Big( \E\Big[\sup_{t\in[0,s]} \Big|  \int_0^t \sum_{n=0}^{\infty} \big(\widetilde\mu(Z_u) - \widetilde\mu(Z^{(M)}_{\tau_n}) - \widetilde\sigma(Z^{(M)}_{\tau_n}) d_{\widetilde\mu}(Z^{(M)}_{\tau_n})(W_u-W_{\tau_n})\big) \mathds{1}_{(\tau_n,\tau_{n+1}]}(u) \diff u\Big|^p\Big] \\
    &\quad\quad\quad +\E\Big[\sup_{t\in[0,s]} \Big|  \int_0^t \sum_{n=0}^{\infty} \widetilde\sigma(Z^{(M)}_{\tau_n}) d_{\widetilde\mu}(Z^{(M)}_{\tau_n}) (W_u-W_{\tau_n}) \mathds{1}_{(\tau_n,\tau_{n+1}]}(u) \diff u\Big|^p\Big]\Big) .
    \end{aligned}
\end{equation}
Observe that for all $u\in(\tau_n,\tau_{n+1}]$ by Lemma \ref{ConAss},
\begin{equation}\label{M3}
    \begin{aligned}
    &\big|\widetilde\mu(Z_u) - \widetilde\mu(Z^{(M)}_{\tau_n}) - \widetilde\sigma(Z^{(M)}_{\tau_n}) d_{\widetilde\mu}(Z^{(M)}_{\tau_n})(W_u-W_{\tau_n})\big|\\
    &\leq \big|\widetilde\mu ( Z_{u}) - \widetilde\mu (Z^{(M)}_{u})\big|
    + \big|\widetilde\mu (Z^{(M)}_{u}) - \widetilde\mu (Z^{(M)}_{\tau_{n}}) -d_{\widetilde\mu} (Z^{(M)}_{\tau_{n}})(Z^{(M)}_{u} - Z^{(M)}_{\tau_{n}})\big|\mathds{1}_{S_{\widetilde\mu}^c}(Z^{(M)}_{u},Z^{(M)}_{\tau_{n}}) \\
    &\quad + \big|\widetilde\mu (Z^{(M)}_{u}) - \widetilde\mu (Z^{(M)}_{\tau_{n}}) -d_{\widetilde\mu} (Z^{(M)}_{\tau_{n}})(Z^{(M)}_{u} - Z^{(M)}_{\tau_{n}})\big|\mathds{1}_{S_{\widetilde\mu}}(Z^{(M)}_{u},Z^{(M)}_{\tau_{n}})\\
    &\quad + \Big|d_{\widetilde\mu} (Z^{(M)}_{\tau_{n}})\Big( \widetilde\mu(Z^{(M)}_{\tau_n})(u-\tau_n) +\frac{1}{2} \widetilde\sigma (Z^{(M)}_{\tau_{n}})d_{\widetilde\sigma} (Z^{(M)}_{\tau_{n}})\big((W_u-W_{\tau_n})^2-(u-\tau_n)\big)\Big)\Big|\\
    &\leq L_{\widetilde\mu} |Z_{u} - Z^{(M)}_{u}|
    + b_{\widetilde\mu } |Z^{(M)}_{u} - Z^{(M)}_{\tau_{n}} |^2
    + (L_{\widetilde\mu} + \|d_{\widetilde\mu}\|_\infty) |Z^{(M)}_{u} - Z^{(M)}_{\tau_{n}}|\mathds{1}_{S_{\widetilde\mu}}(Z^{(M)}_{u},Z^{(M)}_{\tau_{n}})\\
    &\quad + \|d_{\widetilde\mu}\|_\infty \Big(c_{\widetilde\mu} + \frac{1}{2}\|d_{\widetilde\sigma}\|_\infty c_{\widetilde\sigma}\Big)(1+|Z^{(M)}_{\tau_n}|) |u-\tau_n| +\frac{1}{2}\|d_{\widetilde\mu}\|_\infty\|d_{\widetilde\sigma}\|_\infty c_{\widetilde\sigma}(1+|Z^{(M)}_{\tau_n}|) |W_u-W_{\tau_n}|^2.
    \end{aligned}
\end{equation}
For the first summand of \eqref{M2} this and Jensen's inequality ensure that there exists a constant $\widetilde c_1\in(0,\infty)$ such that
\begin{equation}\label{M4}
    \begin{aligned}
    &\E\Big[\sup_{t\in[0,s]} \Big|  \int_0^t \sum_{n=0}^{\infty} \big(\widetilde\mu(Z_u) - \widetilde\mu(Z^{(M)}_{\tau_n}) - \widetilde\sigma(Z^{(M)}_{\tau_n}) d_{\widetilde\mu}(Z^{(M)}_{\tau_n})(W_u-W_{\tau_n})\big) \mathds{1}_{(\tau_n,\tau_{n+1}]}(u) \diff u\Big|^p\Big] \\
    &\leq \E\Big[ \Big(  \int_0^s \sum_{n=0}^{\infty} \big|\widetilde\mu(Z_u) - \widetilde\mu(Z^{(M)}_{\tau_n}) - \widetilde\sigma(Z^{(M)}_{\tau_n}) d_{\widetilde\mu}(Z^{(M)}_{\tau_n})(W_u-W_{\tau_n})\big| \mathds{1}_{(\tau_n,\tau_{n+1}]}(u) \diff u\Big)^p\Big] \\  
    &\leq \widetilde c_1\, \Big(\E\Big[  \int_0^s \sum_{n=0}^{\infty} |Z_{u} - Z^{(M)}_{u}|^p \mathds{1}_{(\tau_n,\tau_{n+1}]}(u) \diff u\Big]\\
    &\quad\quad\quad  + \E\Big[  \int_0^s \sum_{n=0}^{\infty} |Z^{(M)}_{u} - Z^{(M)}_{\tau_{n}} |^{2p}\mathds{1}_{(\tau_n,\tau_{n+1}]}(u) \diff u\Big] \\
    &\quad\quad\quad +\E\Big[ \Big(  \int_0^s \sum_{n=0}^{\infty}  |Z^{(M)}_{u} - Z^{(M)}_{\tau_{n}}|\mathds{1}_{S_{\widetilde\mu}}(Z^{(M)}_{u},Z^{(M)}_{\tau_{n}})\mathds{1}_{(\tau_n,\tau_{n+1}]}(u) \diff u\Big)^p\Big]\\
    &\quad\quad\quad + \E\Big[ \int_0^s \sum_{n=0}^{\infty} (1+|Z^{(M)}_{\tau_n}|)^p |u-\tau_n|^p\mathds{1}_{(\tau_n,\tau_{n+1}]}(u) \diff u \Big]\\
    &\quad\quad\quad+ \E\Big[  \int_0^s \sum_{n=0}^{\infty} (1+|Z^{(M)}_{\tau_n}|)^p |W_u-W_{\tau_n}|^{2p} \mathds{1}_{(\tau_n,\tau_{n+1}]}(u) \diff u\Big] \Big).
    \end{aligned}
\end{equation}
Next we estimate each of the summands of \eqref{M4} separately. For the first one we calculate
\begin{equation}\label{M5}
    \begin{aligned}
    &\E\Big[ \int_0^s \sum_{n=0}^{\infty} \big|Z_{u} - Z^{(M)}_{u}\big|^p\mathds{1}_{(\tau_n,\tau_{n+1}]}(u) \diff u \Big] 
    \leq \int_0^s \E\Big[ \sup_{v\in[0,u]} \big|Z_{v} - Z^{(M)}_{v}\big|^p\Big]  \diff u.
    \end{aligned}
\end{equation}
For the second summand of \eqref{M4} we obtain using  Lemma \ref{ME} equation \eqref{LME2} that there exists $c_2\in(0,\infty)$ such that
\begin{equation}\label{M6}
    \begin{aligned}
    &\E\Big[ \int_0^s \sum_{n=0}^{\infty} \big|Z^{(M)}_{u} - Z^{(M)}_{\tau_{n}} \big|^{2p}\mathds{1}_{(\tau_n,\tau_{n+1}]}(u) \diff u\Big]
    &\leq T c_2 \big(1+ |\widetilde\xi|^{2p}\big)\delta^{p}.
    \end{aligned}
\end{equation}
For the third summand of \eqref{M4} we apply Proposition \ref{Prop} and obtain that there exists $c_{10}\in(0,\infty)$ such that
\begin{equation}\label{M7}
    \begin{aligned}
    &\E\Big[\Big( \int_0^s \sum_{n=0}^{\infty} \big|Z^{(M)}_{u} - Z^{(M)}_{\tau_{n}}\big|\mathds{1}_{S_{\widetilde\mu}}(Z^{(M)}_{u},Z^{(M)}_{\tau_{n}})\mathds{1}_{(\tau_n,\tau_{n+1}]}(u) \diff u\Big)^{p}\Big] \leq c_{10}\, \delta^{p}.
    \end{aligned}
\end{equation}
For the fourth summand of \eqref{M4} we use Lemma \ref{ME} equation \eqref{LME1} to obtain
\begin{equation}\label{M8}
    \begin{aligned}
    &\E\Big[ \int_0^s \sum_{n=0}^{\infty} (1+|Z^{(M)}_{\tau_n}|)^p |u-\tau_n|^p\mathds{1}_{(\tau_n,\tau_{n+1}]}(u) \diff u\Big]\\
    &\leq 2^{p-1} \delta^p  \int_0^s\Big(1 +\E\Big[ \sup_{v\in[0,T]} |Z^{(M)}_{v}|^p \Big]\Big)\diff u 
    \leq 2^{p-1} T \Big(1 +c_1 \big(1+ |\widetilde\xi|^p\big)\Big) \delta^{p}.
    \end{aligned}
\end{equation}
For the fifth summand of \eqref{M4} we use Lemma \ref{MW} and Lemma \ref{ME} equation \eqref{LME1} to get
\begin{equation}\label{M9}
    \begin{aligned}
    &\E\Big[ \int_0^s \sum_{n=0}^{\infty} (1+|Z^{(M)}_{\tau_n}|)^p |W_u-W_{\tau_n}|^{2p} \mathds{1}_{(\tau_n,\tau_{n+1}]}(u) \diff u\Big]\\
    &\leq 2^{p-1} c_{W_{2p}}\, \delta^p \int_0^s \Big(1+ \E\big[\sup_{v\in[0,T]}  |Z^{(M)}_{v}|^p\big]\Big)  \diff u
    \leq 2^{p-1} c_{W_{2p}}\, T \Big(1+ c_1 \big(1+ |\widetilde\xi|^p\big)\Big)\delta^{p} .
    \end{aligned}
\end{equation}
Plugging \eqref{M5}, \eqref{M6}, \eqref{M7}, \eqref{M8}, and \eqref{M9} into \eqref{M4} we obtain that there exists $\widetilde c_2\in(0,\infty)$ such that
\begin{equation}\label{M10}
    \begin{aligned}
    &\E\Big[\sup_{t\in[0,s]} \Big|  \int_0^t \sum_{n=0}^{\infty} \big(\widetilde\mu(Z_u) - \widetilde\mu(Z^{(M)}_{\tau_n}) - \widetilde\sigma (Z^{(M)}_{\tau_n})d_{\widetilde\mu}(Z^{(M)}_{\tau_n})(W_u-W_{\tau_n})\big) \mathds{1}_{(\tau_n,\tau_{n+1}]}(u) \diff u\Big|^p\Big] \\
    &\leq \widetilde c_2\, \int_0^s \E\Big[ \sup_{v\in[0,u]} \big|Z_{u} - Z^{(M)}_{u}\big|^p\Big]  \diff u 
    + \widetilde c_3 \delta^{p}.
    \end{aligned}
\end{equation}
Next we estimate the second summand of \eqref{M2}.
For this we define for all $M\in\N$ and $s\in[0,T]$,
\begin{equation}\label{M13}
    \begin{aligned}
    &U_{M,s} = \int_0^s \sum_{n=0}^{\infty}  \widetilde\sigma(Z^{(M)}_{\tau_n}) d_{\widetilde\mu}(Z^{(M)}_{\tau_n}) (W_u-W_{\tau_n}) \mathds{1}_{(\tau_n,\tau_{n+1}]}(u) \diff u.
    \end{aligned}
\end{equation}
It holds for all $M\in\N$, $m\in\{0,\ldots,M-1\}$, and $s\in[s_m,s_{m+1}]$ that
\begin{equation}\label{M14}
    \begin{aligned}
    &U_{M,s} =  U_{M,s_m} + \int_{s_m}^s \sum_{n=0}^{\infty}  \widetilde\sigma(Z^{(M)}_{\tau_n}) d_{\widetilde\mu}(Z^{(M)}_{\tau_n}) (W_u-W_{\tau_n}) \mathds{1}_{(\tau_n,\tau_{n+1}]}(u) \diff u.
    \end{aligned}
\end{equation}
Now we will show that the sequence $(U_{M,s_m})_{m\in\{0,\ldots,M\}}$ is a discrete martingale with respect to the filtration $(\mathbb{F}_{s_m})_{m\in\{0,\ldots,M\}}$.
To prove this we first show integrability. By Lemma \ref{MW} and Lemma \ref{ME}, for all $m\in\{0,\ldots,M\}$,
\begin{equation}\label{M15}
    \begin{aligned}
    &\E[|U_{M,s_m}|] 
    = \E\Big[\Big| \int_0^{s_m} \sum_{n=0}^{\infty}    \widetilde\sigma(Z^{(M)}_{\tau_n}) d_{\widetilde\mu}(Z^{(M)}_{\tau_n}) (W_u-W_{\tau_n}) \mathds{1}_{(\tau_n,\tau_{n+1}]}(u) \diff u \Big|\Big]\\
    &\leq L_{\widetilde\mu} c_{\widetilde\sigma} \int_0^{T} \E\Big[ \sum_{n=0}^{\infty} (1+ |Z^{(M)}_{\tau_n}|) \big| W_u-W_{\tau_n}\big| \mathds{1}_{(\tau_n,\tau_{n+1}]}(u)\Big] \diff u \\
    &\leq L_{\widetilde\mu} c_{\widetilde\sigma}c_{W_1} \delta^{\frac{1}{2}} \int_0^{T} \Big( 1+ \E\Big[ \sup_{t\in[0,T]}|Z^{(M)}_{t}| \Big]\Big)\diff u 
    \leq L_{\widetilde\mu} c_{\widetilde\sigma}c_{W_1}  \delta^{\frac{1}{2}} T \Big( 1+   c_1 \big(1+ |\widetilde\xi|\big)\Big) <\infty.
    \end{aligned}
\end{equation}
Further it is obvious by definition that $U_{M,s_m}$ is $\F_{s_m}$-measurable. Hence $(U_{M,s_m})_{m\in\N}$ is adapted. Last we prove the martingale property. For $m\in\{0,\ldots,M-1\}$,
\begin{equation}\label{M16}
    \begin{aligned}
    &\E\Big[ U_{M,s_{m+1}} - U_{M,s_{m}}  \Big|\F_{s_m} \Big] \\
    & =  \int_{s_m}^{s_{m+1}}\E\Big[ \sum_{n=0}^{N_u + M}  \widetilde\sigma(Z^{(M)}_{\tau_n}) d_{\widetilde\mu}(Z^{(M)}_{\tau_n}) (W_u-W_{\tau_n}) \mathds{1}_{(\tau_n,\tau_{n+1}]}(u)\Big|\F_{s_m} \Big]  \diff u  \\
    & =  \int_{s_m}^{s_{m+1}} \sum_{k=0}^{\infty} \sum_{n=0}^{k + M}\E\Big[   \widetilde\sigma(Z^{(M)}_{\tau_n}) d_{\widetilde\mu}(Z^{(M)}_{\tau_n}) (W_u-W_{\tau_n}) \mathds{1}_{(\tau_n,\tau_{n+1}]}(u)\mathds{1}_{\{N_u = k\}} \Big|\F_{s_m} \Big]  \diff u  \\
    & =  \int_{s_m}^{s_{m+1}} \sum_{k=0}^{\infty} \sum_{n=0}^{k + M}\E\Big[   \widetilde\sigma(Z^{(M)}_{\tau_n}) d_{\widetilde\mu}(Z^{(M)}_{\tau_n}) (W_u-W_{\tau_n}) \mathds{1}_{(\tau_n,\tau_{n+1}]}(u)\mathds{1}_{\{N_{\tau_n} = k\}} \Big|\F_{s_m} \Big]  \diff u .
    \end{aligned}
\end{equation}
Let $\hat \tau = \tau_n \vee s_m$. Then it holds $\P$-a.s.~that  
\begin{equation}\label{M17}
    \begin{aligned}
    & \E\Big[   \widetilde\sigma(Z^{(M)}_{\tau_n}) d_{\widetilde\mu}(Z^{(M)}_{\tau_n}) (W_u-W_{\tau_n}) \mathds{1}_{(\tau_n,\tau_{n+1}]}(u)\mathds{1}_{\{N_{\tau_n} = k\}}\mathds{1}_{\{\tau_n \geq s_m\}} \Big|\F_{s_m} \Big] \\
    &= \E\Big[ \E\Big[ \widetilde\sigma(Z^{(M)}_{\tau_n}) d_{\widetilde\mu}(Z^{(M)}_{\tau_n}) (W_u-W_{\tau_n}) \mathds{1}_{(\tau_n,\tau_{n+1}]}(u)\mathds{1}_{\{N_{\tau_n} = k\}}\mathds{1}_{\{\tau_n \geq s_m\}} \Big|\F_{\hat \tau} \Big]\Big|\F_{s_m} \Big].
    \end{aligned}
\end{equation}
Next, we express $\tau_{n+1}$ as a measurable function $f$ of $\tau_n$ and $(N_{\tau_n+s} -N_{\tau_n})_{s\geq0}$, as in the proof of Lemma \ref{MW}. Note that on $\{\tau_n\geq s_m\}$, $\hat\tau = \tau_n$ and that $\mathds{1}_{\{\tau_n \geq s_m\}}$ is $\F_{\hat \tau}$-measurable. Hence, $\P$-a.s,
\begin{equation}\label{M18}
    \begin{aligned}
    & \E\Big[\widetilde\sigma(Z^{(M)}_{\tau_n}) d_{\widetilde\mu}(Z^{(M)}_{\tau_n}) (W_u-W_{\tau_n}) \mathds{1}_{(\tau_n,\tau_{n+1}]}(u)\mathds{1}_{\{N_{\tau_n} = k\}}\mathds{1}_{\{\tau_n \geq s_m\}} \Big|\F_{s_m} \Big] \\
    &= \E\Big[ \E\Big[ \widetilde\sigma(Z^{(M)}_{\hat \tau}) d_{\widetilde\mu}(Z^{(M)}_{\hat \tau}) (W_u-W_{\hat \tau}) \mathds{1}_{\{\hat \tau<u\}} \mathds{1}_{\{u\leq f(\hat \tau,(N_{\hat \tau+s} -N_{\hat \tau})_{s\geq0})\}} \mathds{1}_{\{N_{\hat \tau} = k\}}\mathds{1}_{\{\tau_n \geq s_m\}} \Big|\F_{\hat \tau} \Big]\Big|\F_{s_m} \Big]\\
    &= \E\Big[\widetilde\sigma(Z^{(M)}_{\hat \tau}) d_{\widetilde\mu}(Z^{(M)}_{\hat \tau})\mathds{1}_{\{\hat \tau<u\}}\mathds{1}_{\{N_{\hat \tau} = k\}}\mathds{1}_{\{\tau_n \geq s_m\}}
    \E\Big[  (W_u-W_{\hat \tau})  \mathds{1}_{\{u\leq f(\hat \tau,(N_{\hat \tau+s} -N_{\hat \tau})_{s\geq0})\}}  \Big|\F_{\hat\tau} \Big]\Big|\F_{s_m} \Big].
    \end{aligned}
\end{equation}
By \cite[p.~33]{Grikhman2004} and the independence of $W$ and $N$ it holds $\P$-a.s.~that 
\begin{equation}\label{M23}
    \begin{aligned}
    &\mathds{1}_{\{\hat \tau<u\}} \E\Big[  (W_u-W_{\hat \tau})  \mathds{1}_{\{u\leq f(\hat \tau,(N_{\hat \tau+s} -N_{\hat \tau})_{s\geq0})\}} \Big|\F_{\hat\tau} \Big]\\
    &= \mathds{1}_{\{\hat \tau<u\}} \E\Big[  (W_u-W_{z})  \mathds{1}_{\{u\leq f(z,(N_{\hat \tau+s} -N_{\hat \tau})_{s\geq0})\}} \Big|\F_{\hat \tau} \Big]\Big|_{z=\hat\tau}\\
    &= \mathds{1}_{\{\hat \tau<u\}} \E\Big[  (W_u-W_{z}) \mathds{1}_{\{u\leq f(z,(N_{\hat \tau+s} -N_{\hat \tau})_{s\geq0})\}} \Big]\Big|_{z=\hat\tau}\\
    &= \mathds{1}_{\{\hat \tau<u\}}  \E\Big[  W_u-W_{z}\Big] \E\Big[ \mathds{1}_{\{u\leq f(z,(N_{\hat \tau+s} -N_{\hat \tau})_{s\geq0})\}}  \Big]\Big|_{z=\hat\tau}
    = 0.
    \end{aligned}
\end{equation}
Finally, we combine \eqref{M23}, \eqref{M18}, and \eqref{M16} to obtain
\begin{equation}\label{M24}
    \begin{aligned}
    &\E\Big[ U_{M,s_{m+1}} - U_{M,s_{m}}  \Big|\F_{s_m} \Big] = 0.
    \end{aligned}
\end{equation}
Hence, it holds that $(U_{M,s_m})_{m\in\{0,\ldots,M\}}$ is a discrete martingale with respect to $(\mathbb{F}_{s_m})_{m\in\{0,\ldots,M\}}$.

Observe that
\begin{equation}\label{M25}
    \begin{aligned}
    &\sup_{0\leq s\leq T} |U_{M,s}|^p\\
    &\leq 2^{p-1} \Big(\sup_{0\leq s\leq T} |U_{M,\underline{s}}|^p  +\sup_{0\leq s\leq T} \Big|\int_{\underline{s}}^s \sum_{n=0}^{\infty}  \widetilde\sigma(Z^{(M)}_{\tau_n}) d_{\widetilde\mu}(Z^{(M)}_{\tau_n}) (W_u-W_{\tau_n}) \mathds{1}_{(\tau_n,\tau_{n+1}]}(u) \diff u\Big|^p\Big)\\
    &\leq 2^{p-1} \Big(\max_{m=0,\ldots,M} |U_{M,s_m}|^p \\
    &\quad\quad\quad +\max_{m = 0,\ldots,M-1} \Big( \int_{s_m}^{s_{m+1}} \Big|\sum_{n=0}^{\infty}  \widetilde\sigma(Z^{(M)}_{\tau_n}) d_{\widetilde\mu}(Z^{(M)}_{\tau_n}) (W_u-W_{\tau_n}) \mathds{1}_{(\tau_n,\tau_{n+1}]}(u)\Big| \diff u\Big)^p \Big).
    \end{aligned}
\end{equation}
Next we estimate the expectation of each summand separately. For the first one we apply the discrete Burkholder-Davis-Gundy inequality and Jensen's inequality to obtain that there exists a constant $\widetilde c_4\in(0,\infty)$ such that
\begin{equation}\label{M26}
    \begin{aligned}
    &\E\Big[\max_{m=0,\ldots,M} |U_{M,s_m}|^p\Big] 
    \leq \widetilde c_4\, \E\Big[\Big( \sum_{k=0}^{M-1} |U_{M,s_{m+1}}- U_{M,s_m}|^2\Big)^{\frac{p}{2}}\Big] \\
    &\leq \widetilde c_4\, \E\Big[\Big( \sum_{k=0}^{M-1} \delta \int_{s_m}^{s_{m+1}}  \Big|\sum_{n=0}^\infty \widetilde\sigma(Z^{(M)}_{\tau_n}) d_{\widetilde\mu}(Z^{(M)}_{\tau_n})(W_u - W_{\tau_n}) \mathds{1}_{(\tau_n,\tau_{n+1}]} (u)\Big|^2 \diff u\Big)^{\frac{p}{2}}\Big] \\
    &\leq \widetilde c_4\,\delta^{\frac{p}{2}}\, c_{\widetilde\sigma}^{p} L_{\widetilde\mu}^{p} \E\Big[\sup_{t\in[0,T]}\big(1+|Z^{(M)}_{t}|\big)^p\Big( \sum_{k=0}^{M-1}  \int_{s_m}^{s_{m+1}}  \sum_{n=0}^\infty |W_u - W_{\tau_n}|^2 \mathds{1}_{(\tau_n,\tau_{n+1}]} (u) \diff u\Big)^{\frac{p}{2}}\Big] .
    \end{aligned}
\end{equation}
We define $\widetilde c_5 = \widetilde c_4 c_{\widetilde\sigma}^{p} L_{\widetilde\mu}^{p}$ and apply the Cauchy-Schwarz inequality, the Minkowski inequality, Jensen's inequality, and Lemma \ref{ME}\eqref{LME1} to obtain
\begin{equation}\label{M27}
    \begin{aligned}
    &\E\Big[\max_{m=0,\ldots,M} |U_{M,s_m}|^p\Big] \\
    &\leq \widetilde c_5 \delta^\frac{p}{2}\,  \E\Big[\sup_{t\in[0,T]}\big(1+|Z^{(M)}_{t}|\big)^{2p}\Big]^\frac{1}{2}
    \E\Big[\Big( \sum_{k=0}^{M-1}  \int_{s_m}^{s_{m+1}}  \sum_{n=0}^\infty |W_u - W_{\tau_n}|^2 \mathds{1}_{(\tau_n,\tau_{n+1}]} (u) \diff u\Big)^{p}\Big]^\frac{1}{2} \\
    &\leq 2^{2p-1} \widetilde c_5 \delta^{p-\frac{1}{2}}\, \big(1+ c_1 \big(1+|\widetilde\xi|^{2p}\big)\big)^\frac{1}{2}
    \Big( \sum_{k=0}^{M-1}\Big( \int_{s_m}^{s_{m+1}} \E\Big[\sum_{n=0}^\infty |W_u - W_{\tau_n}|^{2p} \mathds{1}_{(\tau_n,\tau_{n+1}]} (u) \Big]\diff u \Big)^\frac{1}{p}\Big)^\frac{p}{2}.
    \end{aligned}
\end{equation}
Next we apply Lemma \ref{MW} with $p=0$ and $q=2p$ to obtain
\begin{equation}\label{M28}
    \begin{aligned}
    &\E\Big[\max_{m=0,\ldots,M} |U_{M,s_m}|^p\Big] 
    \leq 4^p \widetilde c_5\delta^{p-\frac{1}{2}}\, \big(1+ c_1 \big(1+|\widetilde\xi|^{2p}\big)\big)^\frac{1}{2} \Big( \sum_{k=0}^{M-1}\big( c_{W_{2p}}\delta^{p+1} \big)^\frac{1}{p}\Big)^\frac{p}{2}\\
    &=  4^p \widetilde c_5 \, \big(1+ c_1 \big(1+|\widetilde\xi|^{2p}\big)\big)^\frac{1}{2}c_{W_{2p}}^\frac{1}{2} T^\frac{p}{2}\delta^{p}.
    \end{aligned}
\end{equation}
For the expectation of the second summand of \eqref{M25}, Jensen's inequality, the Cauchy-Schwarz inequality, Lemma \ref{ME} \eqref{LME1}, and Lemma \ref{MW} ensure
\begin{equation}\label{M29}
    \begin{aligned}
    &\E\Big[\max_{m = 0,\ldots,M-1} \Big( \int_{s_m}^{s_{m+1}} \Big|\sum_{n=0}^{\infty}  \widetilde\sigma(Z^{(M)}_{\tau_n}) d_{\widetilde\mu}(Z^{(M)}_{\tau_n}) (W_u-W_{\tau_n}) \mathds{1}_{(\tau_n,\tau_{n+1}]}(u)\Big| \diff u\Big)^p \Big]\\
    &\leq \delta^{p-1} \E\Big[\max_{m = 0,\ldots,M-1} \int_{s_m}^{s_{m+1}} \Big|\sum_{n=0}^{\infty}  \widetilde\sigma(Z^{(M)}_{\tau_n}) d_{\widetilde\mu}(Z^{(M)}_{\tau_n}) (W_u-W_{\tau_n}) \mathds{1}_{(\tau_n,\tau_{n+1}]}(u)\Big|^p \diff u\Big]\\
    &\leq \delta^{p-1} c_{\widetilde\sigma}^{p} L_{\widetilde\mu}^p\, \E\Big[\sup_{t\in[0,T]}\big(1+ |Z^{(M)}_{t}|\big)^p\max_{m = 0,\ldots,M-1} \int_{s_m}^{s_{m+1}} \sum_{n=0}^{\infty}   |W_u-W_{\tau_n}|^p \mathds{1}_{(\tau_n,\tau_{n+1}]}(u) \diff u\Big]\\
    &\leq \delta^{p-1} c_{\widetilde\sigma}^{p} L_{\widetilde\mu}^p\, \E\Big[\sup_{t\in[0,T]}\big(1+ |Z^{(M)}_{t}|\big)^{2p}\Big]^\frac{1}{2}\\
    &\quad\quad \cdot \E\Big[\max_{m = 0,\ldots,M-1} \Big(\int_{s_m}^{s_{m+1}} \sum_{n=0}^{\infty}   |W_u-W_{\tau_n}|^p \mathds{1}_{(\tau_n,\tau_{n+1}]}(u) \diff u\Big)^2\Big]^\frac{1}{2} \\
    &\leq 2^{p-\frac{1}{2}}\delta^{p-\frac{1}{2}} c_{\widetilde\sigma}^{p} L_{\widetilde\mu}^p \big(1+c_1\big(1+|\widetilde\xi|^{2p}\big)\big)^\frac{1}{2}
    \E\Big[\sum_{m=0}^{M-1} \int_{s_m}^{s_{m+1}} \sum_{n=0}^{\infty}   |W_u-W_{\tau_n}|^{2p} \mathds{1}_{(\tau_n,\tau_{n+1}]}(u) \diff u\Big]^\frac{1}{2} \\
    &\leq 2^{p-\frac{1}{2}}\delta^{p-\frac{1}{2}} c_{\widetilde\sigma}^{p} L_{\widetilde\mu}^p \big(1+c_1\big(1+|\widetilde\xi|^{2p}\big)\big)^\frac{1}{2}     \Big(\sum_{m=0}^{M-1} \int_{s_m}^{s_{m+1}} c_{W_{2p}} \delta^{p} \diff u\Big)^\frac{1}{2} \\
    &= 2^{p-\frac{1}{2}}T^{\frac{1}{2}} c_{\widetilde\sigma}^{p} L_{\widetilde\mu}^p \big(1+c_1\big(1+|\widetilde\xi|^{2p}\big)\big)^\frac{1}{2} c_{W_{2p}}^{\frac{1}{2}} \delta^{\frac{3}{2}p-\frac{1}{2}}.
    \end{aligned}
\end{equation}
Combining \eqref{M25}, \eqref{M28}, and \eqref{M29} we obtain that there exists a constant $\widetilde c_6\in(0,\infty)$ such that
\begin{equation}\label{M30}
    \begin{aligned}
    &\E\Big[\sup_{0\leq s\leq T} |U_{M,s}|^p\Big]\leq \widetilde c_6 \delta^{p}.
    \end{aligned}
\end{equation}
For estimating the second summand of \eqref{M1}, observe that for  $u\in(\tau_n,\tau_{n+1}]$ by Lemma \ref{ConAss},
\begin{equation}\label{M52}
    \begin{aligned}
    &\Big|\widetilde\sigma(Z_u) -\widetilde\sigma(Z^{(M)}_{\tau_n}) - \int_{\tau_n}^u \widetilde\sigma(Z^{(M)}_{\tau_n}) d_{\widetilde\sigma} (Z^{(M)}_{\tau_n}) \diff W_v\Big|\\
    &\leq \big|\widetilde\sigma ( Z_{u}) - \widetilde\sigma (Z^{(M)}_{u})\big|
    + \big|\widetilde\sigma (Z^{(M)}_{u}) - \widetilde\sigma (Z^{(M)}_{\tau_{n}}) -d_{\widetilde\sigma} (Z^{(M)}_{\tau_{n}})(Z^{(M)}_{u} - Z^{(M)}_{\tau_{n}})\big|\mathds{1}_{S_{\widetilde\sigma}^c}(Z^{(M)}_{u},Z^{(M)}_{\tau_{n}}) \\
    &\quad + \big|\widetilde\sigma (Z^{(M)}_{u}) - \widetilde\sigma (Z^{(M)}_{\tau_{n}}) -d_{\widetilde\sigma} (Z^{(M)}_{\tau_{n}})(Z^{(M)}_{u} - Z^{(M)}_{\tau_{n}})\big|\mathds{1}_{S_{\widetilde\sigma}}(Z^{(M)}_{u},Z^{(M)}_{\tau_{n}})\\
    &\quad + \Big|d_{\widetilde\sigma} (Z^{(M)}_{\tau_{n}})\Big( \widetilde\mu(Z^{(M)}_{\tau_n})(u-\tau_n) +\frac{1}{2} \widetilde\sigma (Z^{(M)}_{\tau_{n}})d_{\widetilde\sigma} (Z^{(M)}_{\tau_{n}})\big((W_u-W_{\tau_n})^2-(u-\tau_n)\big)\Big)\Big|\\
    &\leq L_{\widetilde\sigma} |Z_{u} - Z^{(M)}_{u}|
    + b_{\widetilde\sigma } |Z^{(M)}_{u} - Z^{(M)}_{\tau_{n}} |^2 
    + (L_{\widetilde\sigma} + \|d_{\widetilde\sigma}\|_\infty) |Z^{(M)}_{u} - Z^{(M)}_{\tau_{n}}|\mathds{1}_{S_{\widetilde\sigma}}(Z^{(M)}_{u},Z^{(M)}_{\tau_{n}})\\
    &\quad + \|d_{\widetilde\sigma}\|_\infty \Big(c_{\widetilde\mu} + \frac{1}{2}\|d_{\widetilde\sigma}\|_\infty c_{\widetilde\sigma}\Big)(1+|Z^{(M)}_{\tau_n}|) |u-\tau_n| +\frac{1}{2}\|d_{\widetilde\sigma}\|^2_\infty c_{\widetilde\sigma}(1+|Z^{(M)}_{\tau_n}|) |W_u-W_{\tau_n}|^2.
    \end{aligned}
\end{equation}
Using this, the Burkholder-Davis-Gundy inequality, and Jensen's inequality we obtain for the second summand of \eqref{M1} that there exist constants $\widetilde c_7, \widetilde c_8 \in (0,\infty)$ such that
\begin{equation}\label{M53}
    \begin{aligned}
    &\E\Big[\sup_{t\in[0,s]} \Big| \int_0^t \sum_{n=0}^{\infty} \Big(\widetilde\sigma(Z_u) -\widetilde\sigma(Z^{(M)}_{\tau_n}) - \int_{\tau_n}^u \widetilde\sigma(Z^{(M)}_{\tau_n}) d_{\widetilde\sigma} (Z^{(M)}_{\tau_n}) \diff W_v \Big)\mathds{1}_{(\tau_n,\tau_{n+1}]}(u) \diff W_u\Big|^p\Big]\\
    &\leq \widetilde c_7 \, \E\Big[\Big( \int_0^s \sum_{n=0}^{\infty} \Big|\widetilde\sigma(Z_u) -\widetilde\sigma(Z^{(M)}_{\tau_n}) - \int_{\tau_n}^u \widetilde\sigma(Z^{(M)}_{\tau_n}) d_{\widetilde\sigma} (Z^{(M)}_{\tau_n}) \diff W_v \Big|^2\mathds{1}_{(\tau_n,\tau_{n+1}]}(u) \diff u\Big)^{\frac{p}{2}}\Big]\\
    &\leq \widetilde c_8\, \Big( \E\Big[ \int_0^s \sum_{n=0}^{\infty} \big|Z_{u} - Z^{(M)}_{u}\big|^p\mathds{1}_{(\tau_n,\tau_{n+1}]}(u) \diff u \Big] \\
    &\quad\quad\quad + \E\Big[ \int_0^s \sum_{n=0}^{\infty} \big|Z^{(M)}_{u} - Z^{(M)}_{\tau_{n}} \big|^{2p}\mathds{1}_{(\tau_n,\tau_{n+1}]}(u) \diff u\Big] \\
    &\quad\quad\quad +\E\Big[\Big( \int_0^s \sum_{n=0}^{\infty} \big|Z^{(M)}_{u} - Z^{(M)}_{\tau_{n}}\big|^2\mathds{1}_{S_{\widetilde\sigma}}(Z^{(M)}_{u},Z^{(M)}_{\tau_{n}})\mathds{1}_{(\tau_n,\tau_{n+1}]}(u) \diff u\Big)^{\frac{p}{2}}\Big]\\
    &\quad\quad\quad + \E\Big[ \int_0^s \sum_{n=0}^{\infty} (1+|Z^{(M)}_{\tau_n}|)^p |u-\tau_n|^p\mathds{1}_{(\tau_n,\tau_{n+1}]}(u) \diff u\Big]\\
    &\quad\quad\quad +\E\Big[ \int_0^s \sum_{n=0}^{\infty} (1+|Z^{(M)}_{\tau_n}|)^p |W_u-W_{\tau_n}|^{2p}\Big) \mathds{1}_{(\tau_n,\tau_{n+1}]}(u) \diff u\Big]\Big).
    \end{aligned}
\end{equation}
For the third summand of \eqref{M53}, Proposition \ref{Prop} yields
\begin{equation}\label{M56}
    \begin{aligned}
    &\E\Big[\Big( \int_0^s \sum_{n=0}^{\infty} \big|Z^{(M)}_{u} - Z^{(M)}_{\tau_{n}}\big|^2\mathds{1}_{S_{\widetilde\sigma}}(Z^{(M)}_{u},Z^{(M)}_{\tau_{n}})\mathds{1}_{(\tau_n,\tau_{n+1}]}(u) \diff u\Big)^{\frac{p}{2}}\Big] \leq c_{10}\, \delta^{\frac{3}{4}p}.
    \end{aligned}
\end{equation}
Plugging this, \eqref{M5}, \eqref{M6}, \eqref{M8}, and \eqref{M9} into \eqref{M53} that there exists a constant $\widetilde c_9\in(0,\infty)$ such that
\begin{equation}\label{M59}
    \begin{aligned}
    &\E\Big[\sup_{t\in[0,s]} \Big| \int_0^t \sum_{n=0}^{\infty} \Big(\widetilde\sigma(Z_u) -\widetilde\sigma(Z^{(M)}_{\tau_n}) - \int_{\tau_n}^u \widetilde\sigma d_{\widetilde\sigma}(Z^{(M)}_{\tau_n}) (Z^{(M)}_{\tau_n}) \diff W_v \Big)\mathds{1}_{(\tau_n,\tau_{n+1}]}(u) \diff W_u\Big|^p\Big]\\
    &\leq \widetilde c_9\, \int_0^s \E\Big[ \sup_{v\in[0,u]} \big|Z_{u} - Z^{(M)}_{u}\big|^p\Big]  \diff u 
    + \widetilde c_9\,  \delta^{\frac{3}{4}p}.\\
    \end{aligned}
\end{equation}
For the third summand of \eqref{M1}, Lemma \ref{BDGaKunita} ensures
\begin{equation}\label{M60}
    \begin{aligned}
    &\E\Big[\sup_{t\in[0,s]} \Big| \int_0^t \big(\widetilde\rho(Z_{u-}) -\widetilde\rho(Z^{(M)}_{u-}) \big) \diff N_u\Big|^p\Big]
    \leq \hat c\, L_{\widetilde\rho}^p  \int_0^s \E\Big[ \sup_{v\in[0,u]} \big| Z_{v} - Z^{(M)}_{v} \big|^p\Big] \diff u.
    \end{aligned}
\end{equation}
Combining \eqref{M1}, \eqref{M2}, \eqref{M10}, \eqref{M13}, \eqref{M30}, \eqref{M59}, and \eqref{M60} we obtain that there exists a constant $\widetilde c_{10}\in(0, \infty)$ such that 
\begin{equation}\label{M61}
    \begin{aligned}
    &\E\Big[\sup_{t\in[0,s]}|Z_t - Z_t^{(M)}|^p\Big]
    \leq \widetilde c_{10}\, \Big(\int_0^s \E\Big[ \sup_{v\in[0,u]} \big| Z_{v} - Z^{(M)}_{v} \big|^p\Big] \diff u + \delta^{\frac{3p}{4}}\Big). 
    \end{aligned}
\end{equation}
Note that $\E\Big[\sup_{v\in[0,T]} \big| Z_{v} - Z^{(M)}_{v} \big|^p\Big]<\infty$ and that $s\mapsto \E\Big[\sup_{v\in[0,s]} \big| Z_{v} - Z^{(M)}_{v} \big|^p\Big]$ is a Borel measurable mapping, because it is monotonically increasing. Hence we apply Gronwall's inequality and obtain that there exists a constant $c_{11}\in(0,\infty)$ such that
\begin{equation}\label{M62}
    \begin{aligned}
    &\E\Big[\sup_{t\in[0,T]}|Z_t - Z_t^{(M)}|^p\Big]
    \leq c_{11}  \delta^{\frac{3p}{4}}.
    \end{aligned}
\end{equation}
This proves the first statement. For the second statement we assume $S_{\widetilde\sigma}= \emptyset$. This improves the estimate in \eqref{M59}. Hence, the claim follows from analogue calculations.
\end{proof}

\section{Convergence of the transformation-based jump-adapted quasi-Milstein scheme}

To get from SDE \eqref{eq:SDE} satisfying Assumption \eqref{assX} to SDE \eqref{eq:TSDE} satisfying Assumption \eqref{assZ} we use a certain transformation. It is based on the transformation from \cite{LS17b}; we slightly modify it exactly as in \cite{muellergronbach2019b}.
The function $G$ is constructed such that the coefficients of SDE \eqref{eq:SDE} are adjusted locally around the potential discontinuities of the drift.
For the convenience of the reader we recall the definition and its properties.

Let $\alpha_1,\ldots,\alpha_m$ be defined for all $i\in\{1,\ldots,m\}$ by 
\begin{equation}\label{alpha}
\begin{aligned}
\alpha_i = \frac{\mu(\zeta_i -)- \mu(\zeta_i +)}{2\sigma^2(\zeta_i)}.
\end{aligned}
\end{equation}
Set
\begin{align*}
\varphi := \min\bigg\{\min\limits_{1\le i\le m}\frac{1}{8|\alpha_i|},\min\limits_{1\le i\le m-1}\frac{\zeta_{i+1}-\zeta_i}{2}\bigg\},
\end{align*}
using the conventions that $\min \emptyset =\infty$ and $\frac{1}{0} = \infty$, fix $\nu\in(0,\varphi)$ and define the bump function
\begin{align}\label{eq:bump}
\phi(u)=
\begin{cases}
(1-u^2)^4 & \text{if } |u|\leq 1,\\
0 & \text{otherwise}.
\end{cases}
\end{align}
With this we define the transformation function $G\colon\R\to\R$ for all $x\in\R$ by 
\begin{align}\label{eq:G-1d}
G(x)=x+ \sum_{i=1}^k \alpha_i\, \phi\!\left(\frac{x-\zeta_i}{\nu}\right)(x-\zeta_i)|x-\zeta_i|.
\end{align}

The following lemma provides properties of the transformation $G$.
\newpage

\begin{lemma}[\text{\cite[Lemma 1]{muellergronbach2019b}}]\label{lemx1}
The function $G$ satisfies the following properties.
\begin{itemize}
\item[\namedlabel{L1i}{(i)}] The function $G$  is differentiable on $\R$.
\item[\namedlabel{L1ii}{(ii)}] The derivative $G'$ is Lipschitz continuous.
\item[\namedlabel{L1iii}{(iii)}] The derivative satisfies $G'(\zeta_i) = 1$ for all $i\in\{1,\dots,m\}$.
\item[\namedlabel{L1iv}{(iv)}] The derivative satisfies $\inf_{x\in\R} G'(x)>0$ and there exists a constant $c \in (0,\infty)$ such that for all $x\in\R$ with $|x|>c$, $G'(x) = 1$. 
\item[\namedlabel{L1v}{(v)}] The function $G$ has a global inverse $G^{-1}\colon \R\to \R$ which is Lipschitz continuous.
\item[\namedlabel{L1vi}{(vi)}] For all $i\in\{1,\dots,m+1\}$, the function $G'$ is twice differentiable on $(\zeta_{i-1},\zeta_i)$ with Lipschitz continuous derivatives $G''$ and $G'''$.
\item[\namedlabel{L1vii}{(vii)}] For all $i\in\{1,\dots,m\}$ the one-sided limits  $G''(\zeta_i-) $ and $G''(\zeta_i+)$ exist and satisfy
\[
G''(\zeta_i-) = -2\alpha_i,\quad G''(\zeta_i+) = 2\alpha_i.
\]
\end{itemize}
\end{lemma}

We know that the function $G''\colon\bigcup_{i=1}^m (\zeta_{i-1},\zeta_i) \to \R$ is well defined. In the following we extend this mapping to $G''\colon\R\to\R$ by defining
\begin{equation}\label{exG2}
G''(\zeta_i) = 2\alpha_i + 2 \frac{\mu(\zeta_i+)-\mu(\zeta_i)}{\sigma^2(\zeta_i)}, \quad i\in\{1,\ldots,m\}.
\end{equation}
Now we define the process $Z\colon[0,T]\times\Omega\to \R$ by $Z = G(X)$. Itô's formula shows that $Z$ satisfies 
\begin{align}
\diff Z_t &=\widetilde \mu(Z_t)  \diff t + \widetilde \sigma(Z_t) \diff W_t+\widetilde \rho(Z_{t-})\diff N_t , \quad t\in[0,T], \quad Z_0=\widetilde \xi,
\end{align}
where 
\begin{equation}\label{eqTCoeff}
\begin{aligned}
&\widetilde \mu = (G'\cdot \mu +\frac{1}{2} G''\cdot\sigma^2) \circ G^{-1}, \\ 
&\widetilde \sigma = (G'\cdot \sigma) \circ G^{-1}, \\
&\widetilde \rho = G(G^{-1} +\rho(G^{-1})) - \operatorname{id},
\end{aligned}
\end{equation}
see \cite[Theorem 3.1]{PS20}.

\begin{lemma}\label{TCoeff}
Assume the coefficients $\mu$, $\sigma$, and $\rho$ satisfy Assumption \ref{assX}. Then $\widetilde \mu$, $\widetilde\sigma$, and $\widetilde\rho$ from \eqref{eqTCoeff} satisfy Assumption \ref{assZ}.
\end{lemma}

For $\widetilde \mu$ and $\widetilde \sigma$ this result is already proven  in \cite[Lemma 2]{muellergronbach2019b}. For $\widetilde \rho$ the proof works analogously to the respective part in \cite[proof of Theorem 3.1]{PS20}.
Further we obtain the following lemma similar to \cite[Lemma~3]{muellergronbach2019b} and \cite[Theorem 3.1]{PS20}.

\begin{lemma}\label{TCoeff2}
Let Assumption \ref{assX} hold. Then the process $Z$ is the unique solution of SDE \eqref{eq:TSDE} with $\widetilde \mu$, $\widetilde\sigma$, and $\widetilde\rho$ defined as in Lemma \ref{TCoeff}.
\end{lemma}

\begin{proof}
As in the proof of \cite[Therem 3.1]{PS20} we can apply the Meyer-Itô formula of \cite[p.~221,~Theorem~7]{protter2005} to derive that $Z$ is a solution of SDE \eqref{eq:TSDE}. Because by Lemma \ref{TCoeff} $\widetilde \mu$ , $\widetilde\sigma$, and $\widetilde\rho$ are Lipschitz continuous, we know by \cite[p.~255,~Theorem~6]{protter2005} that the solution of the SDE is unique.
\end{proof}

Now we have all tools at hand to prove that there exists an approximation scheme with strong rate of convergence $3/4$. This scheme is the transformation-based jump-adapted quasi-Milstein scheme.

\begin{theorem}\label{MainWeakAss}
Let Assumption \ref{assX} hold. Let $p\in[1,\infty)$, $G$ be the transformation defined in \eqref{eq:G-1d}, let $Z\colon[0,T]\times\Omega\to \R$ be defined by $Z = G(X)$, and let $Z^{(M)}$ be the jump-adapted quasi-Milstein scheme for $Z$ as introduced in \eqref{EqDefMSLL} and \eqref{EqDefMS}.
Then there exists a constant $c\in(0,\infty)$ such that for all $M\in\N$ and $\delta = \frac{T}{M}$ it holds that
\begin{equation}
    \begin{aligned}
    \E\Big[ \sup_{t\in[0,T]} |X_t- G^{-1}(Z^{(M)}_t) |^p\Big]^\frac{1}{p} \leq c \delta^\frac{3}{4}.
    \end{aligned}
\end{equation}
\end{theorem}

\begin{proof}
Using $X = G^{-1}(Z)$, the Lipschitz continuity of $G^{-1}$ (Lemma \ref{lemx1} \ref{L1v}), the fact that the coefficients of $Z$ satisfy Assumption \ref{assZ} (Lemma \ref{TCoeff2}), and Theorem \ref{ConvResTSDE} we obtain that there exists a constant $c\in(0,\infty)$ such that
\begin{equation}
    \begin{aligned}
    &\E\Big[ \sup_{t\in[0,T]} |X_t- G^{-1}(Z^{(M)}_t) |^p\Big]^\frac{1}{p}
    \leq \E\Big[ \sup_{t\in[0,T]} |G^{-1}(Z_t) - G^{-1}(Z^{(M)}_t) |^p\Big]^\frac{1}{p} \\
    &\leq L_{G^{-1}} \E\Big[ \sup_{t\in[0,T]} |Z_t - Z^{(M)}_t|^p\Big]^\frac{1}{p} 
    \leq L_{G^{-1}} c_{11} \delta^{\frac{3}{4}} 
    \leq c\, \delta^\frac{3}{4}.
    \end{aligned}
\end{equation}
\end{proof}

\section*{Acknowledgements}
V. Schwarz and M. Sz\"olgyenyi are supported by the Austrian Science Fund (FWF): DOC 78.
We would like to thank Thomas M\"uller-Gronbach, who at the B\c{e}dlewo workshop "Numerical analysis and applications of SDEs" pointed out to us that under Assumption \ref{assZ} in the special case of $m_{\widetilde\sigma} =0$ we could possibly prove convergence order 1, which was successful and improved Theorem \ref{ConvResTSDE}.

\appendix

\section{Additional proofs}\label{appendix}

\begin{proof}[Proof of Lemma \ref{MW}]
We recall \eqref{EqTauGen} to obtain
\begin{equation}
    \begin{aligned}\label{MW2}
    &\tau_{n+1}
    = \big(\tau_n + \inf\{s \geq 0 : N_{\tau_n+s} -N_{\tau_n} >0\}\big)\wedge \min\{s_m: m\in\{0,\ldots,M\}, s_m > \tau_{n}\}\wedge T.
    \end{aligned}
\end{equation}
Hence, $\tau_{n+1}$ can be expressed as a measurable function of $\tau_n$ and $(N_{\tau_n+s} -N_{\tau_n})_{s\geq 0}$. We denote this function by $f$ and compute
\begin{equation}
    \begin{aligned}\label{MW3}
    &\E\big[\big(1+ \big|Z^{(M)}_{\tau_n}\big|^p\big)   |W_u-W_{\tau_n}|^q \mathds{1}_{(\tau_n,\tau_{n+1}]}(u)  \mathds{1}_{\{N_u = k\}} \big]\\
    &= \E\big[\big(1+ \big|Z^{(M)}_{\tau_n}\big|^p\big)   |W_u-W_{\tau_n}|^q \mathds{1}_{(\tau_n,\tau_{n+1}]}(u)  \mathds{1}_{\{N_{\tau_n} = k\}} \big]\\
    &\leq \E\Big[\big(1+ \big|Z^{(M)}_{\tau_n}\big|^p\big)  \sup_{s\in[0,\delta]} |W_{s+\tau_n}-W_{\tau_n}|^q \mathds{1}_{\{\tau_n <u\}} \mathds{1}_{\{f(\tau_n, (N_{s+\tau_n}-N_{\tau_n}))_{s\geq 0})\geq u\}} \mathds{1}_{\{N_{\tau_n} = k\}} \Big].
    \end{aligned}
\end{equation}
Next we use the conditional expectation and \cite[p.~33]{Grikhman2004} to calculate
\begin{equation}
    \begin{aligned}\label{MW4}
    &\E\big[\big(1+ \big|Z^{(M)}_{\tau_n}\big|^p\big)   |W_u-W_{\tau_n}|^q \mathds{1}_{(\tau_n,\tau_{n+1}]}(u)  \mathds{1}_{\{N_u = k\}} \big]\\
    &\leq \E\Big[\E\Big[\big(1+ \big|Z^{(M)}_{\tau_n}\big|^p\big)  \sup_{s\in[0,\delta]} |W_{s+\tau_n}-W_{\tau_n}|^q \mathds{1}_{\{\tau_n <u\}} \mathds{1}_{\{f(\tau_n, (N_{s+\tau_n}-N_{\tau_n}))_{s\geq 0})\geq u\}} \mathds{1}_{\{N_{\tau_n} = k\}} \Big| \F_{\tau_n} \Big] \Big] \\
    &= \E\Big[\big(1+ \big|Z^{(M)}_{\tau_n}\big|^p\big) \mathds{1}_{\{\tau_n <u\}} \mathds{1}_{\{N_{\tau_n} = k\}} \E\Big[ \sup_{s\in[0,\delta]} |W_{s+\tau_n}-W_{\tau_n}|^q  \mathds{1}_{\{f(\tau_n, (N_{s+\tau_n}-N_{\tau_n}))_{s\geq 0})\geq u\}}  \Big| \F_{\tau_n} \Big] \Big] \\
    &= \E\Big[\big(1+ \big|Z^{(M)}_{\tau_n}\big|^p\big) \mathds{1}_{\{\tau_n <u\}} \mathds{1}_{\{N_{\tau_n} = k\}} \E\Big[ \sup_{s\in[0,\delta]} |W_{s+\tau_n}-W_{\tau_n}|^q  \mathds{1}_{\{f(z, (N_{s+\tau_n}-N_{\tau_n}))_{s\geq 0})\geq u\}}  \Big| \F_{\tau_n} \Big]\Big|_{z =\tau_n} \Big].
    \end{aligned}
\end{equation}
By Lemma \ref{PropDiscGrid} it holds that $(W_{\tau_n+s}- W_{\tau_n})_{s\geq 0}$ and $(N_{\tau_n+s}- N_{\tau_n})_{s\geq 0}$ are independent of $\F_{\tau_n}$ and that $(W_{\tau_n+s}- W_{\tau_n})_{s\geq 0}$ is independent of $(N_{\tau_n+s}- N_{\tau_n})_{s\geq 0}$. Using this, Lemma \ref{MEW}, and \cite[p.~33]{Grikhman2004} we obtain
\begin{equation}
    \begin{aligned}\label{MW4a}
    &\E\big[\big(1+ \big|Z^{(M)}_{\tau_n}\big|^p\big)   |W_u-W_{\tau_n}|^q \mathds{1}_{(\tau_n,\tau_{n+1}]}(u)  \mathds{1}_{\{N_u = k\}} \big]\\
    &=\E\Big[\big(1+ \big|Z^{(M)}_{\tau_n}\big|^p\big) \mathds{1}_{\{\tau_n <u\}} \mathds{1}_{\{N_{\tau_n} = k\}} \E\Big[ \sup_{s\in[0,\delta]} |W_{s+\tau_n}-W_{\tau_n}|^q\Big] \E\Big[ \mathds{1}_{\{f(z, (N_{s+\tau_n}-N_{\tau_n}))_{s\geq 0})\geq u\}} \Big]\Big|_{z =\tau_n} \Big]\\
    &\leq c_{W_q} \delta^\frac{q}{2} \E\Big[\big(1+ \big|Z^{(M)}_{\tau_n}\big|^p\big) \mathds{1}_{\{\tau_n <u\}} \mathds{1}_{\{N_{\tau_n} = k\}} \E\Big[ \mathds{1}_{\{f(z, (N_{s+\tau_n}-N_{\tau_n}))_{s\geq0})\geq u\}} \Big]\Big|_{z =\tau_n} \Big]\\
    &= c_{W_q} \delta^\frac{q}{2} \E\Big[\big(1+ \big|Z^{(M)}_{\tau_n}\big|^p\big) \mathds{1}_{\{\tau_n <u\}}  \mathds{1}_{\{f(\tau_n, (N_{s+\tau_n}-N_{\tau_n}))_{s\geq0})\geq u\}} \mathds{1}_{\{N_{\tau_n} = k\}} \Big]\\
    &= c_{W_q} \delta^\frac{q}{2} \E\Big[\big(1+ \big|Z^{(M)}_{\tau_n}\big|^p\big) \mathds{1}_{(\tau_n, \tau_{n+1}] }(u)  \mathds{1}_{\{N_{u} = k\}} \Big],
    \end{aligned}
\end{equation}
which proves the first statement.

Using \eqref{MW4a} we get
\begin{equation}
    \begin{aligned}\label{MW5}
    &\E\Big[\sum_{n=0}^\infty \big(1+ \big|Z^{(M)}_{\tau_n}\big|^p\big)   |W_u-W_{\tau_n}|^q \mathds{1}_{(\tau_n,\tau_{n+1}]}(u) \Big]\\
    &= \E\Big[\sum_{n=0}^{N_u+M} \big(1+ \big|Z^{(M)}_{\tau_n}\big|^p\big)   |W_u-W_{\tau_n}|^q \mathds{1}_{(\tau_n,\tau_{n+1}]}(u) \Big]\\
    &= \sum_{k=0}^\infty \sum_{n=0}^{k+M} \E\Big[ \big(1+ \big|Z^{(M)}_{\tau_n}\big|^p\big)   |W_u-W_{\tau_n}|^q \mathds{1}_{(\tau_n,\tau_{n+1}]}(u) \mathds{1}_{\{N_u=k\}} \Big]\\
    &\leq \sum_{k=0}^\infty \sum_{n=0}^{k+M} c_{W_q} \delta^\frac{q}{2} \E\Big[\big(1+ \big|Z^{(M)}_{\tau_n}\big|^p\big) \mathds{1}_{(\tau_n, \tau_{n+1}] }(u)  \mathds{1}_{\{N_{u} = k\}} \Big]\\
    &= c_{W_q} \delta^\frac{q}{2} \E\Big[\sum_{n=0}^{\infty} \big(1+ \big|Z^{(M)}_{\tau_n}\big|^p\big) \mathds{1}_{(\tau_n, \tau_{n+1}] }(u) \Big].
    \end{aligned}
\end{equation}
This proves the second statement.
\end{proof}

\begin{proof}[Proof of Lemma \ref{FiniteMom}]
Observe that
\begin{equation}
    \begin{aligned}\label{FM3}
    &\E\Big[ \sum_{n=0}^{N_T + M} \big(1+ \big|Z^{(M)}_{\tau_{n}}\big|^p\big) \Big] 
    = \sum_{k=0}^{\infty}\sum_{n=0}^{k+M} \Big( \P(N_T=k) + \E\Big[  \big|Z^{(M)}_{\tau_{n}}\big|^p \mathds{1}_{\{N_T =k\}}\Big]\Big) \\
    &=\lambda T +  M + \sum_{k=0}^{\infty} \sum_{n=0}^{k+M} \E\Big[  \big|Z^{(M)}_{\tau_{n}}\big|^p \mathds{1}_{\{N_T =k\}}\Big].
    \end{aligned}
\end{equation}
Using \eqref{EqDefMS} we get that for all $n, k\in\N$ there exist constants $\widetilde c_1, \widetilde c_2 \in (0,\infty)$ such that
\begin{equation}
    \begin{aligned}\label{FM4}
    &\E\Big[\big|Z^{(M)}_{\tau_{n}}\big|^p\mathds{1}_{\{N_T = k\}}\Big]
    \leq 2^{p-1}\Big( \E\Big[\big|Z^{(M)}_{\tau_{n}-}\big|^p\mathds{1}_{\{N_T = k\}}\Big]  + \E\Big[\big|\rho(Z^{(M)}_{\tau_{n}-})\big|^p \big|N_{\tau_n} - N_{\tau_{n-1}}\big|^p\mathds{1}_{\{N_T = k\}}\Big] \Big)\\
    &\leq 2^{p-1}\Big( \E\Big[\big|Z^{(M)}_{\tau_{n}-}\big|^p\mathds{1}_{\{N_T =k\}}\Big] + 2^{p-1} c_{\widetilde\rho}^p \Big( \P(N_T =k) + \E\Big[\big|Z^{(M)}_{\tau_{n}-}\big|^p\mathds{1}_{\{N_T = k\}}\Big]\Big) \Big)\\
    &\leq \widetilde c_1\, \P( N_T =k) + \widetilde c_2\, \E\Big[\big|Z^{(M)}_{\tau_{n}-}\big|^p\mathds{1}_{\{N_T = k\}}\Big],
    \end{aligned}
\end{equation}
where we used that $N_{\tau_n}- N_{\tau_{n-1}} \in\{0,1\}$. Using \eqref{EqDefMSLL} we obtain for all $k,n \in\N$ that
\begin{equation}
    \begin{aligned}\label{FM5}
    &\E\Big[\big|Z^{(M)}_{\tau_{n}-}\big|^p\mathds{1}_{\{N_T = k\}} \Big]\\
    &\leq 4^{p-1}\Big(  \E\Big[\big|Z^{(M)}_{\tau_{n-1}}\big|^p\mathds{1}_{\{N_T = k\}} \Big] + \E\Big[\big|\widetilde\mu(Z^{(M)}_{\tau_{n-1}}) (\tau_n- \tau_{n-1})\big|^p\mathds{1}_{\{N_T =k\}} \Big] \\
    &\quad\quad\quad +\E\Big[\big|\widetilde\sigma(Z^{(M)}_{\tau_{n-1}}) (W_{\tau_n} - W_{\tau_{n-1}})\big|^p\mathds{1}_{\{N_T = k\}} \Big]\\ 
    &\quad\quad\quad +\E\Big[\Big| \frac{1}{2} \widetilde\sigma(Z^{(M)}_{\tau_{n-1}}) d_{\widetilde\sigma} (Z^{(M)}_{\tau_{n-1}}) \big((W_{\tau_n} - W_{\tau_{n-1}})^2 - (\tau_n- \tau_{n-1})\big)\Big|^p\mathds{1}_{\{N_T = k\}} \Big]\Big)\\
    &\leq 4^{p-1}\Big(  \E\Big[\big|Z^{(M)}_{\tau_{n-1}}\big|^p\mathds{1}_{\{N_T =k\}} \Big] + 2^{p-1} c_{\widetilde\mu}^p\,  \E\Big[\big( 1+ \big|Z^{(M)}_{\tau_{n-1}}\big|^p\big) \big|\tau_n- \tau_{n-1}\big|^p\mathds{1}_{\{N_T =k\}}\Big] \\
    &\quad\quad + 2^{p-1} c_{\widetilde\sigma}^p \E\Big[ \big( 1+ \big|Z^{(M)}_{\tau_{n-1}}\big|^p\big)
    \big|W_{\tau_n} - W_{\tau_{n-1}}\big|^p\mathds{1}_{\{N_T =k\}} \Big]\\ 
    &\quad\quad + 2^{p-2} L_{\widetilde\sigma}^p c_{\widetilde\sigma}^p \, \E\Big[ \big( 1+ \big|Z^{(M)}_{\tau_{n-1}}\big|^p\big) |W_{\tau_n} - W_{\tau_{n-1}}|^{2p} \mathds{1}_{\{N_T =k\}}\Big]\\
    &\quad\quad + 2^{p-2} L_{\widetilde\sigma}^p c_{\widetilde\sigma}^p\, \E\Big[ \big( 1+ \big|Z^{(M)}_{\tau_{n-1}}\big|^p\big) |\tau_n- \tau_{n-1}|^p \mathds{1}_{\{N_T =k\}}\Big]\Big).
    \end{aligned}
\end{equation}
Lemma \ref{PropDiscGrid} and Lemma \ref{MEW} yield for all $q\in\N$,
\begin{equation}
    \begin{aligned}\label{FM6}
    &\E\Big[ \big( 1+ \big|Z^{(M)}_{\tau_{n-1}}\big|^p\big) \big|W_{\tau_n} - W_{\tau_{n-1}}\big|^q \mathds{1}_{\{N_T =k\}}\Big]\\ 
    &\leq \E\Big[ \big( 1+ \big|Z^{(M)}_{\tau_{n-1}}\big|^p\big) \sup_{t\in[0,\delta]} \big|W_{t+\tau_{n-1}} - W_{\tau_{n-1}}\big|^q \mathds{1}_{\{N_T =k\}}\Big]\\
    &= \E\Big[ \big( 1+ \big|Z^{(M)}_{\tau_{n-1}}\big|^p\big) \mathds{1}_{\{N_T =k\}} \E\Big[   \sup_{t\in[0,\delta]} \big|W_{t+\tau_{n-1}} - W_{\tau_{n-1}}\big|^q \Big| \widetilde \F_{\tau_{n-1}}\Big] \Big]\\
    &= \E\Big[ \big( 1+ \big|Z^{(M)}_{\tau_{n-1}}\big|^p\big) \mathds{1}_{\{N_T =k\}} \E\Big[   \sup_{t\in[0,\delta]} \big|W_{t+\tau_{n-1}} - W_{\tau_{n-1}}\big|^q\Big] \Big]\\
    &\leq c_{W_q} \delta^\frac{q}{2} \Big(\P(N_T = k) +  \E\Big[ \big|Z^{(M)}_{\tau_{n-1}}\big|^p  \mathds{1}_{\{N_T = k \}}  \Big]\Big).
    \end{aligned}
\end{equation}
Plugging \eqref{FM6} into \eqref{FM5} we get that there exist $\widetilde c_3, \widetilde c_4 \in (0,\infty)$ such that
\begin{equation}
    \begin{aligned}\label{FM6a}
    &\E\Big[\big|Z^{(M)}_{\tau_{n}-}\big|^p\mathds{1}_{\{N_T =k \}} \Big]\\
    &\leq 4^{p-1}\Big( \E\Big[\big|Z^{(M)}_{\tau_{n-1}}\big|^p \mathds{1}_{\{N_T =k\}} \Big]
    + 2^{p-1} c_{\widetilde\mu}^p \delta^{p} \Big( \P(N_T =k)+ \E\Big[\big|Z^{(M)}_{\tau_{n-1}}\big|^p\mathds{1}_{\{N_T=k\}}\Big]\Big) \\
    &\quad\quad + 2^{p-1} c_{\widetilde\sigma}^p  c_{W_p} \delta^{\frac{p}{2}} \Big( \P( N_T=k)+ \E\Big[\big|Z^{(M)}_{\tau_{n-1}}\big|^p\mathds{1}_{\{N_T =k\}}\Big]\Big)\\
    &\quad\quad + 2^{p-2} L_{\widetilde\sigma}^p c_{\widetilde\sigma}^p  c_{W_{2p}} \delta^{p} \Big( \P(N_T =k)+ \E\Big[\big|Z^{(M)}_{\tau_{n-1}}\big|^p\mathds{1}_{\{N_T =k\}}\Big] \Big)\\
    &\quad\quad + 2^{p-2} L_{\widetilde\sigma}^p c_{\widetilde\sigma}^p  \delta^{p} \Big( \P(N_T =k)+ \E\Big[\big|Z^{(M)}_{\tau_{n-1}}\big|^p\mathds{1}_{\{N_T =k\}}\Big]\Big) \Big)\\
    &\leq \widetilde c_3 \, \P(N_T =k) + \widetilde c_4 \, \E\Big[\big|Z^{(M)}_{\tau_{n-1}}\big|^p\mathds{1}_{\{N_T =k\}}\Big].
    \end{aligned}
\end{equation}
Combining this with \eqref{FM4} we get that there exist constants $\widetilde c_5,\widetilde c_6,\widetilde c_7,\widetilde c_8 \in(1,\infty)$ such that
\begin{equation}
    \begin{aligned}\label{FM7}
    &\E\Big[\big|Z^{(M)}_{\tau_{n}}\big|^p\mathds{1}_{\{N_T =k\}}\Big]
    \leq \widetilde c_5\, \P(N_T =k)+ \widetilde c_6\, \E\Big[\big|Z^{(M)}_{\tau_{n-1}}\big|^p\mathds{1}_{\{N_T =k\}}\Big]
    \end{aligned}
\end{equation}
and 
\begin{equation}
    \begin{aligned}\label{FM8}
    &\E\Big[\big|Z^{(M)}_{\tau_{n}-}\big|^p\mathds{1}_{\{N_T =k\}}\Big]
    \leq \widetilde c_7\, \P(N_T =k)+ \widetilde c_8\, \E\Big[\big|Z^{(M)}_{\tau_{n-1}-}\big|^p\mathds{1}_{\{N_T =k\}}\Big].
    \end{aligned}
\end{equation}
Using \eqref{FM7} and the fact that $\widetilde\xi$ is deterministic, we recursively  obtain
\begin{equation}
    \begin{aligned}\label{FM9}
    &\E\Big[\big|Z^{(M)}_{\tau_{n}}\big|^p\mathds{1}_{\{N_T =k\}}\Big]
    \leq \sum_{j=0}^{n-1} \widetilde c_5 \widetilde c_6^{\,j}\, \P(N_T =k) + \widetilde c_6^{\,n}\, \E\Big[ |\widetilde \xi| \mathds{1}_{\{N_T =k\}} \Big]\\
    &\leq \widetilde c_5 (1+|\widetilde \xi|) \,\P(N_T =k) \sum_{j=0}^{n}  \widetilde c_6^{\,j} 
    = \widetilde c_5 (1+|\widetilde \xi|)  \frac{1- \widetilde c_6^{\,n+1}}{1- \widetilde c_6}\P(N_T =k).
    \end{aligned}
\end{equation}
Analogously using \eqref{FM8} recursively, then \eqref{FM6a}, and defining $\widetilde c_9 = \widetilde c_7 \vee (\widetilde c_3 + \widetilde c_4 |\widetilde\xi|)$ we get 
\begin{equation}
    \begin{aligned}\label{FM9a}
    &\E\Big[\big|Z^{(M)}_{\tau_{n}-}\big|^p\mathds{1}_{\{N_T =k\}}\Big]
    \leq \sum_{j=0}^{n-2} \widetilde c_7 \widetilde c_8^{\,j}\, \P( N_T =k) + \widetilde c_8^{\,n-1}\, \E\Big[\big|Z^{(M)}_{\tau_{1}-}\big|^p\mathds{1}_{\{N_T =k\}}\Big]\\
    &\leq \sum_{j=0}^{n-2} \widetilde c_7 \widetilde c_8^{\,j}\, \P(N_T =k) + \widetilde c_8^{\,n-1} \Big( \widetilde c_3\, \P(N_T = k) + \widetilde c_4\, \E\Big[\big|\widetilde\xi \big|^p\mathds{1}_{\{N_T =k\}}\Big]\Big)\\
    &\leq \big( \widetilde c_7 \vee \big(\widetilde c_3 + \widetilde c_4 \big|\widetilde\xi \big|^p\big)\big)\, \P(N_T =k) \sum_{j=0}^{n-1} 
    \widetilde c_8^{\,j}
    = \widetilde c_9  \frac{1- \widetilde c_8^{\,n}}{1- \widetilde c_8}\P(N_T =k).
    \end{aligned}
\end{equation}
Recalling $\widetilde c_6 >1$ and plugging \eqref{FM9} into \eqref{FM3} we obtain
\begin{equation}
    \begin{aligned}\label{FM10}
    &\E\Big[ \sum_{n=0}^{N_T + M} \big(1+ \big|Z^{(M)}_{\tau_{n}}\big|^p\big) \Big] 
    \leq\lambda T +  M + \frac{\widetilde c_5 (1+|\widetilde \xi|)}{\widetilde c_6 -1 } \sum_{k=0}^{\infty} \sum_{n=0}^{k+M} \big(  \widetilde c_6^{\,n+1}-1\big) \P(N_T = k) \\
    &=\lambda T +  M + \frac{\widetilde c_5 (1+|\widetilde \xi|)}{\widetilde c_6 -1 }\Big( \sum_{k=0}^{\infty} \P(N_T = k)\sum_{n=0}^{k+M}   \widetilde c_6^{\,n+1}
    -\sum_{k=0}^{\infty} \P(N_T= k)(k+M)  \Big)\\
    &\leq\lambda T +  M + \frac{\widetilde c_5 (1+|\widetilde \xi|)}{\widetilde c_6 -1 } \sum_{k=0}^{\infty} \P(N_T = k) \widetilde c_6 \frac{1- \widetilde c_6^{\,k+M+1}}{1- \widetilde c_6} \\
    &\leq\lambda T +  M +  \widetilde c_6^{\,M+2}\frac{\widetilde c_5 (1+|\widetilde \xi|)}{(\widetilde c_6 -1)^2 } \sum_{k=0}^{\infty} \P(N_T = k)  \widetilde c_6^{\,k}\\
    &=\lambda T +  M +  \widetilde c_6^{\,M+2}\frac{\widetilde c_5 (1+|\widetilde \xi|)}{(\widetilde c_6 -1)^2 } \exp(\lambda T(\widetilde c_6 -1)).
    \end{aligned}
\end{equation}
Similar steps as in \eqref{FM3} together with \eqref{FM9a}, and the fact that $\widetilde c_8>1$ ensure that there exists a constant $\widetilde c_{10} \in(0,\infty)$ such that
\begin{equation}
    \begin{aligned}\label{FM12}
    &\E\Big[ N_T^{p-1} \sum_{n=0}^{N_T + M} \big(1+ \big|Z^{(M)}_{\tau_{n}-}\big|^p\big) \Big] 
    = \sum_{k=0}^\infty \E\Big[ k^{p-1} \sum_{n=0}^{k + M} \big(1+ \big|Z^{(M)}_{\tau_{n}-}\big|^p\big) \mathds{1}_{\{N_T = k\}} \Big] \\
    &= \sum_{k=0}^\infty\sum_{n=0}^{k + M} k^{p-1} \P(N_T = k)
    + \sum_{k=0}^\infty\sum_{n=0}^{k + M} k^{p-1} \E\Big[ \big|Z^{(M)}_{\tau_{n}-}\big|^p \mathds{1}_{\{N_T = k\}} \Big] \\
    &\leq \sum_{k=0}^\infty k^{p} \P(N_T = k)
    + M \sum_{k=0}^\infty k^{p-1} \P(N_T = k)  
    + \sum_{k=0}^\infty\sum_{n=0}^{k + M} k^{p-1} \widetilde c_9  \frac{1- \widetilde c_8^{\,n}}{1- \widetilde c_8}\P(N_T =k) \\
    &\leq \E[N_T^p ] + M\, \E[N_T^{p-1}] 
    + \frac{\widetilde c_9}{\widetilde c_8 -1} \sum_{k=0}^\infty k^{p-1}  \P(N_T =k) \sum_{n=0}^{k + M} \widetilde c_8^{\,n}\\
    &= \E[N_T^p ] + M\, \E[N_T^{p-1}] + \frac{\widetilde c_9}{\widetilde c_8 -1} \sum_{k=0}^\infty k^{p-1}  \P(N_T =k) \frac{1- \widetilde c_8^{\,k+M+1}}{1-\widetilde c_8}\\
    &\leq \E[N_T^p ] + M\, \E[N_T^{p-1}] + \frac{\widetilde c_9 \widetilde c_8^{\,M+1}}{(\widetilde c_8 -1)^2} \sum_{k=0}^\infty k^{p-1}  \P(N_T =k) \widetilde c_8^{\,k}\\
    &\leq \E[N_T^p ] + M\, \E[N_T^{p-1}] 
    + \frac{\widetilde c_{10} \widetilde c_9 \widetilde c_8^{\,M+1}}{(\widetilde c_8 -1)^2} \exp(\lambda T(\widetilde c_8-1)).
    \end{aligned}
\end{equation}
Here we set $\widetilde c_{10}$ as the $(p-1)$-th moment of a Poisson distributed random variable with parameter $\widetilde c_8 \lambda T$.
Choosing $c_M$ as the sum of the constants derived in \eqref{FM10} and \eqref{FM12} finishes the proof.
\end{proof}

\begin{proof}[Proof of Lemma \ref{ME}]
By \eqref{SumAppr} we have for all $s\in[0,T]$,
\begin{equation}
    \begin{aligned}\label{ME1}
    &\E\Big[\sup_{t\in[0,s]} |Z^{(M)}_{t}|^p\Big]  \\
    & \leq 4^{p-1} \Big(  |\widetilde\xi|^p
    + \E\Big[\sup_{t\in[0,s]}\Big|\int_0^t \sum_{n=0}^{\infty} \widetilde \mu\big(Z^{(M)}_{\tau_{n}}\big)\mathds{1}_{(\tau_n,\tau_{n+1}]}(u) \diff u\Big|^p\Big] \\
    &\quad\quad\quad\quad+ \E\Big[\sup_{t\in[0,s]}\Big|\int_0^t \sum_{n=0}^{\infty} \Big(\widetilde \sigma\big(Z^{(M)}_{\tau_{n}}\big) + \widetilde\sigma \big(Z^{(M)}_{\tau_n}\big) d_{\widetilde\sigma} \big(Z^{(M)}_{\tau_n}\big) (W_{u}-W_{\tau_n}) \Big)\mathds{1}_{(\tau_n,\tau_{n+1}]}(u) \diff W_u\Big|^p\Big] \\
    &\quad\quad\quad\quad+ \E\Big[\sup_{t\in[0,s]}\Big|\int_0^t \widetilde \rho\big(Z^{(M)}_{u-}\big)\diff N_u\Big|^p\Big] \Big).
\end{aligned}
\end{equation}
In the following we estimate each summand of \eqref{ME1} separately. For the first one we apply Jensen's inequality, use the fact that $u$ lies in at most one interval $(\tau_n, \tau_{n+1}]$, and apply Lemma \ref{FiniteMom} to get
\begin{equation}
    \begin{aligned}\label{ME2}
    &\E\Big[\sup_{t\in[0,s]}\Big|\int_0^t \sum_{n=0}^{\infty} \widetilde \mu\big(Z^{(M)}_{\tau_{n}}\big)\mathds{1}_{(\tau_n,\tau_{n+1}]}(u) \diff u\Big|^p\Big] 
    = \E\Big[\sup_{t\in[0,s]}\Big|\int_0^t \sum_{n=0}^{N_T + M} \widetilde \mu\big(Z^{(M)}_{\tau_{n}}\big)\mathds{1}_{(\tau_n,\tau_{n+1}]}(u) \diff u\Big|^p\Big] \\
    &\leq \E\Big[\sup_{t\in[0,s]} T^{p-1} \int_0^t  \sum_{n=0}^{N_T + M} \Big|\widetilde \mu\big(Z^{(M)}_{\tau_{n}}\big)\Big|^p  \mathds{1}_{(\tau_n,\tau_{n+1}]}(u) \diff u\Big] \\
    &\leq 2^{p-1}\, T^{p-1}\, c_{\widetilde\mu}^p\, \E\Big[\int_0^T  \sum_{n=0}^{N_T + M} \big(1 + |Z^{(M)}_{\tau_{n}}|^p\big)\mathds{1}_{(\tau_n,\tau_{n+1}]}(u) \diff u\Big] \\
    &= 2^{p-1}\, T^{p-1}\, c_{\widetilde\mu}^p\,  \E\Big[ \sum_{n=0}^{N_T+ M} \big(1 + |Z^{(M)}_{\tau_{n}}|^p\big)(\tau_{n+1} -\tau_n)\Big] \\
    &\leq  2^{p-1}\, T^{p-1}\, c_{\widetilde\mu}^p\,\delta \, c_M <\infty.
\end{aligned}
\end{equation}
Next we consider the second summand of \eqref{ME1} and apply Lemma \ref{BDGaKunita} to show that there exists $\hat c\in(0,\infty)$ so that
\begin{equation}
    \begin{aligned}\label{ME13}
    &\E\Big[\sup_{t\in[0,s]}\Big|\int_0^t \sum_{n=0}^{\infty} \Big(\widetilde \sigma\big(Z^{(M)}_{\tau_{n}}\big) + \widetilde\sigma \big(Z^{(M)}_{\tau_n}\big) d_{\widetilde\sigma} \big(Z^{(M)}_{\tau_n}\big) (W_{u}-W_{\tau_n}) \Big)\mathds{1}_{(\tau_n,\tau_{n+1}]}(u) \diff W_u\Big|^p\Big]\\
    &\leq \hat c\, \E\Big[\int_0^s \sum_{n=0}^{\infty} \big|\widetilde \sigma\big(Z^{(M)}_{\tau_{n}}\big) + \widetilde\sigma \big(Z^{(M)}_{\tau_n}\big) d_{\widetilde\sigma} \big(Z^{(M)}_{\tau_n}\big) (W_{u}-W_{\tau_n}) \big|^p\mathds{1}_{(\tau_n,\tau_{n+1}]}(u) \diff u \Big]\\
    &\leq 2^{p-1} \hat c\, \Big( \E\Big[\int_0^s \sum_{n=0}^{\infty} \big|\widetilde \sigma\big(Z^{(M)}_{\tau_{n}}\big) \Big|^p \mathds{1}_{(\tau_n,\tau_{n+1}]}(u) \diff u \big] \\
    &\quad\quad\quad\quad + \E\Big[\int_0^s \sum_{n=0}^{\infty} \big|\widetilde\sigma \big(Z^{(M)}_{\tau_n}\big) d_{\widetilde\sigma} \big(Z^{(M)}_{\tau_n}\big) (W_{u}-W_{\tau_n}) \big|^p \mathds{1}_{(\tau_n,\tau_{n+1}]}(u) \diff u \Big]\Big).
\end{aligned}
\end{equation}
Estimating the first expectation of \eqref{ME13} using Lemma \ref{FiniteMom} we obtain that there exists a constant $c_M\in(0,\infty)$ such that
\begin{equation}
    \begin{aligned}\label{ME14}
    &\E\Big[\int_0^s \sum_{n=0}^{\infty} \Big|\widetilde \sigma\big(Z^{(M)}_{\tau_{n}}\big) \Big|^p \mathds{1}_{(\tau_n,\tau_{n+1}]}(u) \diff u \Big] 
    \leq 2^{p-1} c_{\widetilde \sigma}^p\, \E\Big[\int_0^T \sum_{n=0}^{N_T +M } \big(1+ \big|Z^{(M)}_{\tau_{n}}\big|^p\big) \mathds{1}_{(\tau_n,\tau_{n+1}]}(u) \diff u \Big] \\
    &= 2^{p-1} c_{\widetilde \sigma}^p\, \delta\, \E\Big[ \sum_{n=0}^{N_T +M } \big(1+ \big|Z^{(M)}_{\tau_{n}}\big|^p \big) \diff u \Big] 
    \leq  2^{p-1} c_{\widetilde \sigma}^p\, \delta\, c_M.
\end{aligned}
\end{equation}
Next we estimate the second expectation of \eqref{ME13} using Lemma \ref{PropDiscGrid}, Lemma \ref{LMEW}, \eqref{MEW2}, and Lemma \ref{FiniteMom}. We get that there exist $c_M, c_{W_p}\in(0,\infty)$ such that
\begin{equation}
    \begin{aligned}\label{ME15}
    &\E\Big[\int_0^s \sum_{n=0}^{\infty} \big|\widetilde\sigma \big(Z^{(M)}_{\tau_n}\big) d_{\widetilde\sigma} \big(Z^{(M)}_{\tau_n}\big) (W_{u}-W_{\tau_n}) \big|^p \mathds{1}_{(\tau_n,\tau_{n+1}]}(u) \diff u \Big]\\
    &\leq 2^{p-1} c_{\widetilde\sigma}^p L_{\widetilde\sigma}^p \delta\, \E\Big[\sum_{n=0}^{N_T +M} \big(1+ \big|Z^{(M)}_{\tau_n}\big|^p\big)  \sup_{v\in[0,\delta]} |W_{v+\tau_n}-W_{\tau_n}|^p \Big]\\
    &= 2^{p-1} c_{\widetilde\sigma}^p L_{\widetilde\sigma}^p \delta \sum_{k=0}^{\infty} \sum_{n=0}^{k+M}  \E\Big[ \E\Big[ \big(1+ \big|Z^{(M)}_{\tau_n}\big|^p\big) \sup_{v\in[0,\delta]} |W_{v+\tau_n}-W_{\tau_n}|^p \mathds{1}_{\{N_T = k\}}\Big| \widetilde \F_{\tau_n} \Big] \Big]\\
    &= 2^{p-1} c_{\widetilde\sigma}^p L_{\widetilde\sigma}^p \delta \sum_{k=0}^{\infty} \sum_{n=0}^{k+M} \E\Big[ \big(1+ \big|Z^{(M)}_{\tau_n}\big|^p\big) \mathds{1}_{\{N_T = k\}}  \E\Big[  \sup_{v\in[0,\delta]} |W_{v+\tau_n}-W_{\tau_n}|^p \Big] \Big]\\
    &\leq 2^{p-1} c_{\widetilde\sigma}^p L_{\widetilde\sigma}^p  c_{W_p} \delta^{\frac{p}{2}+1} \E\Big[ \sum_{n=0}^{N_T+M}  \big(1+ \big|Z^{(M)}_{\tau_n}\big|^p\big)\Big]
    \leq 2^{p-1} c_{\widetilde\sigma}^p L_{\widetilde\sigma}^p  c_{W_p} \delta^{\frac{p}{2}+1} c_M. \\
\end{aligned}
\end{equation}
Plugging \eqref{ME14} and \eqref{ME15} into \eqref{ME13} we obtain
\begin{equation}
    \begin{aligned}\label{ME16}
    &\E\Big[\sup_{t\in[0,s]}\Big|\int_0^t \sum_{n=0}^{\infty} \Big(\widetilde \sigma\big(Z^{(M)}_{\tau_{n}}\big) + \widetilde\sigma \big(Z^{(M)}_{\tau_n}\big) d_{\widetilde\sigma} \big(Z^{(M)}_{\tau_n}\big) (W_{u}-W_{\tau_n}) \Big)\mathds{1}_{(\tau_n,\tau_{n+1}]}(u) \diff W_u\Big|^p\Big]\\
    &\leq 2^{p-1} \hat c \big(2^{p-1} c_{\widetilde \sigma}^p\, \delta\, c_M + 2^{p-1} c_{\widetilde\sigma}^p L_{\widetilde\sigma}^p  c_{W_p} \delta^{\frac{p}{2}+1} c_M \big)
    < \infty.
\end{aligned}
\end{equation}
Next we consider the last summand of \eqref{ME1}. We recall that by $\nu_i$ we denote the $i$-th jump time of $N$ and use Lemma \ref{FiniteMom}. This ensures that there exists $c_M\in(0,\infty)$ such that 
\begin{equation}
    \begin{aligned}\label{ME17}
    &\E\Big[\sup_{t\in[0,s]}\Big|\int_0^t \widetilde \rho\big(Z^{(M)}_{u-}\big)\diff N_u\Big|^p\Big]
    = \E\Big[\sup_{t\in[0,s]}\Big| \sum_{i = 1}^{N_t} \widetilde \rho\big(Z^{(M)}_{\nu_i-}\big) \Big|^p\Big]
    \leq \E\Big[\Big( \sum_{i = 1}^{N_s} \big|\widetilde \rho\big(Z^{(M)}_{\nu_i-}\big)\big| \Big)^p\Big]\\
    &\leq 2^{p-1} c_{\widetilde\rho}^p\, \E\Big[N_s^{p-1} \sum_{i = 1}^{N_s} \big( 1+ \big|Z^{(M)}_{\nu_i-}\big|^p\big)\Big]
    \leq 2^{p-1} c_{\widetilde\rho}^p\, \E\Big[N_T^{p-1} \sum_{n=0}^{N_T +M} \big( 1+ \big|Z^{(M)}_{\tau_n-}\big|^p\big)\Big]
    \leq 2^{p-1} c_{\widetilde\rho}^p c_M <\infty.
\end{aligned}
\end{equation}
Combining \eqref{ME1}, \eqref{ME2}, \eqref{ME16}, and \eqref{ME17} we obtain for all $s\in[0,T]$, all $M\in\N$, and all $p\in\N$,
\begin{equation}
    \begin{aligned}\label{ME18}
    &\E\Big[\sup_{t\in[0,s]} |Z^{(M)}_{t}|^p\Big]  <\infty.
\end{aligned}
\end{equation}
Next we estimate \eqref{ME1} again with the goal to to apply the Gronwall inequality.
We once more estimate all summands separately and start with the first one. Similar as in \eqref{ME2} we obtain
\begin{equation}
    \begin{aligned}\label{ME19}
    &\E\Big[\sup_{t\in[0,s]}\Big|\int_0^t \sum_{n=0}^{\infty} \widetilde \mu\big(Z^{(M)}_{\tau_{n}}\big)\mathds{1}_{(\tau_n,\tau_{n+1}]}(u) \diff u\Big|^p\Big] 
    \leq \E\Big[T^{p-1} \int_0^s  \sum_{n=0}^{\infty} \big|\widetilde \mu\big(Z^{(M)}_{\tau_{n}}\big)\big|^p  \mathds{1}_{(\tau_n,\tau_{n+1}]}(u) \diff u\Big] \\
    &\leq  2^{p-1}\, T^{p-1}\, c_{\widetilde\mu}^p\, \E\Big[\int_0^s  \sum_{n=0}^{\infty} \big(1 + \sup_{v\in[0,u]}  \big|Z^{(M)}_{v}\big|^p\big)\mathds{1}_{(\tau_n,\tau_{n+1}]}(u) \diff u\Big] \\
    &\leq  2^{p-1}\, T^{p}\, c_{\widetilde\mu}^p +  2^{p-1}\,T^{p-1}\, c_{\widetilde\mu}^p\int_0^s\E\Big[ \sup_{v\in[0,u]}  \big|Z^{(M)}_{v}\big|^p\Big]\diff u.
\end{aligned}
\end{equation}
Next we again estimate the second summand of \eqref{ME1} starting from \eqref{ME13}. Estimating the first expectation of \eqref{ME13} we obtain
\begin{equation}
    \begin{aligned}\label{ME21}
    &\E\Big[\int_0^s \sum_{n=0}^{\infty} \big|\widetilde \sigma\big(Z^{(M)}_{\tau_{n}}\big) \big|^p \mathds{1}_{(\tau_n,\tau_{n+1}]}(u) \diff u \Big] \\
    &\leq 2^{p-1} c_{\widetilde \sigma}^p\, \E\Big[\int_0^s \sum_{n=0}^{N_T + M} \big(1+ \sup_{v\in[0,u]} \big|Z^{(M)}_{v}\big|^p \big) \mathds{1}_{(\tau_n,\tau_{n+1}]}(u) \diff u \Big] \\
    &\leq 2^{p-1} c_{\widetilde \sigma}^p \,T + 2^{p-1} c_{\widetilde \sigma}^p \int_0^s\E\Big[ \sup_{v\in[0,u]} \big|Z^{(M)}_{v}\big|^p \Big]\diff u.
\end{aligned}
\end{equation}
Next we estimate the second expectation of \eqref{ME13} using Lemma \ref{MW} and obtain that there exist $c_{W_p}\in(0,\infty)$ such that
\begin{equation}
    \begin{aligned}\label{ME22}
    &\E\Big[\int_0^s \sum_{n=0}^{\infty} \big|\widetilde\sigma \big(Z^{(M)}_{\tau_n}\big) d_{\widetilde\sigma} \big(Z^{(M)}_{\tau_n}\big) (W_{u}-W_{\tau_n}) \big|^p \mathds{1}_{(\tau_n,\tau_{n+1}]}(u) \diff u \Big]\\
    &\leq 2^{p-1} c_{\widetilde\sigma}^p\, L_{\widetilde\sigma}^p \int_0^s \E\Big[\sum_{n=0}^{\infty} \big(1+ \big|Z^{(M)}_{\tau_n}\big|^p\big) |W_{u}-W_{\tau_n}|^p \mathds{1}_{(\tau_n,\tau_{n+1}]}(u) \Big]\diff u\\
    &\leq 2^{p-1} c_{\widetilde\sigma}^p\, L_{\widetilde\sigma}^p c_{W_p} T \delta^\frac{p}{2}
    +2^{p-1} c_{\widetilde\sigma}^p\, L_{\widetilde\sigma}^p c_{W_p} \delta^\frac{p}{2} \int_0^s  \E\Big[\sup_{v\in[0,u]} \big|Z^{(M)}_{v}\big|^p \Big]\diff u.
    \end{aligned}
\end{equation}
Combining \eqref{ME13}, \eqref{ME21}, and \eqref{ME22} we obtain
\begin{equation}
    \begin{aligned}\label{ME24}
    &\E\Big[\sup_{t\in[0,s]}\Big|\int_0^t \sum_{n=0}^{\infty} \Big(\widetilde \sigma\big(Z^{(M)}_{\tau_{n}}\big) + \widetilde\sigma \big(Z^{(M)}_{\tau_n}\big) d_{\widetilde\sigma} \big(Z^{(M)}_{\tau_n}\big) (W_{u}-W_{\tau_n}) \Big)\mathds{1}_{(\tau_n,\tau_{n+1}]}(u) \diff W_u\Big|^p\Big]\\
    &\leq 2^{p-1} \hat c\,\Big( 2^{p-1} c_{\widetilde \sigma}^p\,T +2^{p-1} c_{\widetilde \sigma}^p \int_0^s\E\Big[ \sup_{v\in[0,u]} \big|Z^{(M)}_{v}\big|^p \Big]\diff u \\
    &\quad\quad\quad +2^{p-1} c_{\widetilde\sigma}^p\, L_{\widetilde\sigma}^p c_{W_p} T \delta^\frac{p}{2}
    +2^{p-1} c_{\widetilde\sigma}^p\, L_{\widetilde\sigma}^p c_{W_p} \delta^\frac{p}{2} \int_0^s  \E\Big[\sup_{v\in[0,u]} \big|Z^{(M)}_{v}\big|^p \Big]\diff u \Big).
\end{aligned}
\end{equation}
For the last summand of \eqref{ME1} Lemma \ref{BDGaKunita} yields that there exists $\hat c \in(0,\infty)$ such that
\begin{equation}
    \begin{aligned}\label{ME25}
    &\E\Big[\sup_{t\in[0,s]}\Big|\int_0^t \widetilde \rho\big(Z^{(M)}_{u-}\big)\diff N_u\Big|^p\Big] 
    \leq \hat c \int_0^s \E\Big[\big| \widetilde \rho\big(Z^{(M)}_{u-}\big)\big|^p\Big]\diff u\\
    &\leq 2^{p-1} \,\hat c\, c_{\widetilde\rho}^p\, T + 2^{p-1} \,\hat c\, c_{\widetilde\rho}^p \int_0^s \E\Big[ \sup_{v\in[0,u]}\big|Z^{(M)}_{v}\big|^p\Big]\diff u.
\end{aligned}
\end{equation}
Plugging \eqref{ME19}, \eqref{ME24}, and \eqref{ME25} into \eqref{ME1} we obtain that there exists a constants $\widetilde c_1 \in(0,\infty)$ such that
\begin{equation}
    \begin{aligned}\label{ME26}
    \E\Big[\sup_{t\in[0,s]} |Z^{(M)}_{t}|^p\Big] 
    \leq \widetilde c_1 ( 1+ |\widetilde\xi|^p) + \widetilde c_1 \int_0^s \E\Big[ \sup_{v\in[0,u]}\big|Z^{(M)}_{v}\big|^p\Big]\diff u. 
    \end{aligned}
\end{equation}
Because $[0,T] \ni s \mapsto \E\Big[ \sup_{v\in[0,s]}\big|Z^{(M)}_{v}\big|^p\Big]$ is finite due to \eqref{ME18} and Borel measurable since it is increasing, we conclude by Gronwall's inequality that there exists a constant $c_1\in(0,\infty)$ such that 
\begin{equation}
    \begin{aligned}\label{ME27}
    &\E\Big[\sup_{t\in[0,T]} |Z^{(M)}_{t}|^p\Big] \leq c_1 ( 1+ |\widetilde\xi|^p).
    \end{aligned}
\end{equation}
This proves the first inequality.

Nest observe that there exists a constant $\widetilde c_2\in(0,\infty)$ such that
\begin{equation}\label{ME28}
    \begin{aligned}
    &\E\Big[ \sum_{n=0}^{\infty} \big|Z^{(M)}_s- Z^{(M)}_{\tau_{n}}\big|^p\mathds{1}_{[\tau_n,\tau_n+1)}(s)\Big]\\
    &\leq \E\Big[ \sum_{n=0}^{\infty} \widetilde c_2 \Big(\big| \widetilde \mu ( Z^{(M)}_{\tau_{n}})(s-\tau_n)\big|^p + \big|\widetilde\sigma(Z^{(M)}_{\tau_{n}}) (W_s-W_{\tau_n})\big|^p\\
    &\quad\quad\quad\quad \quad\quad + \Big|\frac{1}{2}\widetilde\sigma(Z^{(M)}_{\tau_{n}})d_{\widetilde\sigma}(Z^{(M)}_{\tau_{n}}) \big((W_s-W_{\tau_n})^2 - (s-\tau_n)\big)\Big|^p\Big) \mathds{1}_{[\tau_n,\tau_n+1)}(s)\Big]\\
    &\leq \widetilde c_2 \Big( 2^{p-1} c_{\widetilde\mu}^p \delta^p\,  \E\Big[ \sum_{n=0}^{\infty} \big( 1+| Z^{(M)}_{\tau_{n}}|^p\big) \mathds{1}_{[\tau_n,\tau_n+1)}(s)\Big]\\
    &\quad\quad\quad+2^{p-1} c_{\widetilde\sigma}^p\, \E\Big[ \sum_{n=0}^{\infty} \big( 1+| Z^{(M)}_{\tau_{n}}|^p\big) |W_s-W_{\tau_n}|^p\mathds{1}_{[\tau_n,\tau_n+1)}(s)\Big]\\
    &\quad\quad\quad+\frac{1}{2} c_{\widetilde\sigma}^p L_{\widetilde\sigma}^p\, \E\Big[ \sum_{n=0}^{\infty} \big( 1+| Z^{(M)}_{\tau_{n}}|^p\big)  \Big|W_s-W_{\tau_n}|^{2p}\mathds{1}_{[\tau_n,\tau_n+1)}(s)\Big] \\
    &\quad\quad\quad+\frac{1}{2} c_{\widetilde\sigma}^p L_{\widetilde\sigma}^p\delta^p\, \E\Big[ \sum_{n=0}^{\infty}   \big( 1+| Z^{(M)}_{\tau_{n}}|^p\big) \mathds{1}_{[\tau_n,\tau_n+1)}(s)\Big] \Big).
    \end{aligned}
\end{equation}
Applying Lemma \ref{MEW} to the second and third summand we obtain that there exists a constant $\widetilde c_3\in(0,\infty)$ such that
\begin{equation}\label{ME29}
    \begin{aligned}
    &\E\Big[ \sum_{n=0}^{\infty} \big|Z^{(M)}_s- Z^{(M)}_{\tau_{n}}\big|^p\mathds{1}_{[\tau_n,\tau_n+1)}(s)\Big]\\
    &\leq \widetilde c_3\big(\delta^p+ \delta^\frac{p}{2} + \delta^p +\delta^p\big)  \E\Big[ \sum_{n=0}^{\infty} \big( 1+| Z^{(M)}_{\tau_{n}}|^p\big) \mathds{1}_{[\tau_n,\tau_n+1)}(s)\Big].
    \end{aligned}
\end{equation}
By the first claim it holds that
\begin{equation}\label{ME30}
    \begin{aligned}
    &\E\Big[ \sum_{n=0}^{\infty} \big( 1+| Z^{(M)}_{\tau_{n}}|^p\big) \mathds{1}_{[\tau_n,\tau_n+1)}(s)\Big]
    \leq 1+ \E\Big[ \sup_{v\in[0,T]} | Z^{(M)}_{v}|^p\big) \Big]
    \leq 1+ c_1 \big( 1+ |\widetilde\xi|^p\big).
    \end{aligned}
\end{equation}
Hence, we obtain that there exists a constant $c_2\in(0,\infty)$ such that 
\begin{equation}\label{ME33}
    \begin{aligned}
    &\E\Big[ \sum_{n=0}^{\infty} \big|Z^{(M)}_s- Z^{(M)}_{\tau_{n}}\big|^p\mathds{1}_{[\tau_n,\tau_n+1)}(s)\Big]
    \leq c_2 ( 1+ |\widetilde\xi|^p) \delta^{\frac{p}{2}},
    \end{aligned}
\end{equation}
which proves the second statement. 

For the third statement let $t\in[0,T-\delta]$. Then by Lemma \ref{BDGaKunita} there exists a constant $\hat c\in(0,\infty)$ such that
\begin{equation}\label{ME34}
    \begin{aligned}
    &\E\Big[ \sup_{s\in[t,t+\delta]} \big|Z^{(M)}_s- Z^{(M)}_{t}\big|^p\Big]\\
    &\leq 3^{p-1}\hat c \Big(\int_{t}^{t+\delta}\E\Big[  \Big|  \sum_{n=0}^{\infty} \widetilde\mu (Z^{(M)}_{\tau_n}) \mathds{1}_{[\tau_n,\tau_{n+1})}(u)\Big|^p\Big] \diff u\\
    &\quad\quad\quad\quad\quad+ \int_{t}^{t+\delta}\E\Big[\Big| \sum_{n=0}^{\infty} \Big( \widetilde\sigma (Z^{(M)}_{\tau_n}) + \widetilde\sigma (Z^{(M)}_{\tau_n}) d_{\widetilde\sigma}(Z^{(M)}_{\tau_n}) (W_u -W_{\tau_n})\Big) \mathds{1}_{[\tau_n,\tau_{n+1})}(u)\Big|^p\Big] \diff u\\
    &\quad\quad\quad\quad\quad +\int_{t}^{t+\delta} \E\Big[\big| \widetilde\rho (Z^{(M)}_{u-}) \big|^p \Big]\diff u\Big).
    \end{aligned}
\end{equation}
Next we consider each expectation separately. For the first one we obtain
\begin{equation}\label{ME35}
    \begin{aligned}
    &\E\Big[\Big|  \sum_{n=0}^{\infty} \widetilde\mu (Z^{(M)}_{\tau_n}) \mathds{1}_{[\tau_n,\tau_{n+1})}(u)\Big|^p\Big]
    \leq 2^{p-1} c_{\widetilde\mu}^p\, \E\Big[  \sum_{n=0}^{\infty} \big(1+ \big|Z^{(M)}_{\tau_n}\big|^p\big) \mathds{1}_{[\tau_n,\tau_{n+1})}(u)\Big]\\
    &\leq 2^{p-1} c_{\widetilde\mu}^p\, \Big( 1+  \E\big[ \sup_{t\in[0,T]}\big|Z^{(M)}_{t}\big|^p\big]\Big)
    \leq  2^{p-1} c_{\widetilde\mu}^p\, \big( 1+  c_1 (1 +|\widetilde \xi|^p) \big).
    \end{aligned}
\end{equation}
For the second one we use Lemma \ref{MW} to obtain that there exists $c_{W_p}\in (0,\infty)$ such that
\begin{equation}\label{ME36}
    \begin{aligned}
    &\E\Big[\Big| \sum_{n=0}^{\infty} \Big( \widetilde\sigma (Z^{(M)}_{\tau_n}) + \widetilde\sigma (Z^{(M)}_{\tau_n}) d_{\widetilde\sigma}(Z^{(M)}_{\tau_n}) (W_u -W_{\tau_n})\Big) \mathds{1}_{[\tau_n,\tau_{n+1})}(u)\Big|^p\Big]\\
    &\leq \E\Big[ \sum_{n=0}^{\infty} 2^{p-1} \Big( \big|\widetilde\sigma (Z^{(M)}_{\tau_n})\big|^p + \big|\widetilde\sigma (Z^{(M)}_{\tau_n}) d_{\widetilde\sigma}(Z^{(M)}_{\tau_n}) (W_u -W_{\tau_n})\big|^p\Big) \mathds{1}_{[\tau_n,\tau_{n+1})}(u)\Big]\\
    &\leq 4^{p-1} c_{\widetilde\sigma}^p\, \E\Big[ \sum_{n=0}^{\infty}  \big(1+\big|Z^{(M)}_{\tau_n}\big|^p\big)\mathds{1}_{[\tau_n,\tau_{n+1})}(u)\Big]\\
    &\quad + 4^{p-1} c_{\widetilde\sigma}^p L_{\widetilde\sigma}^p\, \E\Big[ \sum_{n=0}^{\infty} \big(1+\big|Z^{(M)}_{\tau_n}\big|^p\big) |W_u -W_{\tau_n}|^p\mathds{1}_{[\tau_n,\tau_{n+1})}(u)\Big]\\
    &\leq 4^{p-1} c_{\widetilde\sigma}^p\, \big( 1+ \E\big[\sup_{t\in[0,T]}\big|Z^{(M)}_{\tau_n}\big|^p\big]\big)
    + 4^{p-1} c_{\widetilde\sigma}^p L_{\widetilde\sigma}^p c_{W_p} \delta^\frac{p}{2} \, \E\Big[\sum_{n=0}^{\infty}  \big(1+\big|Z^{(M)}_{\tau_n}\big|^p\big)\mathds{1}_{[\tau_n,\tau_{n+1})}(u)\Big]\\
    &\leq \Big( 4^{p-1} c_{\widetilde\sigma}^p\, + 4^{p-1} c_{\widetilde\sigma}^p L_{\widetilde\sigma}^p c_{W_p} \delta^{\frac{p}{2}}\Big)\big( 1+ c_1(1+|\widetilde\xi|^p)\big).
    \end{aligned}
\end{equation}
For the third term we calculate
\begin{equation}\label{ME37}
    \begin{aligned}
    &\E\Big[\big| \widetilde\rho (Z^{(M)}_{u-}) \big|^p \Big] 
    \leq 2^{p-1} c_{\widetilde\rho}^p \, \big(1+\E\big[\sup_{t\in[0,T]}\big|Z^{(M)}_{t}\big|^p\big]\big)
    \leq 2^{p-1} c_{\widetilde\rho}^p \, \big(1+c_1(1+|\widetilde\xi|^p)\big).
    \end{aligned}
\end{equation}
Plugging \eqref{ME35}, \eqref{ME36}, and \eqref{ME37} into \eqref{ME34} we obtain that there exists a constant $\widetilde c_4\in(0,\infty)$ such that
\begin{equation}\label{ME38}
    \begin{aligned}
    &\E\Big[\sup_{s\in[t,t+\delta]} \big|Z^{(M)}_s- Z^{(M)}_{t}\big|^p\Big]
    \leq \widetilde c_4\, \int_{t}^{t+\delta} \big( 1+ c_1 (1 +|\widetilde \xi|^p) \big) \diff u \\
    &= \widetilde c_4\, \big( 1+  c_1 (1 +|\widetilde \xi|^p) \big) \delta \leq \widetilde c_4\,(1+  c_1 )(1 +|\widetilde \xi|^p)\, \delta.
    \end{aligned}
\end{equation}
Defining $c_3 :=  \widetilde c_4\,(1+  c_1 )$ finishes the proof.
\end{proof}

\setlength{\bibsep}{0pt plus 0.3ex}

\vspace{2em}
\centerline{\underline{\hspace*{16cm}}}

\noindent Pawe{\l} Przyby{\l}owicz  \\
Faculty of Applied Mathematics, AGH University of Science and Technology, Al.~Mickiewicza 30, 30-059 Krakow, Poland\\
pprzybyl@agh.edu.pl\\

\noindent Verena Schwarz \Letter \\
Department of Statistics, University of Klagenfurt, Universit\"atsstra\ss{}e 65-67, 9020 Klagenfurt, Austria\\
verena.schwarz@aau.at\\

\noindent Michaela Sz\"olgyenyi \\
Department of Statistics, University of Klagenfurt, Universit\"atsstra\ss{}e 65-67, 9020 Klagenfurt, Austria\\
michaela.szoelgyenyi@aau.at\\

\end{document}